%% file: EquidistributionTannakaArticle.tex
\documentclass[oneside, 16pt]{amsart}
\usepackage[utf8]{inputenc}

\usepackage{wrapfig}
\usepackage{PersCMDS}

\usepackage{amssymb}
\usepackage{mathtools}
\usepackage{ mathrsfs }
\usepackage{tikz-cd}

\usepackage[hidelinks]{hyperref}
\usepackage{enumitem}

\usepackage[style=alphabetic]{biblatex}

\usepackage{subfiles}

\newtheorem{theorem}{Theorem}[section]
\newtheorem*{theorem*}{Theorem}
\newtheorem{corollary}[theorem]{Corollary}
\newtheorem{proposition}[theorem]{Proposition}
\newtheorem{lemma}[theorem]{Lemma}
\theoremstyle{definition}
\newtheorem{definition}[theorem]{Definition}
\newtheorem{remark}[theorem]{Remark}
\newtheorem{example}[theorem]{Example}

\newcommand{\SC}{\mathscr{C}}
\newcommand{\Perv}[2]{\text{Perv}_{#1}({#2})}
\newcommand{\PervUnr}[2]{\text{Perv}_{#1\text{-unr}}({#2})}
\newcommand{\PervInt}[2]{\text{Perv}_{#1\text{-uint}}({#2})}
\newcommand{\cvar}[1]{\mathscr{C}(#1)}
\newcommand{\FMSH}[1]{\text{FM}_{\chi, \Delta, !}(#1)}
\newcommand{\FMST}[1]{\text{FM}_{\chi, \Delta, *}(#1)}
\newcommand{\FMQ}[1]{\text{FM}_{\chi, \Delta, ?}(#1)}
\newcommand{\FMSHunr}[1]{\text{FM}_{\chi, !}(#1)}
\newcommand{\FMSTunr}[1]{\text{FM}_{\chi, *}(#1)}
\newcommand{\FMQunr}[1]{\text{FM}_{\chi, ?}(#1)}
\newcommand{\dualUni}{\widehat{U}}
\newcommand{\quotationMark}[1]{``#1"}
\renewcommand{\derCat}[3]{\text{D}_{#1}^b(#2)}

\title{Equidistribution for Tannakian Monodromy Groups}
\author[]{Beat Zurbuchen}

\begin{document}
	\begin{abstract}
		We prove that a perverse sheaf on a connected commutatitve algebraic group over a finite field is generically unramified. This implies an equidistribution theorem for Tannakian monodromy groups in previously unavailable generality. We also prove a stratification theorem for exponential sums in families indexed by a scheme and the characters of a connected commutative algebraic group. Our method is based on Tannakian categories introduced by Gabber and Loeser. This method naturally yields fiber functors. We also prove vanishing theorems over a connected commutative algebraic group, classify the negligible sheaves, and prove relative weak propagation theorems for tori. 
	\end{abstract}
	\maketitle
	\section{Introduction}
	Equidistribution of exponential sums indexed by the characters of a connected commutative algebraic group $G$ over a finite field $k$ were first studied in this generality by \cite{KowalskiTannaka}. The formalism of loc. cit. attaches a geometric and an arithmetic Tannakian monodromy group to a perverse sheaf on $G$ to control the equidistribution of the associated Fourier coefficients. This formalism yields equidistribution for continuous trace test functions, but not for arbitrary continuous class functions. Equidistribution is expected to hold more generally, for example, Katz proved equidistribution for all continuous class functions for Tannakian monodromy groups attached to perverse sheaves on $\BG_m$ in \cite{KatzConvEqui}. 
	
	The fundamental issue in the formalism of \cite{KowalskiTannaka} is that the constructed Tannakian categories are not equipped with fiber functors. They introduce the definition of a \textit{generically unramified sheaf}, i.e. a perverse sheaf whose associated Tannakian category admits sufficiently many fiber functors. Under this assumption, they are able to prove the expected equidistribution theorems in  \cite[Thm.~4.15,~Thm.~4.11]{KowalskiTannaka}. The goal of the present work is to construct the required fiber functors for these equidistribution theorems to hold: 
	\begin{theorem*}
	Let $M \in \Perv{}{G}.$ Then $M$ is generically unramified. 
\end{theorem*}
	In fact, we prove that the ramified locus is contained in a thin proper subset. Conceptually, this says that a perverse sheaf on a group ramifies outside a thin closed subset. This is expected because ramification, in the classical setting, occurs along a closed subset. We write out one consequence of the results proved here for the equidistribution of class functions.
	\begin{theorem*} Let $M \in \Perv{}{G}$ be pure of weight zero. Suppose $\mathscr{X}\subseteq\widehat{G}$ is a complement of a thin proper subset such that $M$ is unramified at all $\chi \in\mathscr{X}.$ Assume the geometric and the arithmetic Tannakian monodromy group of $M$ agree. There exists a compact Lie group $K$ and conjugacy classes $\Theta_{M, \chi} \in K^\sharp$ such that all continuous class-functions $f\colon K^\sharp \rightarrow \BC$ satisfy
	\[
	\lim_{n \rightarrow \infty} \frac{1}{|\mathscr{X}(k_n)|}\sum_{\chi \in \mathscr{X}(k_n)} f(\theta_{M, \chi}) = \int_K f(g)dg. 
	\]
\end{theorem*} 
The conjugacy classes $\Theta_{M, \chi}$ in the above theorem are related to the perverse sheaf $M$ via the Tannakian formalism in the following way. For each $\chi \in \mathscr{X}$ and each object $M'$ in the Tannakian category generated by $M$ under convolution, we have the vector space $H^0_c(G, M'_\chi)\otimes_\iota\BC$ on which $K$ acts naturally. The conjugacy class $\Theta_{M,\chi}$ acts coincides with that of the Frobenius endomorphism. The Lefschetz trace formula often yields a concrete expression for the trace of such an operator. The terms with $\chi \in \widehat{G}$ and $\chi \notin \mathscr{X}$ can be bounded by the generalized Riemann hypothesis and the stratification theorem (see below). This allows one to bound these terms with a power saving. Thus Tannakian monodromy groups become a much more flexible tool for the study of families of exponential sums indexed by the characters of a group.
	
This equidistribution statement can be generalized in various directions.  We refer the interested reader to \cite[Ch.~4]{KowalskiTannaka}, especially \cite[Thm.~4.15,~Thm.~4.11]{KowalskiTannaka}. In this article, the main equidistribution theorems are Theorem \ref{THM_Equi1} and Theorem \ref{THM_Equi2}.
\subsection{Statements}
The construction of the fiber functors requires a variety of theorems which are of independent interest. The goal of this section is to summarize these theorems.  We fix notations: the letter $G$ denotes a connected commutative algebraic group  over a finite field $k$ and we simply call $G$  a \textit{group}. The letter $\overline{k}$ denotes a fixed algebraic closure of $k$ and $k_n \subseteq \overline{k}$ is the unique extension of $k$ of degree $n \geq 1$ in $\overline{k}$. The letter $X$ always denotes a scheme of finite type over $k$, $S$ a semiabelian variety, $T$ a torus, and $U$ a unipotent (commutative) group over $k$.
	\subsubsection{Stratification theorems}In \cite{KowalskiTannaka}, the authors apply their vanishing theorems to prove stratification theorems for exponential sums indexed by the characters of a group. We can also utilize the vanishing theorem (see below) to obtain stratification theorems for exponential sums.  We obtain two improvements over the stratification theorems in loc. cit. First, the ramification loci in the stratification theorem are closed. This is one of the crucial improvements over \cite{KowalskiTannaka} that allows for the construction of fiber functors. Secondly, the stratification theorems hold uniformly in a family indexed by a scheme $X$. To our knowledge, this is the first stratification theorem that holds uniformly in families for groups of the form $S\times U$. 
\begin{theorem*}
	Let $X$ be of pure dimension $d$. We assume $G = S\times U$. Let $K \in \derCat{c}{G\times X}{\Qbarl}$ be a complex of weight $\leq 0$ and with perverse amplitude $\leq d$. There exists:
	\begin{enumerate}
		\item A stratification by closed subsets \[\widehat{U}\times X = Z_{0} \supseteq Z_1 \supseteq  \ldots \supseteq Z_{n - 1} \supseteq Z_n = \emptyset \]
		such that for all $0 \leq i \leq n$ the subset $Z_i$ has codimension $\geq i$. 
		\item A stratification by thin subsets
		\[
		\cvar{S} = \Delta_{0} \supset \Delta_1 \supseteq \ldots \supseteq \Delta_{n - 1} \supseteq \Delta_n = \emptyset
		\]
		such that for all $0\leq j \leq n$ the subset $\Delta_j$ has codimension $\geq j$.
	\end{enumerate}
	Let $m \geq 1$ and $0\leq i, j\leq n - 1$.  For all $\chi \in \widehat{G}(k_m)$ and $x \in X(k_m)$ with $(\pi_U^*(\chi), x) \notin Z_{i + 1}$ and $\pi_S^*(\chi) \notin \Delta_{j + 1}$, we have
	\[
	\sum_{t \in G(k_m)}\chi(t)t_K(t, x)\ll_K |k_m|^{(i + j)/2}.
	\]
\end{theorem*}
For estimates on the size of proper thin subsets, the reader is referred to Proposition \ref{PROP_thinsubsetsestimate}. We remark that, in the terminology of \cite[Def.~1.23]{KowalskiTannaka}, a thin subset $\Delta \subseteq\cvar{S}$ is a finite union of tacs (because $S$ is a semiabelian variety). For arbitrary groups, which are not a product, we record an absolute stratification theorem in Theorem \ref{THM_stratsfourier2} 

\subsubsection{Vanishing theorems} In order to prove the existence of the fiber functors and the above stratification theorem, stronger vanishing theorems are required. In particular, constructing fiber functors requires us to prove that the vanishing loci are {closed.} Since vanishing also has to be proven for the unipotent part, we have to prove {relative vanishing theorems for semiabelian varieties.}  Theorem \ref{THM_VanSemiAb} is the main vanishing theorem of this article:
\begin{theorem*}
	Let $M \in \Perv{}{S\times X}$. There exists a dense open subset $U \subseteq X$ and a proper thin subset $\Delta \subset \cvar{S}$ such that:
	\begin{enumerate}
		\item The complexes $\pi_{X?}(M_\chi)$ are lisse over $U$ and their formation commutes with arbitrary base change $X' \rightarrow U$ for all $\chi \in \mathscr{C}(S)$.
		\item The forget supports map
		\[
		\pi_{X!}(M_\chi) \rightarrow \pi_{X*}(M_\chi)
		\]
		is an isomorphism for all $\chi \notin \Delta$.
		\item The complexes $\pi_{X?}(M_\chi)$ are perverse for all $\chi \notin\Delta.$
	\end{enumerate}
\end{theorem*}
We remark, that point (1) holds for all characters on the semiabelian variety $S$.  The terminology of a thin subset is borrowed from \cite{Kr_mer_2015}.

\subsubsection{Relative weak propagation} In order to construct the desired fiber functors, we have to prove a relative weak propagation theorem. The reason for this is that the relative propagation theorems amplify the vanishing via a spectral sequence argument. The idea is inspired by \cite{liumaximwangpropagation}, where propagation theorems are established in the absolute case for a semiabelian variety in a stronger form. To our knowledge, propagation theorems were only considered in the absolute case before and relative propagation theorems have not been known to hold. Our first propagation theorem is Theorem \ref{THM_Propagation1}.
\begin{theorem*}
	Let $M \in \Perv{}{T\times X}$. Suppose there is a closed subset $\Delta \subseteq\cvar{T}$ such that
	\[
	\pervCoh{0}{\pi_{X!}(M_\chi)} = 0.
	\]
	for all $\chi \notin \Delta$. Each perverse subquotient $M'$ of $M_{\overline{k}}$ satisfies
	\[
	\pi_{X?}(M'_\chi) = 0
	\]
	for all $\chi \notin \Delta$ (and either choice of support). 
\end{theorem*}
This theorem is proven from the relative classification of negligible sheaves, the vanishing theorems, and Artin's vanishing theorem.  We can also obtain a result for perverse degree one, see Theorem \ref{THM_Propagation2}.
\begin{theorem*}
	Let $T$ be a torus over $k$ and $X$ a scheme of finite type over $k$. Let $M \in \Perv{}{G}$. Suppose there is a closed subset $\Delta \subseteq\cvar{T}$ such that
	\[
	\pervCoh{1}{\pi_{X!}(M_\chi)} = 0.
	\]
	for all $\chi \notin \Delta$. Then we have
	\[
	\pervCoh{n}{\pi_{X!}(M_\chi)} = 0
	\]
	for all $\chi \notin \Delta$ and $n \geq 1$. 
\end{theorem*}
This theorem is proven by induction on the dimension of $T$ and the first propagation theorem. We are not sure how the relative propagation theorems generalize to higher perverse degrees nor how they generalize to semiabelian varieties. The application of Artin's vanishing theorem prevents a direct generalization to semiabelian varieties for the first propagation theorem. We also remark that the argument applies to all fields, not just finite fields, by the results of \cite{GabberLoeserTore}, but we do not pursue it here.

\subsubsection{Classification of negligible sheaves} We go into detail about the classification of negligible sheaves. One of the cornerstones of the theory presented here is the classification of negligible sheaves on a constant seimabelian variety over a base scheme. Similar classifications have been achieved for abelian varieties, by \cite{Kr_mer_2015}, for tori, by \cite{GabberLoeserTore}, and for semiabelian varieties over the complex numbers by \cite{liumaximwangpropagation}. We prove a relative classification theorem for constant semiabelian varieties over finite fields. The main result is Theorem \ref{THM_ClassCharactersRelative}, which states: 
	\begin{theorem*}
		Let $M \in \Perv{}{S\times X}$ be a geometrically irreducible perverse sheaf, such that there is a character $ \nu \in \mathscr{C}(S)$ with $$[\pi_{X!}(M_\nu)] = 0$$ in the $K$-group of $X$. There is a quotient $\pi\colon S \rightarrow S'$ with connected fibers of dimension $d > 0$, a perverse sheaf $N \in \Perv{}{S'\times X}$, and a character $\chi \in \widehat{S}(k)$ such that
		\[
		M = \pi^*(N)\otimes\SL_\chi[d].
		\]
		Moreover, for each $\chi \in \cvar{S}$ we have $\pi_{X!}(M_\chi) \neq 0$ if  and only if $\chi \in \overline{\chi}\cdot\pi^*(\cvar{S'})$. Dually, we also have $\pi_{X*}(M_\chi) \neq 0$ if  and only if $\chi \in \overline{\chi}\cdot\pi^*(\cvar{S'})$
	\end{theorem*}
	The argument is based on the classification proven in \cite{Kr_mer_2015}. We achieve this by closely following the argument given in \cite{GabberLoeserTore}. One of the crucial steps during the induction is the amplification of the absolute classification to a relative classification. Since we are working exclusively over finite fields, we can not reproduce the argument from \cite{GabberLoeserTore} and prove the amplification based on a specialization argument. 
	
	We remark two things about the classification which appear to be new. First, note that our classification uniquely characterizes the quotient and the character in terms of the ''Fourier support" of the perverse sheaf $M$. The Fourier support plays a central role to our construction of Tannakian category and could be useful in other contexts as well. 
	
	Secondly, the classification is arithmetic, in the sense that the quotient and the character descend to the base field $k$. To our knowledge, both of these points have not appeared in the literature before. The condition on the support is sure to generalize to other fields as well. We are not sure whether the descent generalizes to fields other than finite fields because the argument involves the computation of a Galois cohomology group. We also achieve the classification of negligible sheaves on all groups of the form $S\times U$ in Theorem \ref{THM_ClassificationFull}. In principle, the classification is possible on all connected commutative algebraic groups. For our purposes, this is not required because we can reduce most statements to a product $S\times U$ by the structure theorems for algebraic groups. 	
\subsection{Tannakian categories} We outline the construction of the Tannakian categories, in the hopes of making it more accessible. The primary goal of our construction is that the Tannakian categories come with fiber functors given by Fourier coefficients. 

 \subsubsection{Unramified characters} This is achieved by giving a new definition of an unramified character. This definition is given by Definition \ref{DEF_unramifiedhcar}. We believe this is the right notion to consider for the construction of fiber functors because the unramified characters can be characterized by the following proposition, at least when the perverse sheaf in question is pure or the group is affine. See Proposition \ref{PROP_GysinUnramifiedCriterion} for a slightly more general and precise statement. We hope that the following Proposition allows one to find closed unramified loci in practice.
\begin{theorem*}
	Let $M \in \Perv{}{G}$ and $\Delta\subseteq \cvar{G}$ a closed subset. Suppose $M$ is pure or $G$ is affine. The sheaf $M$ is $\Delta$-unramified, i.e. unramified at all characters $\chi \notin\Delta$, if and only if:
	\begin{enumerate}
		\item The sheaf $M$ is weakly unramified at all $\chi \notin \Delta$.
		\item We have
		\[
		\text{dim}(H^0(G, M_\chi)) = \text{dim}(H^0(G, M_{\chi'}))
		\]
		for all $\chi, \chi' \notin \Delta$.
	\end{enumerate}
\end{theorem*}
Roughly speaking, this says that a character is unramified if and only if the Fourier coefficient is concentrated in degree zero, pure, and has the \quotationMark{right} dimension, i.e. the Tannakian dimension of $M$. In this way, our definition is almost optimal for constructing fiber functors, because it mainly requires the Fourier coefficient to have the right dimension. We then prove the following general vanishing theorem in Theorem \ref{THM_VanishingTHM}.
\begin{theorem*}
	Let $M \in \Perv{}{G}$. There exists a proper thin subset $\Delta\subset\cvar{G}$ such that $M$ is $\Delta$-unramified. 
\end{theorem*}
\subsubsection{Tannakian categories} We now describe the construction of the Tannakian categories and the fiber functors in detail. The Tannakian categories we consider here were first constructed in \cite[Sec.~3.9]{GabberLoeserTore}. As is well-known, the convolution product of two perverse sheaves is, in general, not perverse. However, it is perverse up to certain negligible factors. One constructs a \textit{Tannakian category} with convolution as a tensor product by localizing the derived category at some category of negligible complexes, i.e. complexes which consist of negligible sheaves. One then proves that the \textit{localized convolution product} preserves perversity. Our goal is to localize the Tannakian category in such a way that the Fourier coefficients remain well-defined on the localization. 

This is achieved by gaining control over the negligible factors in a convolution product of two perverse sheaves. Let $\Delta\subset\cvar{G}$ be a closed subset. Roughly speaking, we show that the negligible factors in a convolution of two $\Delta$-unramified perverse sheaves have vanishing Fourier coefficients outside $\Delta$. Moreover, we then prove that the convolution product of two $\Delta$-unramified perverse sheaves is $\Delta$-unramified. This is the basic mechanism of our construction, that we now describe in more detail. 

Instead of neglecting all perverse sheaves, our primary goal is to only neglect perverse sheaves whose Fourier support is contained in a closed subset $\Delta$. The definition of the Fourier support of a complex $K \in \derCat{c}{G}{}$ is
\[
\text{FSupp}_?(K) := \{\chi \in \cvar{G} ~|~ H^*_?(K, M_\chi) \neq 0\}^{\text{cl}}
\]
See also Definition \ref{DEF_fsupp}. Note that we take the closure in the standard topology. An important point in the construction is Theorem \ref{THM_EqualitySupport}, which states
\[
\text{FSupp}_!(K)  = \text{FSupp}_*(K). 
\]
Based on this equality, we simply write $\text{FSupp}(K)$ for the Fourier support. This equality is proven from the classification of negligible sheaves. By Poincare duality, this implies that the Fourier support behaves well with respect to duality
\[
\text{FSupp}(K) = \text{FSupp}(K^\vee).
\]
The basic idea is then to localize the category of perverse sheaves at the category
\[
S_{\Delta} := \{M \in \Perv{}{G}~|~ \text{FSupp}(M) \subseteq \Delta\}.
\]
The Fourier support also behaves well with respect to filtrations of complexes, see Theorem \ref{THM_SupportLemma}. This compatibility of filtrations is also proven from the classification of characters. We do not know how to obtain the compatibility without the classification of negligible sheaves. This compatibility with exact sequences implies that the category $S_\Delta$ is Serre. Thus we can form the Serre quotient 
\[
\Perv{\Delta}{G} := \Perv{}{G}/S_{\Delta}.
\]
We can now form the category
\[
N_\Delta := \{K \in \derCat{c}{G}{} ~|~ \text{FSupp}(K) \subseteq \Delta\}.
\]
and then the Verdier quotient
\[
\derCat{\Delta}{G}{} := \derCat{c}{G}{}/N_\Delta.
\]
The compatibility of the Fourier support with filtrations implies that $\derCat{\Delta}{G}{}$ is a triangulated category with $t$-structure. The heart of this $t$-structure is naturally equivalent to $\Perv{\Delta}{G}$. This is of fundamental importance to our construction. The point of the above localization is that the Fourier coefficient
\[
\derCat{\Delta}{G}{}  \rightarrow \derCat{\text{coh}}{\Qbarl\text{-Vec}}{},~K\mapsto H^*_c(G, K_\chi)
\]
is well-defined for all $\chi \notin \Delta$. These functors will give the desired fiber functors on the Tannakian categories.
\subsubsection{Convolution product} We describe the basic mechanism behind our construction of the convolution product. To obtain the Tannakian categories, we must prove that the localized convolution product of two perverse sheaves is perverse. We proceed (almost) as in \cite[Prop.~3.9.3]{GabberLoeserTore}. The main point there, roughly speaking, is \cite[Prop.~3.9.2]{GabberLoeserTore}. Let $\Delta\subseteq \cvar{G}$ be a closed subset and $M, N \in \Perv{}{G}$ two $\Delta$-unramified perverse sheaves. We have isomorphisms
\[
H^*_c(G, (M*_!N)_\chi) \rightarrow H^*(G, (M*_*N)_\chi)
\] 
for all $\chi \notin \Delta$ by the Künneth formula. Morally (and, for tori, literally), Theorem \ref{THM_LocalizationTheCriterion} says that the above isomorphisms imply that the natural forget supports map
\[
M*_!N \rightarrow M*_*N
\]
is an isomorphism in $\derCat{\Delta}{G}{}$. Therefore, all arrows in the square
\[
\begin{tikzcd}
	{H^*_c(G, (M*_!N)_\chi) } \arrow[r] \arrow[d] & {H^*_c(G, (M*_*N)_\chi) } \arrow[d] \\
	{H^*(G, (M*_!N)_\chi) } \arrow[r]             & {H^*(G, (M*_*N)_\chi) }            
\end{tikzcd}
\]
are isomorphisms for all $\chi \notin \Delta$. The above criterion for isomorphism allows us to apply the decomposition theorem and the hard Lefschetz theorem to the convolution $M*_!N$. This implies that $M*_!N$ is perverse and that all perverse subquotients of $M*_!N$ are $\Delta$-unramified. 
\subsection{Character theory} We describe the character theory in detail, because it is rather long. This summary, hopefully, makes it more accessible. The first step is to construct the \textit{fundamental group $\pi_1^{\text{char}}(G)$} of $G$ in Definition \ref{DEF_FundGroup}. This is a commutative profinite group. The fundamental group behaves well in exact sequences, see Proposition \ref{PROP_FundamentalGroupExact} and Proposition \ref{PROP_DescentFundamentalGroup}. It also satisfies a product formula by Proposition \ref{PROP_KunnethFundamental} A \textit{character} is a continuous group morphism
\[
\pi_1^{\text{char}}(G) \rightarrow K
\]
to the group of units in a finite extension $K/\BQ_\ell$. The \textit{character group} is denoted $\cvar{G}$ and is the group of all characters as above. The exactness of the fundamental group carries over to the character group by Theorem \ref{THM_AppendixCharacterGroupExact} to Propositions \ref{PROP_CharExact} and Proposition \ref{PROP_CharDescent}. We have the Künneth formula $$\cvar{G\times G'} = \cvar{G}\times\cvar{G'}$$ by Proposition \ref{PROP_CharKunneth}. Crucially, Proposition \ref{PROP_CharDescent} implies that the map
\[
q^*\colon \cvar{G'} \rightarrow \cvar{G}
\]
is surjective with finite kernel for an isogeny $q\colon G \rightarrow G'$. 
\begin{example}
	Let $G = S$ be a semiabelian variety. Proposition \ref{PROP_fundcompsemiab} gives an isomorphism $$\pi_1^{\text{char}}(S) = \pi_1^t(S_{\overline{k}}).$$ Hence we recover the characters considered by Gabber-Loeser in \cite{GabberLoeserTore} for a torus and the characters considered by Krämer-Weissauer in \cite{Kr_mer_2015} for an abelian variety.
\end{example} 
\begin{example}\label{finiteorder_EXP}
	If a character $\chi \in \cvar{G}$ has finite order, then it \quotationMark{comes from a group morphism} 
	\[
	\chi\colon G(k_n) \rightarrow \Qbarl^*.
	\]
	by Lemma \ref{LEM_DescentCharAri}. We see that the characters of finite order coincide with the set of characters $\widehat{G}$ considered in \cite{KowalskiTannaka}. 
\end{example}
\begin{example}
	When $G = U$ is unipotent, the fundamental group $\pi_1^{\text{char}}(G)$ is a pro-$p$ group. This implies that any character has finite order. By Example \ref{finiteorder_EXP} the set of characters $\cvar{U}$ is equal to the set of characters considered in \cite{KowalskiTannaka}. In particular, the set $\cvar{U}$ coincides with set of closed points of the dual group $\widehat{U}_{\overline{k}}$.
\end{example}
We equip the space of characters $\cvar{G}$ with two topologies, the thin and the standard topology. When we say a subset is closed, we always mean closed with respect to the standard topology.
\begin{example}
	The characters $\cvar{U}$ are in bijection with the closed points of $\widehat{U}_{\overline{k}}$. We define the thin and the standard topology on $\cvar{U}$ by saying a set $Z \subseteq \cvar{U}$ is closed, if it is the set of closed points of a closed subscheme of $\widehat{U}_{\overline{k}}$ (see Definition \ref{DEF_charuniclosed}).
\end{example}
\begin{example}
	Let $S$ be a semiabelian variety. A subset $\Delta\subseteq\cvar{S}$ is called closed in the thin topology if it is a finite union of tacs (see Definition \ref{DEF_charsemiabthin}). We call it closed in the standard topology, if it is closed with respect to the Zariski topology constructed in \cite{GabberLoeserTore} for a torus (see Section \ref{SEC_charsemiabstand} for a precise definition).
\end{example}
\begin{example}
	Let $S$ be a semiabelian variety and $U$ a unipotent group. A subset $\Delta\subseteq \cvar{S\times U}$ is called \textit{thin}, if there are tacs $\Delta_1, \ldots, \Delta_n \subseteq \cvar{S}$ and closed subsets $Z_i \subseteq \cvar{U}$ such that
	\[
	\Delta = \Delta_1\times Z_1 \cup \ldots \cup \Delta_n\times Z_n.
	\]
	This is precisely the product topology on $\cvar{S\times U} = \cvar{S}\times\cvar{U}$ by Proposition \ref{PROP_TopologyProductStructure}. 
\end{example}
For a general group $G$, we reduce the definition of the topology to a product as follows. There exists an isogeny $q\colon G \rightarrow S\times U$ by the structure theorem, Theorem \ref{THM_AppendixStructureTheorem}. We declare the map
\[
q^*\colon \cvar{S\times U} \rightarrow \cvar{G}
\]
to be a quotient map in the thin and in the standard topology (see Definition \ref{DEF_topology}). This turns out to be a well-defined topology such that $\cvar{G}$ is a topological group in both topologies. In the thin topology, the space $\cvar{G}$ is a Noetherian irreducible space (see Proposition \ref{PROP_thintop}). In the standard topology, the space $\cvar{G}$ is a countable disjoint union of irreducible Noetherian spaces (see Proposition \ref{PROP_stdtop}) Based on this, we define the \textit{dimension} of a subset $\Delta\subseteq \cvar{G}$ to be the Krull dimension in the standard topology (see Definition \ref{DEF_topology}). For any isogeny $q\colon G\rightarrow G'$, the induced map
\[
q^*\colon \cvar{G'} \rightarrow \cvar{G}
\]
is a surjective closed quotient map by Proposition \ref{PROP_CharIsogenyClosed}. For a closed subset $\Delta\subseteq\cvar{G'}$, i.e. closed in the standard topology, we have
\[
\text{dim}(\Delta) = \text{dim}(q^*(\Delta))
\]
by Proposition \ref{PROP_IsogenyPreservesDimension}. For a thin subset $\Delta \subseteq\cvar{G}$ with codimension $d$, we have the estimate
\[
\frac{|\Delta\cap \widehat{G}(k_n)|}{|\widehat{G}(k_n)|} \ll_\Delta |k|^{-nd}
\]
by Proposition \ref{PROP_thinsubsetsestimate}.
\vspace*{2.5mm}
\begin{center}
	\textbf{Notations and Conventions}
\end{center}
\vspace*{1.5mm}
\begin{itemize}[]\item The letter $k$ denotes a finite field of characteristic $p$. We fix an algebraic closure $\overline{k}/k$. We denote the extension of degree $n$ in $\overline{k}$ by $k_n$.
	\item The letter $\ell$ is a prime distinct from $p$. The letter $\iota$ denotes an isomorphism $\iota\colon \Qbarl \rightarrow \BC$.
	\item The letters $X$ and $Y$ always denote a finite type scheme  over $k$. We define $X_{k'} := X\times_k k'$ for an extension $k'/k$. In section ..., we write $X$ for a finite type schemes over $k$ or $\overline{k}$.
	\item We write $H^*_?(X, K)$ to indicate a statement for both usual and compactly supported cohomology. The statement $H^*_?(X, K) = 0$ means $H^*_c(X, K) = 0$ and $H^*(X, K) = 0$. The statement there exists a choice of supports such that $H^*_?(X, K) = 0$ means $H^*_c(X, K) = 0$ or $H^*(X, K) = 0$. 
	\item We denote the category of $\Qbarl$-complexes on a scheme $X$ by $\derCat{c}{X}{\Qbarl}$ (see \cite[~{}2.2.14]{BBD}). The category of perverse sheaves is denoted by $\Perv{}{X}$. 
	\item We call a perverse sheaf $M$ on $X_{\overline{k}}$ \textit{arithmetic} if there is a perverse sheaf $M' \in \Perv{}{X_{k'}}$ over a finite extension $k'/k$ such that $M'_{\overline{k}}\cong M$. It follows from the Jordan-Hölder theorem and \cite[Prop.~5.1.2]{BBD} that all irreducible subquotients of an arithmetic sheaf are arithmetic. We call a perverse sheaf $M$ on $X_{\overline{k}}$ \textit{quasi-arithmetic} if all irreducible subquotients of $M$ are arithmetic. We only consider quasi-arithmetic perverse sheaves in this text. Note that the category of quasi-arithmetic perverse sheaves is a Serre subcategory of the category of all perverse sheaves on $X_{\overline{k}}$. We define the derived category of complexes on $X_{\overline{k}}$ to be the category of all complexes $K$ on $X_{\overline{k}}$ such that $\pervCoh{n}{K}$ is quasi-arithmetic for all $n \in \BZ$. The category $\derCat{c}{G_{\overline{k}}}{\Qbarl}$ is triangulated and has a perverse $t$-structure. Remark that if $K \in \derCat{c}{G}{}$, then $K_{\overline{k}}$ is quasi-arithmetic. 
	\item A group $G, G',\ldots $ or $H$ is a connected commutative algebraic group over $k$. A finite group $\Gamma$ is a finite commutative group scheme over $k$. We must sometimes allow for disconnected groups - we indicate it by saying the group $G$ is \textit{possibly disconnected.} In Sections ... and ...,  we consider groups over $k$ or $\overline{k}$ without changing the notation. We explain the change of notation at the beginning of the section.
	\item The letter $U$ always denotes a commutative unipotent group, $A$ an abelian variety, $T$ a torus, and $S$ a semiabelian variety over $k$. Note that there is a maximal torus $T \subseteq S$ of dimension $d_t$. We sometimes write the dimension of a semiabelian variety $S$ as $d = d_t + d_a$.
	\item When we write $G = T\times U$, then we assume $G$ is isomorphic to the product of a torus and a unipotent group $U$ and we fix such an isomorphism. This notation also applies to other choices of combinations of groups.
\end{itemize}
\newpage

\subfile{./Preliminaries}
\newpage
\subfile{./VanishingTheorems}
\newpage
\subfile{./Classification}
\newpage
\subfile{./supports}
\newpage
\subfile{./TannakianCategories}
\newpage

\subfile{./Appendix}
\newpage
\tableofcontents
	\newpage
	\printbibliography
\end{document}

%% file: Preliminaries.tex
	\section{Preliminaries}

	\subsection{Fundamental group}
	We construct a  fundamental group that classifies the characters of an algebraic group. The construction is such that the arithmetic characters recover the dual group from \cite{KowalskiTannaka}. By introducing the fundamental group below, we also subsume all other previously considered definitions for a group $G$. 
	
	It is technically convenient to introduce the fundamental group of a disconnected group as well - we shall see, however, that this definition only remembers the connected component (see Corollary \ref{COR_ImageFundGroup}). This will be helpful in proving Proposition \ref{PROP_DescentFundamentalGroup}, which plays a central role in many of our arguments with the character group.
	\begin{definition}\label{DEF_FundGroup}
		Let $G$ be a possibly disconnected group. For all $m|n$ with $n \geq 1$ we define the \textit{trace} to be
		\[
		\text{Tr}_{k_m/k_n}\colon G (k_m) \rightarrow G (k_n),~x \mapsto \prod_{\substack{m|i|n \\ i > 0}} \text{Fr}_{k_i}(x).
		\]
		We define the \textit{fundamental group of $G$} to be the limit along the traces
		\[
		\pi_1^{\text{char}}(G) := \varprojlim_{n} {G}(k_n).
		\]
		This is an inverse limit in the category of commutative profinite groups along the ordered set of positive integers ordered by the divisibility relation. 
	\end{definition}
	\begin{remark}\label{REM_InvBaseChangeFundamGrp}
		Let $G $ be a possibly disconnected group over a finite field $k$ and $n \geq 1$. There is a map
		\[
		\pi_1^{\text{char}}(G) \rightarrow \pi_1^{\text{char}}(G_{k_n})
		\]
		given by functoriality of the limit. The diagram defining the fundamental group of $G_{k_n}$ is final in the diagram defining the fundamental group of $G$. Hence the above map is a canonical isomorphism. In particular, the fundamental group does not depend on the basefield $k$.
	\end{remark}
	\subsubsection{Weil group cohomology} We require certain basic facts on the group cohomology of algebraic groups over a finite field for our constructions. The crucial fact to remember is that the Lang torsor is a torsor in the etale topology, hence it induces a surjective map on the geometric points.
	\begin{proposition}[{\cite[p.~113,~Prop.~3]{Serre_AlgebraicGroupsClassFields}}]\label{PROP_LangTorsorTorsor}
		Let $n \geq 1$ a positive integer. The \textit{Lang torsor}
		\[
		\SL_n\colon G_{k_n} \rightarrow G_{k_n}, x \mapsto x^{-1}\text{Fr}_{k_n}(x)
		\] 
		is a Galois covering with automorphism group $G(k_n)$. The automorphisms are the translations by elements in $G(k_n)$. The map
		\[
		G(\overline{k}) \rightarrow G(\overline{k}),~x\mapsto x^{-1}\text{Fr}_{k_n}(x)
		\]
		is surjective. For all $m|n$, there is a commutative diagram 
		\[
		\begin{tikzcd}
			G_{k_m} \arrow[rd, "\SL_m"'] \arrow[r, "\text{Tr}_{k_m/k_n}"] & G_{k_m} \arrow[d, "\SL_n"] \\
			& G_{k_m}                                 
		\end{tikzcd}
		\]
	\end{proposition}
	\begin{remark}
		The diagram in Proposition \ref{PROP_LangTorsorTorsor} allows us to construct a pro-object in the category of Galois coverings of $G$
		\[
		\varprojlim_{n} G_{\overline{k}} \rightarrow G_{\overline{k}}
		\]
		by passing the Lang torsors to the limit. The fundamental group  $\pi_1^{\text{char}}$ is the Galois group of the pro-etale Lang cover.
	\end{remark}
	The following exact sequence can be deduced from Proposition \ref{PROP_LangTorsorTorsor}  by computing the group cohomology of the Frobenius-module $G(\overline{k})$.
	\begin{corollary}\label{COR_ExactSequencePoints}
		Let $G, G',$ and $G''$ be possibly disconnected groups. Suppose there is a short exact sequence
		\[
		0 \rightarrow G' \rightarrow G \rightarrow G'' \rightarrow 0.
		\] 
		Let $n, m \geq 1$ be positive integers such that $m|n$. We have a commutative diagram with exact columns
		\[
		\begin{tikzcd}[column sep=tiny]		%
			0 \arrow[r] & G'(k_m) \arrow[r] \arrow[d, "\text{Tr}_{k_m/k_n}"] & G(k_m) \arrow[r] \arrow[d, "\text{Tr}_{k_m/k_n}"] & G''(k_m) \arrow[r] \arrow[d, "\text{Tr}_{k_m/k_n}"] & \pi_0(G'_{\overline{k}})_{\text{Fr}_{k_m}} \arrow[r] \arrow[d, "\pi"] & \pi_0(G_{\overline{k}})_{\text{Fr}_{k_m}} \arrow[r] \arrow[d, "\pi"] & \pi_0(G''_{\overline{k}})_{\text{Fr}_{k_m}} \arrow[d, "\pi"] \arrow[r] & 0 \\
			0 \arrow[r]                       & G'(k_n) \arrow[r]                                  & G(k_n) \arrow[r]                                  & G''(k_n) \arrow[r]                                  & \pi_0(G'_{\overline{k}})_{\text{Fr}_{k_n}} \arrow[r]                  & \pi_0(G_{\overline{k}})_{\text{Fr}_{k_n}} \arrow[r]                  & \pi_0(G''_{\overline{k}})_{\text{Fr}_{k_n}} \arrow[r]                  & 0
		\end{tikzcd}
		\]
		where $\pi$ is the natural map induced by functoriality of the coinvariants.
	\end{corollary}
	We also note the following well-known surjectivity of the trace map (see the reference \cite[~{}Ch.~VI,~§1, ~6]{Serre_AlgebraicGroupsClassFields})
	\begin{lemma}\label{LEM_TraceIsSurjectiveConnected}
		Suppose we are given positive integers $n,m \geq 1$ with $n|m$. We have
		\[
		\text{Tr}_{k_{n}/k_m}(G(k_n)) = G(k_m).
		\]
	\end{lemma}
	The following is certainly also well-known, but we do not know where it is recorded in the literature. As remarked before, it implies that the fundamental group is unaffected by restriction to the connected component.
	\begin{corollary}\label{COR_ImageFundGroup}
		Let $G$ be a possibly disconnected group. For all positive integers $n, m \geq 1$ with $m|n$, the image of 
		\[
		\text{Tr}_{k_m/k_n}\colon G(k_m) \rightarrow G(k_n)
		\]
		is equal to $G^{0}(k_n)$ if $(m/n)~|~|\pi_0(G_{\overline{k}})|$. In particular, we have
		\[
		\pi_1^{\text{char}}(G^0) = \pi_1^{\text{char}}(G).
		\]
	\end{corollary}
	\begin{proof}
		By Corollary \ref{COR_ExactSequencePoints}, we have a commutative diagram with exact columns
		\[
		\begin{tikzcd}
			0 \arrow[r] & G^0(k_m) \arrow[r] \arrow[d, "\text{Tr}_{k_m/k_n}", two heads] & G(k_m) \arrow[r] \arrow[d, "\text{Tr}_{k_m/k_n}"] & (G/G^0)(k_m) \arrow[r] \arrow[d, "\text{Tr}_{k_m/k_n}"] & {} \\
			0 \arrow[r] & G^0(k_n) \arrow[r]                                             & G(k_n) \arrow[r]                                  & (G/G^0)(k_n) \arrow[r] & 0             
		\end{tikzcd}
		\]
		We can compose the trace with the natural inclusion
		\[
		 (G/G^0)(k_n) \xrightarrow{\text{Tr}_{k_m/k_n}} (G/G^0)(k_m) \rightarrow (G/G^0)(k_n)
		\]
		to obtain multiplication by $m/n$. One can now compute the image of the trace from the Snake Lemma. 
	\end{proof}
	\subsubsection{Functoriality of the fundamental group} 
	\begin{proposition}\label{PROP_FundamentalGrpFunctorial}
		Let $G$ and $H$ be possibly disconnected groups over a finite field $k$ and $\pi\colon G \rightarrow H$ be a morphism. We have a continuous map
		\[
		\pi_*\colon \pi_1^{\text{char}}(G) \rightarrow \pi_1^{\text{char}}(H).
		\]
		which is the limit of the maps
		\[
		G (k_n) \rightarrow H(k_n).
		\]
		The fundamental group defines a functor from the category of possibly disconnected groups defined over $k_0$ to the category of profinite commutative groups.
	\end{proposition}
	\begin{proposition}\label{PROP_FundamentalGroupExact}
		We consider an exact sequence 
		\[
		0 \rightarrow G' \rightarrow G \rightarrow G'' \rightarrow 0.
		\]
		We have a short exact sequence
		\[
		0 \rightarrow \pi_1^{\text{char}}(G')\rightarrow \pi_1^{\text{char}}(G)\rightarrow \pi_1^{\text{char}}(G'') \rightarrow 0
		\]
		of profinite groups.
	\end{proposition}
	\begin{proof} We have an exact sequence by Corollary \ref{COR_ExactSequencePoints}
		\[
		0 \rightarrow G'(k_n) \rightarrow G(k_n) \rightarrow G''(k_n) \rightarrow 0.
		\]
	\end{proof}
	We obtain the product formula.
	\begin{proposition}\label{PROP_KunnethFundamental}
		The natural map
		\[
		\pi_1^{\text{char}}(G\times G') \rightarrow \pi_1^{\text{char}}(G)\times\pi_1^{\text{char}}(G'),~\gamma \mapsto (\pi_{1*}(\gamma), \pi_{2*}(\gamma))
		\]
		is an isomorphism.
	\end{proposition}
	\begin{proof} Indeed, we have exact sequences
		\begin{align*}
			0 \rightarrow \pi_1^{\text{char}}(G')&\rightarrow \pi_1^{\text{char}}(G\times G') \rightarrow \pi_1^{\text{char}}(G) \rightarrow 0 \\
			0 \rightarrow \pi_1^{\text{char}}(G)&\rightarrow \pi_1^{\text{char}}(G\times G') \rightarrow \pi_1^{\text{char}}(G) \rightarrow 0.
		\end{align*}
	\end{proof}
	The following lemma plays a fundamental role in the theory:
	\begin{proposition}\label{PROP_DescentFundamentalGroup}
		We consider an exact sequence
		\[
		0 \rightarrow \Gamma \rightarrow G \rightarrow G' \rightarrow 0,
		\]
		where $\Gamma$ is a finite group. We have an exact sequence
		\[
		0 \rightarrow \pi_1^{\text{char}}(G) \rightarrow \pi_1^{\text{char}}(G') \rightarrow \Gamma(\overline{k}) \rightarrow 0.
		\] 
	\end{proposition}
	\begin{proof} By Remark \ref{REM_InvBaseChangeFundamGrp}, it is sufficient if we provide the exact sequence after a finite extension of $k$. Thus we can assume the action of Frobenius is trivial on $\Gamma(\overline{k})$. We can pass the exact sequence in Corollary \ref{COR_ExactSequencePoints} to the limit and obtain the short exact sequence
		\[
		0 \rightarrow \pi_1^{\text{char}}(\Gamma) \rightarrow \pi_1^{\text{char}}(G) \rightarrow \pi_1^{\text{char}}(G') \rightarrow \Gamma(\overline{k}) \rightarrow 0.
		\]
		Our assumption on the Frobenius action on $\Gamma(\overline{k})$ implies that the maps in the exact sequence from Corollary \ref{COR_ExactSequencePoints} are given by the identity. This justifies our evaluation of the term on the right in the above exact sequence. Lemma \ref{COR_ImageFundGroup} show $\pi_1^{\text{char}}(\Gamma)  = 0.$
	\end{proof}

	\subsection{Character variety}
	\begin{definition} A \textit{character of $G$}  is a continuous group morphism
		\[
		\chi\colon \pi_1^{\text{char}}(G) \rightarrow K^*
		\]
		for a finite extension $K/\BQ_\ell$. We denote the group of characters by $\mathscr{C}(G)$. The set $\cvar{G}$ is called the \textit{character variety} of $G$.
		
		Let $n \geq 1$. An \textit{arithmetic character $\chi$ over $k_n$} is a group morphism
		\[
		G(k_n) \rightarrow \Qbarl^*.
		\]
		We denote the set of arithmetic character over $k_n$ by $\widehat{G}(k_n)$, i.e. we denote it with the notation for the Pontryagin dual.
	\end{definition}
	
	We begin by constructing the arithmetic characters. We formulate the construction of the Lang torsor in convenient form.
	\begin{proposition}\label{PROP_LangTorsorArithmetic}
		For each $n \geq 1$, there is a surjective continuous group morphism
		\[
		\SL_{n*}\colon \pi_1(G_{k_n}, e) \rightarrow G(k_n).
		\]
		This morphism is functorial in $G$.
	\end{proposition}
	\begin{proof} We have proven in Proposition \ref{PROP_LangTorsorTorsor} that the Lang torsor is finite etale cover with Galois group $G(k_n)$. 
	\end{proof}
	We now define the arithmetic character sheaves and describe their trace functions.
	\begin{definition} Let $\chi \in \widehat{G}(k_n)$. We construct a character of $\pi_1(G_{k_n}, e)$ by composing $\SL_{n*}$ with $\chi$, as follows
		\[
		\pi_1(G_{k_n}, e) \xrightarrow{\SL_{n*}} G(k_n) \xrightarrow{\chi} \Qbarl^*.
		\]
		The associated lisse sheaf is called the \textit{arithmetic character sheaf associated to $\chi$} and is denoted by $\SL_\chi$. 
		
		For a complex $K \in \derCat{c}{G\times X}{\Qbarl}$, we denote by
		\[
		K_\chi := K\otimes\pi_{G}^*(\SL_\chi)
		\]
		the twist of $K$ by $\SL_\chi$. Let $m \geq 1$ be a positive integer with $m|n$. The trace function of $K_\chi$ is given by
		\[
		t_{K_\chi}(g, x)= \chi(\text{Tr}_{k_m/k_n}(g))t_K(g, x)
		\]
		for all  $(g, x) \in G(k_m)\times X(k_m)$.
	\end{definition}
	\begin{lemma}\label{LEM_FundGroupSurjects}
		There is a surjective, continuous group morphism
		\[
		\SL\colon \pi_1(G_{\overline{k}}, e) \rightarrow \pi_1^{\text{char}}(G).
		\]
		given by passing the diagram of group morphisms $\SL_n$ described in Lemma \ref{PROP_LangTorsorArithmetic} to the limit. This morphism is functorial in the group $G$.
	\end{lemma}
	\begin{proof}
		The maps $\SL_{n*}$ are surjective and are compatible with the system of trace mappings by Proposition \ref{PROP_LangTorsorArithmetic}. Therefore, the map to the limit exists and is surjective.
	\end{proof}
	We construct character sheaves for all characters.
	\begin{definition} Let $\chi \in\cvar{G}$. We construct a character of $\pi_1(G_{\overline{k}}, e)$ by composing $\SL$ with $\chi$, as follows
		\[
		\pi_1(G_{\overline{k}}, e) \xrightarrow{\SL} \pi_1^{\text{char}}(G) \xrightarrow{\chi} \Qbarl^*.
		\]
		The associated lisse sheaf is called the \textit{character sheaf associated to $\chi$} and is denoted by $\SL_\chi$. 
		
		For a complex $K \in \derCat{c}{G\times X}{\Qbarl}$, we denote by
		\[
		K_\chi := K\otimes\pi_{G}^*(\SL_\chi)
		\]
		the twist of $K$ by $\SL_\chi$. 
	\end{definition}
	We now reformulate the exactness properties we obtained in the last section for the character group. The fundamental tool is Theorem \ref{THM_AppendixCharacterGroupExact}. We only state the exactness properties and note the respective proposition from the previous section.
	\begin{proposition}\label{PROP_CharFunc}
		Let $\pi\colon G \rightarrow G'$ a group morphism. We have an induced map
		\[
		\pi^*\colon \mathscr{C}(G') \rightarrow \mathscr{C}(G).
		\] 
		The assignment $G \mapsto \mathscr{C}(G)$ is a contravariant functor. Let $\chi \in \cvar{G}$ be a character. There is a natural isomorphism
		\[
		\pi^*(\SL_{\chi}) = \SL_{\pi^*(\chi)}.
		\]
	\end{proposition}
	\begin{proof}
		This follows from Proposition \ref{PROP_FundamentalGrpFunctorial}. The functoriality of the character sheaf follows from the functoriality noted in Lemma \ref{LEM_FundGroupSurjects}.
	\end{proof}
	\begin{proposition}\label{PROP_CharExact}
		Let 
		\[
		0 \rightarrow G' \rightarrow G \rightarrow G'' \rightarrow 0
		\]
		be an exact sequence. We have an induced exact sequence on character groups
		\[
		0 \rightarrow \cvar{G''} \rightarrow \cvar{G} \rightarrow \cvar{G'} \rightarrow 0.
		\]
	\end{proposition}
	\begin{proof}
		This follows from Theorem \ref{THM_AppendixCharacterGroupExact} applied to Proposition \ref{PROP_FundamentalGroupExact}. 
	\end{proof}
	\begin{proposition}\label{PROP_CharDescent}
		Let 
		\[
		0 \rightarrow \Gamma \rightarrow G\rightarrow G' \rightarrow 0
		\]
		be an exact sequence such that $\Gamma$ is finite. We have an induced exact sequence
		\[
		0 \rightarrow   \widehat{\Gamma}(\overline{k})\rightarrow \cvar{G'} \rightarrow \cvar{G} \rightarrow 0.
		\]
		The symbol $\widehat{\Gamma}(\overline{k})$ refers to the group of $\ell$-adic characters $\Gamma(\overline{k}) \rightarrow \Qbarl^*$.
	\end{proposition}
	\begin{proof}
		This is an immediate consequence of Proposition \ref{PROP_DescentFundamentalGroup} and Theorem \ref{THM_AppendixCharacterGroupExact}.
	\end{proof}
	\begin{proposition}\label{PROP_CharKunneth}
		We have the Künneth isomorphism 
		\[
		\mathscr{C}(G\times G') =\mathscr{C}(G)\times \mathscr{C}(G').
		\]
		Let $(\chi, \chi') \in \cvar{G\times G'}$ be a character. We have a canonical isomorphism of lisse sheaves
		\[
		\SL_{(\chi, \chi')}= \SL_\chi\boxtimes \SL_{\chi'}.
		\]
	\end{proposition}
	\begin{proof}
		This follows from Proposition \ref{PROP_KunnethFundamental} and Proposition \ref{PROP_CharFunc}.
	\end{proof}
	\subsubsection{Cohomology of character sheaves and projection formulas} 
	An important property of character sheaves is the vanishing of their cohomology. We can recover it in our current setting from the Künneth formula.
	\begin{proposition}\label{PROP_VanishingCohomology}
	Let $\chi \in \cvar{G}$ be a non-trivial character. We have
	\[
	H^*_?(G, K_\chi) = 0.
	\]
\end{proposition}
\begin{proof}
	The argument in \cite{SaibiFourier} applies verbatim, because we have the isomorphism $m^*(\SL_\chi) = \SL_\chi\boxtimes\SL_\chi$ by Proposition \ref{PROP_CharKunneth}.
\end{proof}
We record a projection formula for relative Fourier coefficients.
\begin{proposition}\label{PROP_ProjectionFormula2}
	Let $\pi\colon G \rightarrow G'$ be a group morphism with connected kernel $G''$. Let $\chi \in \cvar{G'}$ be a character and $K \in \derCat{c}{G\times X}{\Qbarl}$. We have
	\[
	\pi_?(K_{\pi^*(\chi)}) = \pi_?(K)_\chi.
	\]
\end{proposition}
\begin{proof}
	This follows from the projection formula for lisse sheaves.
\end{proof}
	There is also the following relative variant, which we present together with a certain projection formula.
	\begin{proposition}\label{PROP_ProjectionFormula1}
		Let $\pi\colon G \rightarrow G'$ be a surjective morphism with possibly disconnected kernel $G''$. Let $X$ be a scheme and $K \in \derCat{c}{G'\times X}{\Qbarl}$. Consider $\chi \in \cvar{G}$.  We have the projection formula
		\[
		\pi_{?}((\pi^* K)_\chi) = K\otimes \pi_{?}(\SL_\chi).
		\]
		To evaluate $\pi_{?}(\SL_\chi)$, we distinguish two cases:	
		If $\pi^*(\chi) = 1$, there exists $\chi' \in \mathscr{C}(G')$ with $\pi^*(\chi') = \chi$ by Proposition \ref{PROP_CharExact}. We have
		\[
		\pi_{?}(\SL_\chi) = \pi_{?}(\Qbarl)_{\chi'}.
		\]
		If $\pi^*(\chi) \neq 1$, then
		\[
		\pi_{?}(\SL_\chi) = 0.
		\]
		Suppose the kernel $G''$ is connected. The pushforward satisfies
		\[
		{H}^n(\pi_{?}(\Qbarl)) = H^n(G'', \Qbarl)\otimes\Qbarl.
		\]
	\end{proposition}
	\begin{proof} We have a Cartesian square
		\[
		\begin{tikzcd}
			G\times G''\times X \arrow[d, "m"] \arrow[r, "\pi_G"] & G \times X \arrow[d] \\
			G\times X \arrow[r]                     & G' \times X        
		\end{tikzcd}
		\]
		By smooth basechange, it is sufficient if we prove the projection formula after base change to $G$. In the split case, it is a consequence of Proposition \ref{PROP_CharKunneth} and the projection formula.
		
		The evaluation of the pushforward in the first case follows from Proposition \ref{PROP_ProjectionFormula2} and Proposition \ref{PROP_CharFunc}. The vanishing can be proven from smooth base change, Proposition \ref{PROP_VanishingCohomology}, and the Künneth formula. Moreover, \cite[Prop.~4.2.5]{BBD} allows us to reduce the statement on $\pi_?(\Qbarl)$ to the split case by the above Cartesian square.
\end{proof}
\subsubsection{Frobenius invariant characters and subsets}
We now prove a certain lemma, which will allow us to prove that Frobenius invariant linear subsets come from an arithmetic subgroup. We begin by computing the cohomology of the Weil group acting on the character group. We characterize quotients in terms of their character subgroup. This allows us to give a certain characterization of Frobenius invariant subsets inside the character variety.  
\begin{definition}
	We define Frobenius on the fundamental group to be the morphism associated to Frobenius $\text{Fr}_k\colon G \rightarrow G$
	\[
	\text{Fr}^*_k\colon \pi_1^{\text{char}}(G) \rightarrow \pi_1^{\text{char}}(G).
	\]
	This morphims admits an inverse, because Frobenius acts bijectively on the set $G(\overline{k})$. Hence we obtain a action of the Weil group $W_k$ on $\pi_1^{\text{char}}(G)$ by continuous group automorphisms.
\end{definition}
\begin{proposition}\label{PROP_FundamentalGroupCoinvariants}
	We have
	\[
	\pi_1^{\text{char}}(G)_{\text{Fr}_k^n} = G(k_n),
	\]
	in the sense that the kernel of the natural projection $\pi_1^{\text{char}}(G) \rightarrow \pi_1^{\text{char}}(G)_{\text{Fr}_k^n}$ agrees with the kernel of the natural projection $\pi_1^{\text{char}}(G) \rightarrow G(k_n)$.
\end{proposition}
\begin{proof}
	Note that we have an exact sequence of algebraic groups
	\[
	0 \rightarrow G(k_n) \rightarrow G \xrightarrow{\text{Fr}_{k}^n - 1} G\rightarrow 0
	\]
	given by the Lang torsor. We can apply Proposition \ref{PROP_DescentFundamentalGroup} to obtain the statement of the Lemma.
\end{proof}
\begin{definition}
	We define the \textit{Frobenius morphism} $\text{Fr}_{k}^*$ to be the morphism
	\[
	\text{Fr}_k^*\colon \cvar{G} \rightarrow \cvar{G}
	\]
	induced by the Frobenius morphism defined above. We obtain an action of the Weil group $W_k$ on $\cvar{G}$.
\end{definition}
\begin{proposition}\label{PROP_GroupCohomologyCharacterGroup}
	We can compute the group cohomology with respect to the Weil group 
	\begin{align*}
		H^0(W_{k}, \mathscr{C}(G)) &= \widehat{G}(k) \\
		H^1(W_{k}, \mathscr{C}(G)) &= 0
	\end{align*}
\end{proposition}
\begin{proof}
	By Proposition \ref{PROP_CharDescent}, the Lang torsor induces an exact sequence
	\[
	0 \rightarrow \widehat{G}(k) \rightarrow \mathscr{C}(G) \xrightarrow{\text{Fr}_k^* - 1} \mathscr{C}(G) \rightarrow 0.
	\]
\end{proof}
We also need the following lemma:
\begin{lemma}\label{LEM_DescentCharAri}
	For a character $\chi\in \mathscr{C}(G)$, the following three conditions are equivalent:
	\begin{enumerate}\item The group morphism $\chi$ factors through a projection
		\[
		\pi_1^{\text{char}}(G) \rightarrow G(k_n),
		\]
		i.e. is arithmetic.
		\item We have $\text{Fr}^*(\chi) = \chi$.
		\item The character $\chi$ has finite order.	
	\end{enumerate}
\end{lemma}
\begin{proof}
	The equivalence of the first two conditions follows from Proposition \ref{PROP_FundamentalGroupCoinvariants}. Note that if $\chi$ has finite order, then it factors through one of the projections, hence it satisfies (1). Moreover, (1) clearly implies (3), hence the claim.
\end{proof}
We prove that the character subgroup uniquely determines the quotients of a group.
\begin{lemma}\label{LEM_CharactersFindSubgroups}
	Let $G$ be a connected commutative algebraic group $G$ and $H, H' \subseteq G$ connected subgroups in $G$. Then $H \subseteq H'$ if and only if
	\[
	\pi_{G/H}^*\mathscr{C}(G/H) \subseteq \pi_{G/H'}^*\mathscr{C}(G/H').
	\]
\end{lemma}
\begin{proof}
	We have the diagram with exact rows and columns
	\[
	\begin{tikzcd}
		& 0 \arrow[d]                                     & 0 \arrow[d]         \\
		0 \arrow[r] & \SC(G/H + H') \arrow[d] \arrow[r]               & \SC(G/H') \arrow[d] \\
		0 \arrow[r] & \SC(G/H) \arrow[d] \arrow[r] \arrow[ru, dashed] & \SC(G) \arrow[d]    \\
		0 \arrow[r] & \SC(H/H\cap H') \arrow[d] \arrow[r]             & \SC(H') \arrow[d]   \\
		& 0                                               & 0                  
	\end{tikzcd}
	\]
	Hence
	\[
	\mathscr{C}(H/H\cap H') = 0.
	\]
	This implies $H \subseteq H'$ by Corollary \ref{COR_ImageFundGroup}.
\end{proof}
\begin{corollary}\label{COR_TACSDetermineGroupUniquely}
	Let $H$ and $H'$ be connected subgroups of $G$ and $\chi_1, \chi_2 \in \mathscr{C}(G)$. We have
	\[
	\chi_1\cdot\pi_{G/H_1}^*(\mathscr{C}(G/H_1)) = \chi_2\cdot\pi_{G/H_2}^*(\mathscr{C}(G/H_2))
	\]
	if and only if $H_1 = H_2$ and $\chi_1/\chi_2 \in \mathscr{C}(G/H_1)$.
\end{corollary}
\begin{proof}
	The equality implies 
	\[
	\chi_2\cdot\pi_{G/H_2}^*(\mathscr{C}(G/H_2)) = \chi_1\cdot\pi_{G/H_2}^*(\mathscr{C}(G/H_2))
	\]
	and hence
	\[
	\pi_{G/H_1}^*(\mathscr{C}(G/H_1)) =\pi_{G/H_2}^*(\mathscr{C}(G/H_2)). 
	\]
	This implies the claim by Lemma \ref{LEM_CharactersFindSubgroups}.
\end{proof}
\begin{remark}
	Let $H_{\overline{k}} \subseteq G_{\overline{k}}$ be a subgroup. There is a finite extension $k'/k$ such that the subgroup descends to a subgroup $H_{k'} \subseteq G_{k'}$. By Remark \ref{REM_InvBaseChangeFundamGrp}, the image of $\cvar{G_{\overline{k}}/H_{\overline{k}}}$ in $\cvar{G}$ is well-defined.
\end{remark}
\begin{theorem}\label{THM_DescendingTACS}
	Let $H_{\overline{k}} \subseteq G_{\overline{k}}$ a closed, connected subgroup. Let $\chi \in \mathscr{C}(G)$ be a character such that the subset $\Delta \subseteq \mathscr{C}(G)$ given by
	\[
	\Delta := \chi\cdot\pi_{G_{\overline{k}}/H_{\overline{k}}}^*(\mathscr{C}(G_{\overline{k}}/H_{\overline{k}}))
	\]
	satisfies $\text{Fr}_k^*(\Delta) = \Delta$. The group $H_{\overline{k}}$ descends to a subgroup $H \subseteq G$ and there exists an arithmetic character $\chi' \in \widehat{G}(k)$ with
	\[
	\Delta = \chi'\cdot \pi_{G/H}^*(\mathscr{C}(G/H)). 
	\]
\end{theorem}
\begin{proof}
	We have
	\[
	\text{Fr}_k^*(\chi)\cdot \pi_{G_{\overline{k}}/\text{Fr}_k(H_{\overline{k}})}^*(\mathscr{C}(G_{\overline{k}}/\text{Fr}_k(H_{\overline{k}}))) = \chi\cdot\pi_{G_{\overline{k}}/H_{\overline{k}}}^*(\mathscr{C}(G_{\overline{k}}/H_{\overline{k}})).
	\]
	Corollary \ref{COR_TACSDetermineGroupUniquely}, we have
	\[
	\text{Fr}_k(H_{\overline{k}}) = H_{\overline{k}}.
	\]
	This implies $H_{\overline{k}}$ exists as a subgroup $H \subseteq G$ over $k$.
	
	We define the cocycle
	\[
	\alpha\colon W_k \rightarrow \mathscr{C}(G/H),~\sigma \mapsto \frac{\sigma^*(\chi)}{\chi}.
	\]
	The map $\alpha$ defines a cohomology class
	\[
	\alpha \in H^1(W_k, \mathscr{C}(G/H))
	\]
	in the group cohomology of the $W_k$-module $\mathscr{C}(G/H)$. This class vanishes by Proposition \ref{PROP_GroupCohomologyCharacterGroup}. Thus there is $\chi_0 \in \mathscr{C}(G/H)$ such that
	\[
	\frac{\sigma(\chi/\chi_0)}{\chi/\chi_0} = 1.
	\]
	Hence $\chi/\chi_0$ is Frobenius invariant. This implies $\chi/\chi_0 \in \widehat{G}(k)$ by Lemma \ref{LEM_DescentCharAri}. Since $\chi/\chi_0 \in \Delta$, we are done.
\end{proof}
	\subsection{Semiabelian varieties} We discuss the character theory of a semiabelian variety $S$ over a finite field $k$ in more detail. We classify the characters by identifying the character group with the Tate group of the semiabelian variety. We then introduce the character variety and recall certain well-known facts about relative Fourier coefficients of semiabelian varieties.
	\subsubsection{Characters} We begin by proving the comparison theorem. We compare the fundamental group constructed here to the tame fundamental group. The definition of the tame fundamental group requires us to construct a normal crossing compactification of $S$. We recall a well-known construction and show that all character sheaves are tamely ramified. We prove the comparison following a well-known method.

	\begin{definition}
		We have an exact sequence
		\[
		0 \rightarrow T_{\overline{k}} \rightarrow S_{\overline{k}} \rightarrow A_{\overline{k}} \rightarrow 0.
		\]
		The scheme $S_{\overline{k}}$ is a $T_{\overline{k}}$-torsor over $A_{\overline{k}}$. We can identify $T_{\overline{k}} \cong \BG_m^n$ for some $n \geq 0$. Note that we have an equivariant simple normal crossing compactification $T_{\overline{k}} \rightarrow \BBP^n$. Since this compactification is equivariant, we can construct the compactification
		\[
		\overline{S}:= \BBP^n\times_{T_{\overline{k}}}S.
		\]
		This yields a simple normal crossing compactifiction $S \rightarrow \overline{S}$. We consider the tame fundamental group $\pi_1^t(S_{\overline{k}}, e)$ with respect to this compactification.
	\end{definition} 
	The following fact is crucial to us:
	\begin{lemma}\label{LEM_charsabtame}
		Every etale isogeny $S' \rightarrow S$ extends to a tamely ramified cover
		\[
		\Tilde{S'} \rightarrow \overline{S}
		\]
		for a normal crossing compactification $\Tilde{S'}$ of $S'$. In particular, the morphism obtained in Lemma \ref{LEM_FundGroupSurjects} factors to give a morphism
		\[
		\pi_1^t(S_{\overline{k}}, e) \rightarrow \pi_1^{\text{char}}(S).
		\]
	\end{lemma}
	\begin{proof}
		The exact sequence appearing in the definition of the normal crossing compactification is functorial. In particular, we can diagonalize the map between the tori to obtain a map between compactifications. This map is automatically tame, because it can be factored into an abelian part, which is smooth, and a toric part, which of degree coprime to $p$ and etale over $S$.  In particular, the Lang torsor is tamely ramified and we obtain the desired morphism of compactifications.
	\end{proof}
	We prove that the map from the tame fundamental group is an isomorphism for semiabelian varieties $S$ based on a well-known argument.
	\begin{proposition}\label{PROP_fundcompsemiab}
	The above map gives a canonical isomorphism
	\[
	\pi_1^{t}(S_{\overline{k}}, e) \rightarrow \pi_1^{\text{char}}(S).
	\]
	\end{proposition}
	\begin{proof}
		The multiplicativity of the tame fundamental group
		\[
		\pi_1^t(S_{\overline{k}}\times S_{\overline{k}}, e) = \pi_1^t(S_{\overline{k}}, e)\times \pi_1^t(S_{\overline{k}}, e).
		\]
		implies that any tame cover of $S_{\overline{k}}$ can be equipped with a group structure (see \cite[Thm.~5.1]{orgogozoFundamentalGroup}). In particular, every tame cover can be represented by an isogeny. An isogeny can always be dominated by the Lang torsor (see \cite[Ch.~VI,~§1,~6]{Serre_AlgebraicGroupsClassFields}). In particular, the map
		\[
		\pi_1^t(S_{\overline{k}}, e) \rightarrow \pi_1^{\text{char}}(S)
		\]
		is injective by Proposition \ref{PROP_FundamentalGroupExact}. Since it is surjective by Lemma \ref{LEM_FundGroupSurjects}, we obtain the isomorphism. 
	\end{proof}
	The character sheaves considered in \cite{Kr_mer_2015} and \cite{GabberLoeserTore}, are precisely the one-dimensional representations of the tame fundamental group of $G$. This means the character theory presented here agrees with their theory. We note a certain consequence of this theorem (see \cite{virk2014eulerpoincarecharacteristics} and \cite{IllusieEulerPoincare}).
	\begin{proposition}\label{PROP_eulerpoincaresav}
		Let $k'/k$ be an extension and $K \in \derCat{c}{S_{k'}}{\Qbarl}$. We have
		\[
		\chi(S, K) = \chi_?(S, K_\chi).
		\]
		for all $\chi \in \cvar{S}$.
	\end{proposition}
	We also make use of the cohomology of a semiabelian variety.
	\begin{proposition}\label{PROP_SemiAbCohomology}
		Let $d = d_a + d_t$ be the dimension of $S$. The cohomology ring $H^*(S, \Qbarl)$ is given by
		\begin{align*}
		H^n(S, \BZ_\ell) &= \wedge^n H^1(S, \BZ_\ell) \\
		H^1(S, \BZ_\ell) &\cong \BZ_\ell^{2d_a + d_t}.
		\end{align*}
	\end{proposition}
	\begin{proof}
		See \cite[Lem.~15.2]{milneAbelianVarieties}.
	\end{proof}
	\begin{remark}
		The following corollary of Lemma \ref{PROP_SemiAbCohomology} will be used repeatedly. The cohomology complex $H^*(S, \Qbarl)$ satisfies
		\[
		H^n(S, \Qbarl) = \begin{cases}
			\neq 0 & \text{ if } 0 \leq n \leq 2d_a + d_t \\
			0 & \text{ else } 
		\end{cases}
		\]
		Poincare duality implies
		\[
		H^n_c(S, \Qbarl) = \begin{cases}
			\neq 0 & \text{ if } d_t \leq n \leq 2d_a + 2d_t\\
			0 & \text{ else } 
		\end{cases}
		\]
	\end{remark}
	\subsubsection{Thin topology} We go on to define the first topology on the space of characters of a semiabelian variety, the \textit{thin topology}. Such subsets have been considered before, for example in \cite{EsnaultKerzQuasiLinear}, \cite{Kr_mer_2015}, \cite{GabberLoeserTore}, \cite{KowalskiTannaka}. In \cite{EsnaultKerzQuasiLinear}, these subsets are called quasi-linear, in \cite{Kr_mer_2015} they are called thin  and in \cite{GabberLoeserTore} and \cite{KowalskiTannaka} they are called finite unions of translated algebraic cotori. We have chosen the terminology thin because neither finite union of tac nor quasi-linear apply in our setting, especially once we generalize them to arbitrary groups.
	\begin{definition}\label{DEF_charsemiabthin}
		Let $S$ be a semiabelian variety over a finite field $k$. A \textit{translated algebraic cotorus (tac)} is a subset $\Delta \subseteq \cvar{S}$ such that there is a group quotient quotient $\pi\colon S_{\overline{k}}\rightarrow S'$ with geometrically connected fibers and a character $\chi \in \cvar{S}$ with
		\[
		\Delta = \chi\cdot \pi^*(\cvar{S}).
		\]
		We say a subset is \textit{thin} if it is a finite union of tacs. The \textit{thin topology} is the topology on $\cvar{S}$ where a subset is closed if it is thin.
	\end{definition}
	The following follows from an argument similar to the proof of Lemma \ref{LEM_CharactersFindSubgroups}.
	\begin{proposition}\label{PROP_IntersectingTacs}
		Suppose we are given (connected closed) subgroups $S_1, S_2\subseteq S$ with associated quotients $\pi_i\colon S \rightarrow S/S_i$  and characters $\chi_1, \chi_2 \in \cvar{S}$. We define the quotient maps $\pi\colon S \rightarrow S/(S_1 + S_2)$. Put $\Delta_i := \chi_i\cdot \pi_i^*(\cvar{S_i})$. All characters $\chi \in \Delta_1\cap\Delta_2$ satisfy
		\[
		\Delta_1\cap \Delta_2 = \chi\cdot \pi^*(\cvar{S/(S_1 + S_2)}).
		\]
		In particular, $\Delta_1\cap\Delta_2$ is a tac.
	\end{proposition}
	\begin{proposition}
		Let $\Delta_i \subseteq \cvar{S}$ be thin subsets indexed by a set $i \in I$. The subset $\bigcap_{i \in I} \Delta_i$ is a thin subset. In particular, the thin sets form a topology on $\cvar{S}$.
	\end{proposition}
	\begin{proof} We argue by induction on the dimension of $S$. If $S$ has dimension zero, there is nothing to prove. Otherwise, let $i \in I$ be such that $\Delta_i$ is a proper thin subset. The subset $\Delta_i$ is a finite union of tacs, say $\Delta_1, \ldots, \Delta_n$. We can assume $\Delta_i = \Delta_1$. The claim follows from the induction hypothesis and Proposition \ref{PROP_IntersectingTacs}.
\end{proof}
\begin{lemma}\label{LEM_ContinuousSABThin}
	Let $\pi\colon S \rightarrow S'$ a group morphism. The map
	\[
	\pi^*\colon \cvar{S'} \rightarrow \cvar{S}
	\]
	is continuous in the thin topology. In particular, the group $\cvar{S}$ is a topological group with respect to the thin topology.
\end{lemma}
We only prove that the group $\cvar{S}$ is topological.
\begin{proof}
	Note that it is a topological group because the diagonal map $S \rightarrow S\times S$ induces the addition map on character groups and the inversion induces the inversion map. 
\end{proof}

\begin{proposition}
	Let $\pi\colon S \rightarrow S'$ be a surjective map. The induced map
	\[
	\pi^*\colon \cvar{S'} \rightarrow \cvar{S}
	\]
	is closed.
\end{proposition}
\begin{proof} We can factor $\pi$ into an isogeny and a surjective map with connected fibers. We only consider the case when $\pi$ is an isogeny, the other case can be treated by a similar argument.
	
	 Suppose $\pi$ is an isogeny. Let $\chi \in \cvar{S'}$ and $\pi'\colon S'_{\overline{k}} \rightarrow S''$ a surjective group morphism with connected fibers. We can factor $\pi'$ into a square over a finite extension $k'/k$
	 \[
	\begin{tikzcd}
		S_{k'} \arrow[d, "\pi''"] \arrow[r, "\pi"] & S'_{k'} \arrow[d, "\pi'"] \\
		S''' \arrow[r, "\pi'''"]              & S''                 
	\end{tikzcd}
	 \]
	such that $\pi'''$ is an isogeny and $\pi''$ has connected fibers. In particular, the map $(\pi''')^*$ is surjective by Proposition \ref{PROP_CharDescent}. We have
	\[
	\pi^*(\chi\cdot \pi'^*(\cvar{S''})) = \pi^*(\chi)\cdot \pi''^*(\cvar{S'''}).
	\]
	Hence $\pi^*$ is closed.
\end{proof}
\begin{proposition}\label{PROP_TACSIrredThin}
	A subset $\Delta \subseteq\cvar{S}$ is irreducible in the thin topology if and only if it is a tac.
\end{proposition}
	\begin{proof}
	 By trivial bounds on the number of points of a tac, we see that a finite union of proper tacs can not cover a tac (see Lemma \ref{LEM_ClosedDimensionTAC} and Lemma \ref{LEM_ClosedEstimation1}).
	\end{proof}
\begin{proposition}
 The thin topology is Noetherian. In particular, thin subsets admit unique decompositions into tacs.
\end{proposition}
\begin{proof}
	Induction on the dimension of $S$ as before.
\end{proof}

	\subsubsection{Standard topology}\label{SEC_charsemiabstand}
	
	We introduce the \textit{standard topology} The Tannakian categories are defined by localizations along closed subsets. This topology is also well-known and has appeared in many other contexts, expecially when considering Mellin transforms on semiabelian varieties. It is analogous to the Zariski topology on the character variety of a complex torus.
	
	Let $S$ be a semiabelian variety over a finite field $k$. A character $\chi \in \mathscr{C}(S)$ can be factored into a pro-$\ell$ part $\chi_\ell \in \cvar{S}$, whose monodromy is pro-$\ell$, and a coprime to $\ell$-part $\chi_{{\ell'}} \in \cvar{S}$, whose monodromy is coprime to $\ell$, such that
	\[
	\chi = \chi_\ell\chi_{\text{f}}.
	\]
	We call the group of characters with pro-$\ell$ monodromy by $\mathscr{C}^\ell(S)$ and the group of characters with mondoromy coprime to $\ell$ by $\mathscr{C}^{\text{f}}(S)$. We have constructed a decomposition
	\[
	\mathscr{C}(S) = \mathscr{C}^{\text{f}}(S)\times \mathscr{C}^{\ell}(S).
	\]
	This leads to a decomposition
	\[
	\mathscr{C}(S) = \bigsqcup_{\chi \in \mathscr{C}^{\text{f}}(S)} \chi\cdot \mathscr{C}^{\ell}(S).
	\]
	The group of pro-$\ell$-characters is identified with the group (the notation is introduced in \ref{SUBSEC_CharGroup})
	\[
	 \mathscr{C}^{\ell}(S) =  \mathscr{C}(\pi_1^{\ell, \text{char}}).
	\]
	We define the Iwasawa algebra
	\[
	R_{S, \text{int}} := \BZ_\ell[[\pi_1^{\ell, \text{char}}]].
	\]
	We form the tensor product
	\[
	R_S := R_{S, \text{int}}\otimes_{\BZ_\ell} \Qbarl.
	\]
	It is well-known that the closed points of $\text{Spec}(R_S)$ identify with the set $\mathscr{C}(\mathscr{C}(\pi_1^{\ell, \text{char}}))$. We topologize $\mathscr{C}^{\ell}(S)$ by equipping it with the Zariski topology coming from the Iwasawa algebra and the above isomorphism. We then equip $\cvar{S}$ with the coproduct topology. The resulting topology is called the \textit{standard topology}. Note that it has countably many connected components. 
	\begin{remark}	
		A thin subset is closed in the standard topology.
	\end{remark}
	We begin by proving that the maps are continuous. 
	\begin{lemma}\label{LEM_ContinuousSABStandard}
		Let $\pi\colon S \rightarrow S'$ a group morphism. The pullback map
		\[
		\pi^*\colon \mathscr{C}(S') \rightarrow \mathscr{C}(S)
		\]
		is continuous in the standard topology. In particular, the group $\cvar{S}$ is a topological group.
	\end{lemma}
	To evaluate the dimension of the components of a thin subset, we need the following lemma.
	\begin{lemma}\label{LEM_semiabclosedimmersion}
		Let $\pi\colon S \rightarrow S'$ be a surjective morphism. The induced map
		\[
		\pi^*\colon \cvar{S'} \rightarrow \cvar{S}
		\] 
		is a closed immersion.
	\end{lemma}
	\begin{proof}[] This can be identified, componentwise, with a map on spectra induced by a morphism of rings. The morpshim $\pi_1^{\text{char}}(S) \rightarrow \pi_1^{\text{char}}(S')$ is surjective, which implies that the ring morphism $R_S \rightarrow R_{S'}$ is surjective. 	
	\end{proof}
	\begin{proof} The map is induced by a ring morphism of the localized Iwasawa algebra.
	\end{proof}
	\begin{definition}
		The dimension of a closed subset is defined to be the Krull dimension. 
	\end{definition}
	\begin{remark}
		The Krull dimension of a countable union of irreducible subsets is the supremum over the irreducible components. 
	\end{remark}
	\begin{lemma}\label{LEM_IsogPreseversDimensionSAB}
		Let $q\colon S \rightarrow S'$ be an isogeny. A closed subset $\Delta\subseteq \cvar{S'}$ satisfies
		\[
		\text{dim}(q^*(\Delta)) = \text{dim}(\Delta).
		\]  
	\end{lemma}
	\begin{proof}
		The map $q^*$ is a countable disjoint union of maps on spectra induced by a finite morphism of commutative rings. Hence it preserves dimension.
	\end{proof}
	
	\begin{lemma}\label{LEM_ClosedDimensionTAC}
	Let $S$ be a semiabelian variety of dimension $d = d_a + d_t$. The dimension of any irreducible component of $\cvar{S}$ (in the standard topology) is $2d_a + d_t$. In particular, for a quotient $\pi \colon S \rightarrow S'$ and a character $\chi \in \cvar{S}$, the dimension of any irreducible component of $\chi\cdot\pi^*(\cvar{S'})$ is $2d'_a + d'_t$, where $d' = d'_a + d'_t$ is the dimension of $S'$.  
\end{lemma}
\begin{proof}
	Proposition \ref{PROP_SemiAbCohomology} implies
	\begin{align*}
		\pi_1^{\text{char}}(S, e) &= \pi_1^{t}(S_{\overline{k}}, e)\\
		&= \text{Hom}(H^1(S, \BZ_\ell), \BZ_\ell) \\
		&\cong \BZ_\ell^{2d_a + d_t}.
	\end{align*}
	In particular, the ring $R_{S, \text{int}}$ can be identified with a power series ring over $\BZ_\ell$ in $2d_a + d_t$ variables. Hence $R_{S}$ has dimension $2d_a + d_t$ (see \cite[~{}3.1]{GabberLoeserTore}). The second statement follows from Lemma \ref{LEM_semiabclosedimmersion}. 
\end{proof}
The following theorem is used at a crucial point.
\begin{proposition}\label{PROP_irredcompinctacs}
	Let $\Delta \subseteq\cvar{S}$ be a tac and $\Delta' \subset \Delta$ a proper thin subset. Any irreducible component of $\Delta'$ is properly contained in an irreducible component of $\Delta$. 
\end{proposition}
\begin{proof}
	By passing to an irreducible component of $\Delta'$ in the thin topology, we can assume $\Delta'$ is a tac. Note that the dimension of each irreducible components $\Delta'$ or $\Delta$ has the same dimension as $\Delta$ or $\Delta'$ by Lemma \ref{LEM_ClosedDimensionTAC}. Thus any component of $\Delta'$ has strictly lower dimension than any component of $\Delta$.
\end{proof}
\begin{proposition}\label{PROP_irredcomptequal}
	Let $\Delta, \Delta' \subseteq \cvar{S}$ be tacs. Suppose there exists a closed irreducible subset $Z \subseteq\cvar{S}$ such that $Z$ is an irreducible component of $\Delta$ and $\Delta'$. Then $\Delta = \Delta'$.
\end{proposition}
\begin{proof}
	The subset $\Delta\cap\Delta'$ is a tac by Proposition \ref{PROP_IntersectingTacs} and we have $\Delta\cap\Delta' \subseteq \Delta$. The sets $\Delta\cap\Delta'$ and $\Delta$ have an irreducible component in common. Thus Proposition \ref{PROP_irredcompinctacs} implies 
	\[
	\Delta\cap\Delta' = \Delta.
	\]
	This means $\Delta \subseteq \Delta'$. By symmetry, we obtain $\Delta = \Delta'$. 
\end{proof}
	\subsection{Unipotent groups}

	We introduce the Fourier transform on a general unipotent commutative groups $U$. We prove that all characters on a unipotent group are arithmetic and that they are described by the points of an algebraic group, the dual group.
	
	\subsubsection{Characters} The identification of characters of a unipotent group with the points of the dual group rests on the following lemma:
	\begin{lemma}\label{LEM_charuniprop}
		The group $\pi_1^{\text{char}}(U)$ is a pro-$p$ group. In particular, every character $\psi \in \cvar{U}$ has finite order.
	\end{lemma}
	\begin{proof}
		There is $n \geq 1$ such that $p^n$ annihilates $U$.
	\end{proof}
	
	We begin by introducing the dual group.
	\begin{definition}
		The \textit{dual group $\dualUni$ of $U$} is the connected unipotent algebraic group $U^*$ over $k$ (see \cite[]{BegueriDualityCorpsLocal}) such that for all perfect $k$-schemes $S$
		\[
		U^*(S) = \varinjlim_{n} \text{Ext}^1_{}(U\times S, p^{-n}\BZ_p/\BZ_p)
		\]
		The algebraic group $U^*$ is unique up to unique isomorphism and $U \mapsto U^*$ is an involutive, contravariant functor. 
	\end{definition} 
	The following is a reformulation of the construction in \cite{SaibiFourier} in a convenient form.
	\begin{lemma}\label{LEM_UnipotentDual}
	Let $n \geq 1$. We choose a faithful character $\psi_0\colon \BQ_p/\BZ_p \rightarrow \Qbarl^*$.  The character $\psi_0$ induces a functorial isomorphism
	\[
	U^*(k_n) = \widehat{U}(k_n)
	\]
	for all $n \geq 1$.  There exists a lisse, $\ell$-adic sheaf $\SL_{\psi_0}$ on $U\times U^*$ which can be described as follows. For all $\psi \in \widehat{U}(k_n)$ there is a canonical isomorphism
	\[
	\SL_{\psi_0}|_{U\times\{\psi\}} = \SL_\psi.
	\]
	\end{lemma}
	\begin{remark}
		We regard the choice of $\psi_0$ as fixed throughout and surpress it. We write $\SL$ for the universal sheaf constructed in the above Lemma. Lemma \ref{LEM_UnipotentDual} allows us to identify the closed points of $U^*$ with the arithmetic characters of $U$. This justifies writing $\psi$ for a point on $U^*$. Moreover, we denote the dual group by $\widehat{U}$ from now on. 
	\end{remark}
	\subsubsection{Fourier transform} \begin{definition}
		Let $X$ be a scheme over $k$. The \textit{Fourier transform} is defined to be the functor
		\begin{align*}
			\text{FT}_?\colon \derCat{c}{U\times X}{\Qbarl} \rightarrow \derCat{c}{\dualUni\times X}{\Qbarl},~K\mapsto \pi_{\dualUni?}(\pi_{U}^*K\otimes \SL[d]).
		\end{align*}
		We write $\text{FT} := \text{FT}_!$.
	\end{definition}
	We list the basic properties of the Fourier transform (see \cite{SaibiFourier}).
	\begin{theorem}\label{THM_FTBasic}
		Let $U$ be a unipotent group with dimension $d$. Consider $K \in \derCat{c}{U\times X}{\Qbarl}$.
		\begin{enumerate}\item The natural morphism
			\[
			\text{FT}_!(K) \rightarrow \text{FT}_*(K).
			\]
			is an isomorphism. 
			\item We have
			\[
			\text{inv}^*D(\text{FT}(K)) = \text{FT}(\text{inv}^*D(K))(d).
 			\]
 			\item We have
 			\[
 			K = \text{inv}^*(\text{FT}(\text{FT}(M)))(-d).
 			\]
 			\item The Fourier transform is a $t$-exact equivalence of categories (with respect to the perverse $t$-structure).
 			\item For each closed point $\psi \in \widehat{U}$, we have
 			\[
 			m_\psi^*(\text{FT}(K)) = \text{FT}(K_\psi).
 			\]
 			where $m_\psi(x) := x + \psi$ is the additive shift.
 		\end{enumerate}
	\end{theorem}
	The Fourier transform is functorial:
	\begin{proposition}\label{PROP_ftfuncgroup}
		Let $U$ and $U'$ be unipotent groups of dimension $d$ and $d'$ respectively. Let $\pi\colon U \rightarrow U'$ be a group morphism. We consider $K \in \derCat{c}{U\times X}{\Qbarl}$.  We have
			\begin{align*}
			\text{FT}(\pi_!(K)) &= \widehat{\pi}^*(\text{FT}(K))[d - d']\\
			\text{FT}(\pi_*(K)) &= \widehat{\pi}^!(\text{FT}(K))(d' - d)[d' - d]
			\end{align*}
	\end{proposition}
	If we apply this theorem to the multiplication map, then we obtain the following corollary (see Section \ref{SUBSEC_Convolution})
	\begin{corollary}\label{COR_ftconv}
		Let $K, L \in \derCat{c}{U}{\Qbarl}.$ We have isomorphisms 
		\begin{align*}
			\text{FT}(K*_!L) &= \text{FT}(K)\otimes\text{FT}(L)[-d] \\
			\text{FT}(K*_*L) &= \text{Hom}(D(\text{FT}(K)), \text{FT}(L))(d)[d]
		\end{align*}
	\end{corollary}
	We also require the functoriality of the Fourier transform in the base space. 
	\begin{proposition}\label{PROP_ftfuncbase}
	Let $\pi\colon X \rightarrow Y$ be a morphism. Let $K \in \derCat{c}{U\times X}{\Qbarl}$ and $F$ a lisse sheaf on $X$. We have
		\begin{align*}
			\text{FT}(\pi_*(K)) &= \pi_*(\text{FT}(K)) \\
			\text{FT}(\pi_!(K)) &= \pi_!(\text{FT}(K)) \\
			\text{FT}(K\otimes\pi_{X}^*(F)) &=\text{FT}(K)\otimes \pi_{X}^*F.
		\end{align*}
	\end{proposition}
	\begin{proof}
		The third formula can be verified from the projection formula using the standard argument for proving identities involving Fourier transforms.
	\end{proof}
	We describe the character variety of a unipotent group in terms of its dual.
	\begin{lemma}
		There is a bijection
		\[
		U(\overline{k}) = \cvar{U}.
		\]
		which only depends on our implicit choice of character $\psi_0$. This bijection is functorial in $U$. 
	\end{lemma}
	\begin{proof}
		We have proven that all characters are arithmetic. Thus 
		\[
		\cvar{U} = \varinjlim_{n} \widehat{U}(k_n) = \widehat{U}(\overline{k}).
		\]
	\end{proof}
	\subsubsection{Topology}
	\begin{definition}\label{DEF_charuniclosed}
		We say a set $Z \subseteq \cvar{U}$ is closed in the thin and the standard topology if and only if it is the set of closed points of a closed subset $Z' \subseteq \widehat{U}_{\overline{k}}$. We define the \textit{dimension} of a closed subset to be its Krull dimension.
	\end{definition}
	\begin{remark}
		The dimension of a closed subset $Z \subseteq \cvar{U}$ agrees with the Krull dimension of the associated closed subsets $Z' \subseteq\widehat{U}$.
	\end{remark}
	\begin{lemma}\label{LEM_ContinuousUnipotent}
		Let $\pi\colon U \rightarrow U'$ be a group morphism. The induced map
		\[
		\pi^*\colon \cvar{U'} \rightarrow \cvar{U}
		\]
		is a continuous. In particular, $\cvar{U}$ is a topological group. 
	\end{lemma}
	\begin{proof}
		The morphism $\pi^*$ can be described by a morphism of spectra induced by a morphism of rings. 
	\end{proof}
	\begin{lemma}
		Let $\pi\colon U \rightarrow U'$ be a surjective morphism. The map
		\[
		\pi^*\colon \cvar{U'} \rightarrow \cvar{U}
		\]
		is closed.
	\end{lemma}
	\begin{proof}
		We can factor the morphism into an isogeny and a surjective morphism with closed fibers. For an isogeny, the map $\pi^*\colon \widehat{U'} \rightarrow\widehat{U}$ is surjective, hence flat. Thus it is open. An open morphism of topological groups is closed. 
	\end{proof}
	\begin{lemma}\label{LEM_TopologyIsogenyDimensionUNI}
		Let $q\colon U \rightarrow U'$ be an isogeny and $\Delta\subseteq\cvar{U'}$ a closed subset. We have
		\[
		\text{dim}(q^*(\Delta)) = \text{dim}(\Delta).
		\]
	\end{lemma}
	\begin{proof}
		The morphism of schemes $q^*\colon \widehat{U'} \rightarrow\widehat{U}$ is finite.
	\end{proof}
	
	\subsection{Character variety}\label{charvar_SUBSEC} The goal of this section is to introduce the \textit{thin} and the \textit{standard} topology on the character variety. The standard topology behaves much like the topology of a Noetherian scheme, with the exception that there can be countably many irreducible components. The dimension can be defined in terms of the standard topology. 
	
	\subsubsection{Topology on Products} We first construct the standard and the thin topology for certain products. 
	\begin{definition}
		Let $G = S\times U$ be a connected commutative algebraic group over a finite field $k$. The \textit{thin} and the \textit{standard} topology in $\cvar{G}$ are given by the product topologies along the isomorphism
		\[
		\cvar{G} = \cvar{S}\times\cvar{U}.
		\]
		We define the \textit{dimension} of a closed subset $Z \subseteq \cvar{G}$ to be the Krull dimension in the standard topology. Unless stated otherwise, a closed subset always means a closed set in the standard topology.
	\end{definition}
	This definition has as an immediate advantage.
	\begin{lemma}
		Let $G = S\times U$ and $G' = S'\times U'$ be connected commutative algebraic groups over a finite field $k$. Let $\pi\colon G \rightarrow G'$ be a group morphism. The induced map
		\[
		\pi^*\colon \cvar{G'} \rightarrow \cvar{G} 
		\]
		is a continuous map in the thin and the standard topology.
	\end{lemma}
	\begin{proof}
		The group morphism $\pi$ can be written as a product of maps $S \rightarrow S'$ and $U \rightarrow U'$ up to a radiciel morphism $U\rightarrow S'$. By \ref{PROP_DescentFundamentalGroup}, we can assume the morphism is diagonal. Lemma \ref{LEM_ContinuousSABStandard}, Lemma \ref{LEM_ContinuousSABThin}, and Lemma \ref{LEM_ContinuousUnipotent} say the compoonents of this map are continuous.
	\end{proof}
	
	\subsubsection{Topology} We can describe the connected components of the standard topology.
	\begin{lemma}\label{LEM_TopologyIrred}
		Let $G = S\times U$. The character variety splits, in the standard topology, into a disjoint union
		\[
		\cvar{G} = \bigsqcup_{\chi \in \mathscr{C}^{\text{fin}(S)}} \chi\cdot (\mathscr{C}^\ell(S)\times\cvar{U}).
		\]
		The product $\mathscr{C}^\ell(S)\times\cvar{U}$ is a Noetherian topological space. A closed subset $\Delta \subseteq \cvar{G}$ is irreducible if and only if there is a character $\chi \in \mathscr{C}^f(S)$, a closed, irreducible subset $\Delta_S \subseteq \mathscr{C}^\ell(S)$ and an irreducible subset $Z_U \subseteq \cvar{U}$ such that
		\[
		\Delta = \chi\cdot(\Delta_S\times Z_U).
		\]
		In particular, the dimension of $Z$ is
		\[
		\text{dim}(Z) = \text{dim}(\Delta_S)  + \text{dim}(Z_U).
		\]
	\end{lemma}
	\begin{proof}
		The product of two irreducible Noetherian topological spaces is irreducible and Noetherian.
	\end{proof}
	We can now describe all closed subsets as countable unions of irreducible subsets as in Lemma \ref{LEM_TopologyIrred} and then determine the dimension as a maximum over all irreducible subset. Note that the decomposition in this Lemma is unique, hence we obtain irreducible components of character sets. 
	\begin{proposition}
		Let $G = S\times U$. The character variety $\cvar{G}$ in the standard topology is a countable disjoint union of irreducible Noetherian topological spaces. In particular, each closed subset admits a unique irreducible decomposition. 
	\end{proposition}
	
	For the thin topology, we obtain similar statements.
	\begin{proposition}\label{PROP_TopologyProductStructure}
		Let $G = S\times U$. The character variety $\cvar{G}$ is an irreducible Noetherian space in the thin topology. In particular, each thin closed subset $\Delta \subseteq \cvar{G}$ can be uniquely written as finite a union of sets of the form $\Delta_S\times Z_U$, where $\Delta_S \subseteq \cvar{S}$ is a tac and $Z \subseteq \cvar{U}$ is the set of closed points of an irreducible closed subset of $\widehat{U}$. 
	\end{proposition}
	\begin{proof} The product of two Noetherian spaces is Noetherian by CITE. It is irreducible by CITE.	
	\end{proof}
	\begin{remark}
		The Krull dimension of a countable union of irreducible subsets is the supremum over the irreducible components. 
	\end{remark}
	We can estimate the number of points of a thin subset based on its dimension.
	\begin{lemma}\label{LEM_ClosedEstimation1}
		Let $G = S\times U$ be a connected commutative algebraic group over a finite field $k$. Let $\Delta \subset \cvar{G}$ be a thin subset. Suppose $\Delta$ has codimension $d$. Then
		\[
		\frac{|\Delta\cap\widehat{G}(k_n)|}{|\widehat{G}(k_n)|} \ll_\Delta |k|^{-nd}.
		\]
	\end{lemma}
	\begin{proof}
		A thin subset is a finite union of sets of the form $Z\times \Delta$ for a tac $\Delta\subseteq\mathscr{C}(S)$ and a closed subset $Z \subseteq \widehat{U}$. For $Z$, we can appeal to trivial bounds based on the dimension. For $\Delta$, we can again appeal to trivial bounds and the computation of the dimension given by Lemma \ref{LEM_ClosedDimensionTAC}.
	\end{proof}
	
	\subsubsection{Isogenies}  We record the following crucial corollary.
	\begin{corollary}\label{COR_TopologyIsogenyPre}
		Let $G = S\times U$ be a connected commutative algebraic group over a finite field $k$. For an isogeny $q\colon G\rightarrow G'$, the induced map
		\[
		q^*\colon \mathscr{C}(G') \rightarrow \mathscr{C}(G)
		\]
		is a closed map in the thin and the standard topology.
	\end{corollary}
	\begin{proof}
		The isogeny is a product of isogenies. The irreducible subsets in the closed and the thin topology are therefore mapped to closed subsets. The irreducible decompositions constructed above imply that all closed sets are mapped to closed sets.
	\end{proof}
	\subsubsection{General case} We define the topology in the general case.
	\begin{definition}\label{DEF_topology}
		Let $G$ be a connected commutative algebraic group. There is an isogeny $q\colon G \rightarrow S\times U$. We define the standard topology and the thin topology on $\cvar{G}$ by declaring the map
		\[
		q^*\colon \cvar{G'} \rightarrow \cvar{G}
		\]
		to be a quotient map. The \textit{dimension} of a closed subset $\Delta \subseteq \cvar{G}$ is the Krull dimension in the standard topology.
	\end{definition}
	\begin{lemma}
		This definition does not depend on $q$.
	\end{lemma}
	\begin{proof}
		Given two such isogenies, we can always find a third isogeny to a product dominated by the two. Thus it is sufficient if we prove that an isogeny between algebraic groups of the form $S\times U$ induces a quotient map. We have proven this in Corollary
	\end{proof}
	The definition implies the favourable properties we established for products:
	\begin{proposition}\label{PROP_stdtop}
		The topological space $\cvar{G}$ is a countable union of irreducible Noetherian spaces in the standard topology. 
	\end{proposition}
	\begin{proposition}\label{PROP_thintop}
		The topological space $\cvar{G}$ is an irreducible Noetherian space in the thin topology.
	\end{proposition}
	\begin{proposition}\label{PROP_CharIsogenyClosed}
		Let $\pi\colon G \rightarrow G'$ be a surjective morphism. The map
		\[
		\pi^*\colon \cvar{G'} \rightarrow \cvar{G}
		\]
		is closed in the thin and the standard topology.
	\end{proposition}
	\begin{proof}
		Up to isogeny, we can write the morphism as a product.
	\end{proof}
	\begin{proposition}\label{PROP_IsogenyPreservesDimension}
		Let $q\colon G \rightarrow G'$ be an isogeny. The morphism
		\[
		q^*\colon \cvar{G'} \rightarrow \cvar{G}
		\]
		preserves dimension, i.e. we have 
		\[
		\text{dim}(q^*(\Delta)) = \text{dim}(\Delta)
		\]
		for all closed subsets $\Delta\subseteq \cvar{G'}$.
	\end{proposition}
	\begin{proof}
		We have
		\[
		\text{dim}(q^*(\Delta)) \leq \text{dim}(\Delta).
		\]
		Up to isogeny, we can write the morphism as a diagonal morphism of products of the form $S\times U$. We can write the dimension component wise by Lemma \ref{LEM_TopologyIrred}. Lemma \ref{LEM_IsogPreseversDimensionSAB} and Lemma \ref{LEM_TopologyIsogenyDimensionUNI} imply the required equalities.
	\end{proof}
	We prove quantitative bounds on the number of rational points in a thin subset based on its dimension.
	\begin{proposition} \label{PROP_thinsubsetsestimate}
		Let $G$ be a connected commutative algebraic group over a finite field $k$ and $\Delta$ a thin subset of codimension $d$. We have
		\[
		\frac{|\Delta\cap\widehat{G}(k_n)|}{|\widehat{G}(k_n)|} \ll_\Delta |k|^{-nd}.
		\]
	\end{proposition}
	\begin{proof}
		We have an isogeny $q\colon S\times U \rightarrow G$ with kernel $\Gamma$. We put 
		\[
		\Delta' = q^*(\Delta).
		\]
		The map $q^*$ is a closed map in the thin topology, hence $\Delta'$ is a thin subset. Moreover, the dimension of $q^*(\Delta)$ agrees with the dimension of $\Delta$ by Proposition \ref{PROP_IsogenyPreservesDimension}. We have the map
		\[
		q^*\colon \Delta \cap \widehat{G}(k_n) \rightarrow q^*(\Delta)\cap \widehat{S\times U}(k_n). 
		\]
		Note that this map has fibers of order at most $|\Gamma|$ by Corollary \ref{COR_ExactSequencePoints}. Thus the estimate follows from Lemma \ref{LEM_ClosedEstimation1}.
	\end{proof}
	
	\subsection{Fourier-Mellin transforms} During the course of our arguments, it will be convenient to have a shorthand notation for certain functors which appear repeatedly. For this reason, we introduce \textit{Fourier-Mellin transforms.} 
	\begin{definition}
		We consider a map $\pi\colon G \rightarrow U$ whose restriction to the maximal connected unipotent subgroup of $G$ induces an isomorphism. This map is unique, because it is the quotient by the maximal semiabelian subgroup. For each $\chi \in \cvar{G}$, we define the \textit{Fourier-Mellin transform at $\chi$ with or without supports} to be the functor
		\[
		\text{FM}_{\chi, ?} := \text{FT}(\pi_?(K_\chi))
		\]
		for each $K \in \derCat{c}{G}{\Qbarl}$. Let $\Delta \subseteq \cvar{G}$ be a closed subset. We define the subset
		\[
		\dualUni_{\chi, \Delta} := (i^*)^{-1}(\Delta^c).
		\]
		This is an open subset inside $\cvar{U} = \dualUni(\overline{k})$. Then we define the \textit{restricted Fourier-Mellin transform with or without supports} to be the functor
		\[
		\text{FM}_{\chi, ?} := \text{FT}(\pi_?(K_\chi))|_{\dualUni_{\chi, \Delta}}.
		\]
	\end{definition}
	\begin{remark}
		In most cases, we surpress $\pi$ and do not mention it. Almost all statements have no dependency on the particular choice of $\pi$.
	\end{remark}
	We record variuos formulas for the Fourier-Mellin transform that are used repeatedly. For each complex, the Fourier-Mellin transform is the Fourier transform of a relative  Mellin coefficient. For a product, we can write the Fourier-Mellin transform as a relative Mellin coefficient {of} a Fourier transform. The functors $\text{FM}_{\chi, ?}$ are the relative Mellin-coefficient evaluated at a character $\chi \in \cvar{S}$ of one complex. 
	\begin{lemma}\label{LEM_fmprojformula}
		Let $G = S\times U$ be a connected commutative algebraic group over a finite field $k$ and $\chi \in \cvar{G}$. We can write
		\[
		\text{FM}_{\chi, \Delta, ?}(K) = \pi_{\widehat{U}?}(\text{FT}(K)_\chi)|_{\dualUni_{\chi, \Delta}}
		\]
		for all $K \in \derCat{c}{G}{\Qbarl}$.
	\end{lemma} 
	\begin{proof}
		This follows from the formulas noted in Proposition \ref{PROP_ftfuncbase} and Proposition \ref{PROP_ftfuncgroup}.
	\end{proof}
	We also note the behaviour under duality.
	\begin{lemma}\label{LEM_fmdual}
		We have
		\begin{align*}
		\text{inv}^*D(\FMSH{K}) &= \FMST{\text{inv}^*D(K)} \\
		\text{inv}^*D(\FMST{K})	&= \FMSH{\text{inv}^*D(K)}.
		\end{align*}
		for all $K \in \derCat{c}{G}{}$ and $\chi \in \cvar{G}$. 
	\end{lemma}
	\begin{proof}
		The dual of $\SL_\chi$ is $\SL_{\chi^{-1}}$ and $\text{inv}^*(\SL_\chi) = \SL_{\chi^{-1}}$ by Proposition \ref{PROP_CharFunc}. We can combine these computations with Theorem \ref{THM_FTBasic} to obtain the above formulae.
	\end{proof}
	We also note the behaviour under proper smooth pushforward.
	\begin{lemma}\label{LEM_fmproperpush}
		Let $\pi\colon G \rightarrow G'$ be a proper smooth morphism and $K \in \derCat{c}{G}{\Qbarl}$. We have
		\[
		\text{FM}_{\chi, ?}(\pi_?(K)) = \text{FM}_{\pi^*(\chi), ?}(K)
		\]
		for all $\chi \in \cvar{G'}$. For a closed subset $\Delta\subseteq\cvar{G}$ and a character $\chi \in \cvar{G'}$ with $\pi^*(\chi) \notin \Delta$, we have
		\[
		\text{FM}_{\pi^*(\chi), \Delta, ?}(\pi_{?}(K)) = \text{FM}_{\chi, (\pi^*)^{-1}(\Delta)}(K).
		\]
	\end{lemma}
	\begin{proof}
		Let $\pi'\colon G' \rightarrow U$ with connected kernel which induces an isogeny from the maximal connected unipotent subgroup of $G'$ onto $U$. Then $\pi\pi'\colon G \rightarrow U$ also satisfies this property. We obtain the formulae from Proposition \ref{PROP_ProjectionFormula2}.
	\end{proof}
	We can also compute the stalks of the Fourier-Mellin transforms.
	\begin{lemma}\label{LEM_fmstalks}
		Let $K \in \derCat{c}{G}{\Qbarl}$, $\chi \in \cvar{G}$, and $\psi \in \widehat{U}$ a closed point. We have
		\begin{align*}
			i_\psi^*(\FMSH{K}) &= H^*_c(G, M_{\chi\pi^*(\psi)})[d]\\
			i_\psi^!(\FMST{K}) &= H^*(G, M_{\chi\pi^*(\psi)})(-d)[-d]
		\end{align*}
	\end{lemma}
	\begin{proof}
		This follows from the shift formula in Theorem \ref{THM_FTBasic} and the formula for pushforwards in Proposition \ref{PROP_ftfuncbase}.
	\end{proof}

%% file: VanishingTheorems.tex
	\section{Vanishing theorems}
	\subsection{Definition of $\Delta$-unramified} The goal of this section is to introduce the notion of a $\Delta$-unramified sheaf for a closed subset $\Delta \subseteq \mathscr{C}(G)$. One of the core insights of this article is that one can give a convenient definition of an unramified character for a perverse sheaf. 
	
	Proposition \ref{PROP_GysinUnramifiedCriterion} will give a more concrete and usable characterization than the definition just below. Proposition \ref{PROP_GysinUnramifiedCriterion} shows that the unramified characters defined here are precisely those that can give fiber functors on Tannakian categories.
	\begin{definition}\label{DEF_unramifiedhcar}
		 Let $\pi\colon G \rightarrow U$ be a morphism, whose restriction to the maximal unipotent connected subgroup of $G$ induces an isogeny. Let $M \in \Perv{}{G}$. We say $M$ is \textit{unramified} at a character $\chi \in \mathscr{C}(G)$ if:
		\begin{enumerate}[]\item The natural map
			\[
			\text{FT}(\pi_{!}(M_\chi)) \rightarrow \text{FT}(\pi_{*}(M_\chi))
			\]
			is an isomorphism at $0$.
			\item The complexes $\text{FT}(\pi_{?}({M}_\chi))$ are perverse in a neighborhood of zero.
			\item The complexes $\text{FT}(\pi_{?}(M_\chi))$ are lisse in a neighborhood of zero.
		\end{enumerate}
		Let $\Delta \subseteq \mathscr{C}(G)$ be a subset of $\mathscr{C}(G)$. We say $M$ is \textit{$\Delta$-unramified} if all $\chi \notin \Delta$ are unramified for $M$.
	\end{definition}
	We should get the following right out of the way:
	\begin{lemma}
		The above definition does not depend on $\pi$.
	\end{lemma}
	\begin{proof}
		If we are given two maps $\pi_i\colon G \rightarrow U_i$ for $i = 1, 2$ as in the above definition, we can find a unipotent group $U$ with isogenies $q_i\colon U_i \rightarrow U$ such that $q_1\pi_1 = q_2\pi_2.$
		 For any perverse sheaf $M \in \Perv{}{G}$ and character $\chi \in \cvar{G}$ we have a functorial isomorphisms
		\[
		\text{FT}((q_i\pi_1)_?(M_\chi)) = q_i^{*}\text{FT}((\pi_i)_?(M_\chi))
		\]
		by Proposition \ref{PROP_ftfuncgroup}. Therefore, the sheaf $M$ does not ramify at $\chi$ \quotationMark{with respect to} $\pi_1$ if and only if it does not ramify with respect to $q_1\pi_1$ if and only if it does not ramifiy with respect to $\pi_2$.
	\end{proof}
	\begin{definition}
		Let $M \in \Perv{}{G}$. The sheaf $M$ is called \textit{weakly unramified} at a character $\chi \in \cvar{G}$ if the following two conditions are satisfied.
		\begin{enumerate}[]\item The natural map
			\[
			H^*_c(G, M_\chi) \rightarrow H^*(G, M)
			\]
			is an isomorphism.
			\item The complexes $H^*_?(G, M_\chi)$ are concentrated in degree zero.
		\end{enumerate}
	\end{definition}
	The following two examples attempt to highlight the difference and the similarities between the definition of an unramified character and the definition of a weakly unramified character. Once we have developed more theory, we also discuss Example \ref{EXP_GysinWUnrVSUnr}, which discusses the case $G = \BA^1$ in detail. 
	\begin{example}
		Let $U$ be a unipotent group over a finite field and $M \in \Perv{}{U}$ a perverse sheaf on $U$. A character $\psi \in \mathscr{C}(U)$ is unramified for $M$ if and only if $\text{FT}(M)$ is lisse at $\psi$.
	\end{example}
	\begin{example}
		Let $S$ be a semiabelian group over a finite field $k$ and $M \in \Perv{}{S}$ a perverse sheaf. The sheaf $M$ is unramified at a character $\chi \in \mathscr{C}(S)$ if and only if it is weakly unramified, i.e. the forget supports map
		\[
		H^*_c(S, M_\chi) \rightarrow H^*(S, M_\chi)
		\]
		is an isomomorphism, and $H^*_?(S, M_\chi)$ are concentrated in degree zero.
	\end{example}
	
	The above lemma has as as immediate consequence the following proposition, which will turn out to be crucial when we want to reduce the vanishing theorems to a product.
	\begin{proposition}\label{PROP_RamifiedFiniteDescent}
	 Let $q\colon G\rightarrow G'$ be an isogeny. Let $M \in \Perv{}{G}$ be a perverse sheaf on $G$. A character $\chi \in \mathscr{C}(G')$ is unramified for $q_*M$ if and only if $q^*(\chi)$ is unramified for $M$. 
	\end{proposition}
	\begin{proof}
		Let $\pi \colon G' \rightarrow U$ be a surjective morphism to a unipotent group which induces an isogeny from the maximal unipotent connected subgroup of $G'$ onto $U$. Then $q\pi \colon G \rightarrow U$ also satisfies this condition. The projection formula for lisse sheaves gives an isomorphism
		\[
		\text{FT}((\pi q)_?(M_{\pi^*\chi})) = \text{FT}(\pi_?((q_*M)_{\chi})).
		\]
	\end{proof}
	We can give a criterion in terms of Fourier-Mellin transforms.
	\begin{proposition}\label{PROP_fmcritunr}
		Let $M \in \Perv{}{G}$ and $\Delta\subseteq\cvar{G}$ a closed subset. Then $M$ is $\Delta$-unramified if and only if
		\begin{enumerate}[]\item The natural map
			\[
			\FMSHunr{M} \rightarrow \FMSTunr{M}
			\]
			is an isomorphism for all $\chi \notin \Delta$.
			\item The complexes $\FMQunr{M}$ are lisse for all $\chi \notin \Delta$.
			\item The complexes $\FMQunr{M}$ are perverse for all $\chi \notin \Delta$.
		\end{enumerate}
	\end{proposition}
	\begin{proof}
		This follows from the formula for shifting a Fourier transform  (See Theorem \ref{THM_FTBasic}).
	\end{proof}
	\subsection{Canonical map and Fourier transform} Currently, it is not even clear why the above definition implies that an unramified character is weakly unramified. For this purpose, we introduce the canonical map. Let $x \in X$ be a point. We denote the inclusion of the point $x$ by $i_x\colon x \rightarrow X$. For each complex $K \in \derCat{c}{X}{\Qbarl}$, we introduce the \textit{canonical map}
	\[
	i^*_xK\otimes i_x^!\Qbarl \rightarrow i_x^!K.
	\]
	The canonical map, applied to a Fourier transform, is the forget supports map. Moreover, the canonical map is an isomorphism for lisse complexes. These two facts are what make it so useful for us. The following proposition describes the canonical map applied to a Fourier transform: 
	\begin{proposition}\label{PROP_GysinUnipotent}
		Let $U$ be a unipotent group of dimension $d$, $K \in \derCat{c}{U}{\Qbarl}$, and $\psi \in \widehat{U}$. The canonical map
		\[
		i_\psi^*\text{FT}(K)\otimes i_\psi^!\Qbarl \rightarrow i_\psi^!\text{FT}(K)
		\]
		is canonically isomorphic to the forget supports map
		\[
		H^*_c(U, K_\psi)(-d)[-d] \rightarrow H^*(U, K_\psi)(-d)[-d]
		\]
		for each closed point $\psi \in \widehat{U}.$
	\end{proposition}
	Before we begin proving this proposition, let us look at an example where we can compute the canonical map.
	\begin{example}\label{EXP_GysinWUnrVSUnr}
		Consider $U = \BA^1$ over a finite field $k$. Let $M$ be a perverse sheaf on $U$ such that the Fourier transform $\text{FT}(M)$ has no punctual sections. Put $$F := H^{-1}(\text{FT}(M))_{\overline{\eta}}.$$ Consider a character $\psi \in \widehat{U}$. By Proposition \ref{PROP_GysinUnipotent}, the forget supports map
		\[
		H^*_c(U, M_\psi) \rightarrow H^*(U,M_\psi)
		\]
		is equal to a shifted twist of the canonical map. Because $\BA^1$ is smooth of dimension 1 and $\text{FT}(M)$ has no punctual sections, we see that the canonical map identifies with the natural map
		\[
		F^I(-1)[-1] \rightarrow F_I(-1)[-1]
		\] 
		from the invariants of $F$ at $\psi$ to the coinvariants. We see that the condition of being weakly unramified is strictly weaker than the condition of being unramified. For $\BA^1$, being weakly unramified corresponds to the Gysin map being an isomorphism. This is strictly weaker than the condition of being unramified, which requires the sheaf to be lisse.
	\end{example}
	The proof of the above Proposition is based on the following diagram. 
	\begin{proposition}\label{PROP_GysinDiagram}
	Let $\pi\colon X \rightarrow Y$ be a smooth morphism of smooth schemes. We consider $K \in \derCat{c}{X}{\Qbarl}$ and $y \in Y$. We have the sequence of morphisms
	\begin{align*}
		H^*_c(X_y, i_y^* (K) \otimes i_y^!\Qbarl)  &= i_y^*\pi_!(K)\otimes i_y^!\Qbarl \\ &\xrightarrow{(1)} i_y^!\pi_!(K)
		\\ &\xrightarrow{(2)} i_y^!\pi_*(K)
		\\ &=   \pi_*i_y^! (K)
		\\ &=   H^*(X_y, i_y^! K)
	\end{align*}
	This morphism coincides with the natural forget support map composed with the canonical map
	\[
	H^*_c(X_y, i_{y}^*(K)\otimes i_y^! \Qbarl) \rightarrow  H^*(X_y, i_y^!(K) ).
	\]
	The map $(1)$ is the canonical map applied to $\pi_!K$. The map $(2)$ is the natural forget supports map.
\end{proposition}
This diagram is verified in the appendix, see Section \ref{SUBSEC_CanonicalMap}.
	We evaluate the required term via the following proposition.	
	\begin{proposition}\label{PROP_GysinIsomorphismULA}
		Let $\pi\colon X \rightarrow Y$ be a morphism. Let $K \in \derCat{c}{Y}{\Qbarl}$ be a complex on $X$ which is universally locally acyclic over $Y$ and $L \in \derCat{c}{X}{\Qbarl}$. The canonical map
		\[
		i_y^*(K)\otimes i_y^!(L) \rightarrow i_y^!(K\otimes L)
		\]
		is an isomorphism for all $y \in Y$. 
	\end{proposition}
	\begin{proof}
		This is proven in \cite[Prop.~2.3]{Lu2022}.
	\end{proof}
	The following corollary could be proven by noting that the formation of the canonical map commutes with etale base change. It also follows from the Proposition \ref{PROP_GysinIsomorphismULA}.
	\begin{corollary}\label{COR_gysinisolisse}
		Let $K \in \derCat{c}{X}{\Qbarl}$ be lisse and $x \in X$. The canonical map
		\[
		i_x^*(K)\otimes i^!_x\Qbarl \rightarrow i_x^!(K)
		\]
		is an isomorphism. 
	\end{corollary}
	\begin{proof}
		A complex $K$ is universally locally acyclic with respect to the identity if and only if it is lisse. We can apply Proposition \ref{PROP_GysinIsomorphismULA} to obtain the claim.
	\end{proof}
	\begin{proof}[Proof (Prop. \ref{PROP_GysinUnipotent}).]
		By Proposition \ref{PROP_ftfuncgroup}, we can replace $U$ by $U^{\text{red}}$. Hence we can assume $U$ is smooth. The Fourier transform is given by
		\[
		\text{FT}(K) = \pi_{\widehat{U}!} (\pi_{U}^* (K)\otimes \SL[d]).
		\]
		We apply Proposition \ref{PROP_GysinDiagram} to compute the canonical map for the Fourier transform at a fixed closed point $\psi \in \widehat{U}$ with $\pi = \pi_{\widehat{U}}$. Note that the map
		\[
		\pi_{\widehat{U}!} (\pi_{U}^* (K)\otimes \SL[d])\rightarrow  \pi_{\widehat{U}*} (\pi_{U}^* (K)\otimes \SL[d])
		\]
		is an isomorphism by Theorem \ref{THM_FTBasic}. Hence the map (2) in Proposition \ref{PROP_GysinDiagram} is an isomorphism. Thus we obtain an isomorphism of the canonical map applied to the Fourier transform with the composite
		\[
		H^*_c(U, i_\psi^*(\pi_{U}^* (K)\otimes \SL[d])\otimes i_\psi^!\Qbarl[d]) \rightarrow H^*(U, i_\psi^!(\pi_{U}^* (K)\otimes \SL[d]))
		\]
		of the forget supports map and the canonical map applied to the complex
		\[
		\pi_{U}^* (K)\otimes \SL[d].
		\]
		We prove that the canonical map applied to this complex is an isomorphism. Note that $\pi_U^* K$ is universally locally acyclic over $\widehat{U}$ by \cite[Cor.~9.3.4]{LeiFu}. Thus $\pi_U^*K\otimes\SL$ is universally locally acyclic. Proposition \ref{PROP_GysinIsomorphismULA} shows that the canonical map
		\[
		i_\psi^*(\pi_U^*(K)\otimes \SL)\otimes i_\psi^!\Qbarl \rightarrow i_\psi^!(\pi_U^*(K)\otimes \SL)
		\]
		is an isomorphism. Therefore, the canonical map applied to the Fourier transform is isomorphic to the forget supports map
		\[
		H^*_c(U, i_\psi^*(K)\otimes i_\psi^!\Qbarl[d]) \rightarrow H^*(U, i_\psi^*(K)\otimes i_\psi^!\Qbarl[d]).
		\]
		We evaluate $i_\psi^!(\Qbarl)$ from the smoothness of $U$. We have
		\[
		i_\psi^!\Qbarl = D(i_\psi^*D(\Qbarl)) = \Qbarl(-d)[-2d] 
		\]
		because $\widehat{U}$ is smooth of dimension $d$. Therefore, the canonical map applied to the Fourier transform is isomorphic to
		\[
		H^*_c(U, K_\psi)(-d)[-d] \rightarrow H^*(U, K_\psi)(-d)[-d].
		\]
	\end{proof}

	This allows us to evaluate the canonical map on the Fourier-Mellin transforms.
	\begin{proposition}\label{PROP_GysinGeneralGroup}
		Let $G$ be a group, whose maximal connected unipotent subgroup has dimension $d$. Let $K \in \derCat{c}{G}{\Qbarl}$ be a complex. The canonical map
		\[ 
		i_0^*\FMSHunr{K}\otimes i_0^!\Qbarl \rightarrow i_0^!\FMSTunr{K}
		\]
		is canonically isomorphic to the natural map
		\[
		H^*_c(G, K_\chi)(-d)[-d] \rightarrow H^*(G, K_\chi)(-d)[-d].
		\]
	\end{proposition}
	\begin{proof}		
		The Fourier-Mellin transform is by definition
		\[
		\FMQunr{K} = \text{FT}(\pi_{?}(K_\chi))
		\]
		for $\chi \in \cvar{G}$. Therefore, Proposition \ref{PROP_GysinUnipotent} identifies the diagram
		\[
		\begin{tikzcd}
			i_0^*\FMSHunr{K}\otimes i_0^!\Qbarl \arrow[r] \arrow[d] & i_0^!\FMSHunr{K} \arrow[d] \\
			i_0^*\FMSTunr{K}\otimes i_0^!\Qbarl \arrow[r]           & i_0^!\FMSTunr{K}          
		\end{tikzcd}
		\]
		with the diagram
		\[
		\begin{tikzcd}
			{H^*_c(U, \pi_!(K_\chi))(-d)[-d]} \arrow[r] \arrow[d] \arrow[r] & {H^*(U, \pi_!(K_\chi))(-d)[-d]} \arrow[d] \\
			{H^*_c(U, \pi_*(K_\chi))(-d)[-d]} \arrow[r]                       & {H^*(U, \pi_*(K_\chi))(-d)[-d]}          
		\end{tikzcd}
		\]
		This identifies the diagonal map in the above diagram, the canonical map, with the diagonal map in the bottom diagram. The Leray spectral sequence identifies the diagonal map in the bottom diagram with a shifted twist of the forget supports map. 
	\end{proof}

	\begin{corollary}\label{COR_GysinUnrImpliesWunr}
		Let $M \in \Perv{}{G}$, and $\chi \in \cvar{G}$ a character. If $M$ is unramified at $\chi$, then $M$ is weakly unramified at $\chi$.
	\end{corollary}
	\begin{proof}
		Suppose $M$ is unramified at $\chi \in \cvar{G}$. The canonical map
		\[
		i_0^*\FMSHunr{M}\otimes i_0^!\Qbarl \rightarrow i_0^!\FMST{M}
		\]
		is an isomorphism by Proposition \ref{COR_gysinisolisse}. Moreover, the stalk $i_0^*\FMSHunr{M}$ is the stalk of a perverse lisse sheaf on a smooth scheme. Hence it is concentrated in degree $-d$. The claim follows from Proposition \ref{PROP_GysinGeneralGroup}.
	\end{proof}
	Based on this description of the canonical map, we can obtain another criterion for a set of characters to be unramified.
	\begin{proposition}\label{PROP_GysinUnramifiedCriterion}
	Let $M \in \Perv{}{G}$ and $\Delta \subseteq \mathscr{C}(G)$ a closed subset. Consider the following conditions.
	\begin{enumerate}
		\item The natural map
		\[
		H^*_c(G, M_\chi) \rightarrow H^*(G, M_\chi)
		\]
		is an isomorphism for all $\chi \notin \Delta$.
		\item The complexes $H^*_?(G, M_\chi)$ are concentrated in degree zero for all $\chi \notin \Delta$. 
		\item[(3)] There exists a choice of support $?$ with
		\[
		\text{dim}(H^0_?(G, M_\chi)) = \text{dim}(H^0_?(G, M_{\chi'}))
		\]	
		for all $\chi, \chi' \notin \Delta$  
		\item[(3')] There exists a choice of support $?$ such that 
		\[
		\text{dim}(H^0_?(G, M_\chi)) = \text{dim}(M)
		\]
		for all $\chi \notin \Delta$. The dimension on the right refers to the Tannakian dimension.
	 \end{enumerate}
	If $M$ is $\Delta$-unramified, then (1) - (3') above are satisfied. If $M$ is geometrically semisimple or $G$ is affine, the following are equivalent:
\begin{enumerate}
		\item[(a)] The perverse sheaf $M$ is $\Delta$-unramified.
		\item[(b)] The conditions (1) - (3) are true.
		\item[(c)] The conditions (1) - (3') are true.
	\end{enumerate}
	
\end{proposition}
\begin{remark}
	The equivalence between  (3) and (3') above follow from the construction of a Tannakian category and the vanishing theorems. This proposition justifies the name unramified because the unramified characters are precisely those that \textit{can} give fiber functors. We prove the equivalence between (3) and (3') in Proposition soandso. 
\end{remark}
\begin{proof}
	We begin by reducing to a product $A\times T\times U$. By Theorem \ref{THM_AppendixStructureTheorem}, there is an isogeny $q\colon G \rightarrow A\times T \times U$. Since Fourier coefficients can be computed after pushforward along an isogeny, the conditions (1) - (3) and the condition of being $\Delta$-unramified are equivalent for $M$ and $\Delta$ and for $q_*M$ and $(q^*)^{-1}(\Delta)$. This follows from Proposition \ref{PROP_ProjectionFormula2}, Proposition \ref{PROP_CharDescent}, and Proposition \ref{PROP_RamifiedFiniteDescent}. Hence we can assume $G = A\times T\times U$. 
	
	Suppose $M$ is $\Delta$-unramified. Conditions (1) and (2) follow from Corollary \ref{COR_GysinUnrImpliesWunr}. Let $\eta$ be the generic point of $\widehat{U}$. We are going to prove the formula
	\[
	\chi(S, i_\eta^*(\text{FT}(M))) = \text{dim}(H^0_?(G, M_\chi))
	\]
	for all $\chi \notin \Delta$. This implies (3) because the left side of this equation does not depend on the character.
	
	Let $\chi \notin \Delta$. Generic basechange implies
	\[
	\chi(i_\eta^*(\FMQ{M})) = \chi(S, i_\eta^*(\text{FT}(M_\chi))
	\]
	by Lemma \ref{LEM_fmprojformula}. We can write $\chi = (\chi', \psi) \in \cvar{S}\times\cvar{U}$ by Proposition \ref{PROP_CharKunneth}. Because the Fourier transform admits a formula for twists by characters, we have an isomorphism 
	\[
	 i_\eta^*(\text{FT}(M_\chi)) =  i_\eta^*(\text{FT}(M)_{\chi'}).
	\]
	by Theorem \ref{THM_FTBasic} and Proposition \ref{PROP_ftfuncbase}. The Euler-Poincare characteristic does not depend on the character $\chi'$, hence we obtain
	\[
	\chi(i_\eta^*(\FMQ{M})) = \chi(S, i_\eta^*(\text{FT}(M)))
	\]
	by Proposition \ref{PROP_eulerpoincaresav}. Note that the complexes $\FMQ{M}$ are lisse. Therefore, the Euler-Poincare characteristic of the stalk and the costalk of a closed point agree with the Euler-Poincare characteristic of the stalk at the generic point. We can compute the stalks and the costalks of the Fourier-Mellin transform by Proposition \ref{LEM_fmstalks} to obtain
	\[
	\chi(S, i_\eta^*(\text{FT}(M))) = \chi_?(G, M_\chi).
	\]
	The cohomology group $H^*_?(G, M_\chi)$ is concentrated in degree zero, so
	\[
	\chi_?(G, M_\chi) = \text{dim}(H^0_?(G, M_\chi)).
	\]
	Thus
	\[
	\chi(S, i_\eta^*(\text{FT}(M))) = \text{dim}(H^0_?(G, M_\chi))
	\]
	for all $\chi \notin \Delta$. The left hand side of this equation does not depend on $\chi$, so we obtain (3).
	
	Suppose conditions (1) - (3) are satisfied and $M$ is geometrically semisimple. We are going to reduce to the affine case by writing the Fourier coefficient of $M$ as the Fourier coefficient of the relative Fourier coefficient $\pi_{T\times U*}(M_\chi)$ and applying the decomposition theorem and hard Lefschetz theorem to $\pi_{T\times U*}(M_\chi)$. This will permit us to prove that only the zeroth perverse cohomology sheaf of the relative Fourier coefficient of $\pi_{T\times U*}(M_\chi)$ can contribute to the Fourier coefficient of $M$, allowing us to reduce to the affine case.
	
	Let $\chi = (\chi_A, \chi_T, \psi) \notin \Delta$. The projection formula from Proposition \ref{PROP_ProjectionFormula2} and the decomposition theorem allow us to write the relative Fourier coefficient as
	\[
	\pi_{T\times U*}(M_{\chi}) = \bigoplus_{i \in \BZ} M^i_{(\chi_T, \psi)}[-i]
	\]
	for semisimple, perverse sheaves $M^i \in \Perv{}{T\times U}$.  By condition (2), we have a decomposition
	\[
	H^0_?(G, M_\chi) = H^*_?\bigg(T\times U, ~\bigoplus_{i \in \BZ} (M^i)_{(\chi_T,\psi)}[-i]\bigg) = \bigoplus_{i \in \BZ}  H^{*}_?(T\times U, (M^i)_{(\chi_T, \psi)})[-i].
	\]
	Note that we have
	\[
	M^i \cong M^{-i}(-i)
	\] 
	by the hard Lefschetz theorem for all $i \geq 0$.  Therefore, condition (2) forces
	\[
	H^*(T\times U, (M^i)_{(\chi_A, \psi)}) = 0
	\]
	for all $\chi \notin \Delta$ and $i \neq 0$. This implies
	\[
	\text{FM}_{\chi', (\pi_{T\times U}^*)^{-1}(\Delta), ?}(M^i) = 0
	\]
	for all $\chi'  \in \cvar{T\times U}$ with $\pi_{T\times U}^*(\chi') \notin \Delta$ and $i \neq 0$, because we can compute the stalks and costalks of the Fourier-Mellin transforms by Lemma \ref{LEM_fmstalks}. Condition (1) is unaffected by proper pushforward, therefore Proposition \ref{LEM_fmproperpush}, Proposition \ref{PROP_fmcritunr}, and the above vanishing allows us to reduce to $G$ affine.
	
	Suppose $G = T\times U$ and (1) - (3) are satisfied. We can assume $U$ is You are systematically pushing everything through:
	
	smooth by passing to $U^{\text{red}}$. Let $\chi \notin \Delta$, we recall the definition of the Fourier-Mellin transform
	\[
	\text{FM}_{\chi, \Delta, !}(M) = \text{FT}(\pi_{{U}!} (M_\chi))|_{\widehat{U}_{\chi, \Delta}}.
	\]
	Artin's vanishing theorem implies this complex is concentrated in perverse degrees $\geq 0$. On the other hand, note that we can compute the stalks by proper base change to be concentrated in degree $-d$. Hence this complex is concentrated in perverse degrees $\leq 0$. Therefore, the complex $\text{FM}_{\chi, \Delta, !}(M)$ is perverse. By condition (3), the stalks of the perverse sheaf $\text{FM}_{\chi, \Delta, !}(M)$ have the same dimension at all closed points. Therefore, Lemma \ref{LEM_PerverseLisseCrit} implies that the perverse sheaf $\text{FM}_{\chi, \Delta, !}(M)$ is lisse. Poincare duality and Lemma \ref{LEM_fmdual} allow us to conclude $\text{FM}_{\chi, \Delta, *}(M)$ is perverse and lisse as well. The canonical map 
	\[
	i_\psi^*\text{FM}_{\chi, \Delta, !}(M) \otimes i_\psi^!\Qbarl \rightarrow i_\psi^!\text{FM}_{\chi, \Delta, *}(M)
	\]
	computes the forget supports map by Proposition \ref{PROP_GysinGeneralGroup} at all closed points. By assumption (1) it is an isomorphism. Since the canonical map is an isomorphism for lisse complexes by Corollary \ref{COR_gysinisolisse}, this implies the map
	\[
	\text{FM}_{\chi, \Delta, !}(M) \rightarrow \text{FM}_{\chi, \Delta, *}(M)
	\]
	is an isomorphism.
\end{proof}

We record a corollary of the proof of Proposition \ref{PROP_GysinUnramifiedCriterion}.
\begin{corollary}\label{COR_unramdimformula}
	Let $q\colon G \rightarrow S\times U$ be an isogeny. Let $M \in \Perv{}{G}$ be a $\Delta$-unramified perverse sheaf for a closed subset $\Delta \subseteq \cvar{G}$ and $\eta$ the generic point of $\widehat{U}$. We have
	\[
	\text{dim}(H^0_?(G, M_\chi)) = \chi(S, i_\eta^* \text{FT}(q_*M)).
	\]
	for all $\chi \notin \Delta$.
\end{corollary}
\subsection{Stratifying Fourier Coefficients} The goal of this section is to stratify Fourier coefficients for semiabelian varieties. This is a crucial element in our stratification and vanishing theorems. Note that it is not possible to stratify Fourier coefficients for general groups, because character sheaves for unipotent groups have wild ramification. Similar stratifications were obtained in \cite[Ch.~2.4]{KowalskiTannaka}. We give the following example showcasing the general behavior of Fourier coefficients of complexes on semiabelian varieties in families.
\begin{example}\label{EXP_LissMell2}
	Let $n \geq 1$ and consider a non-trivial additive character $\psi$ on $k$. Define the group morphisms $i_1,\ldots, i_n\colon \BG_m \rightarrow \BG_m^m$ given by inclusion in the $i$'th coordinate. For example, $i_1(x) = (x, 1, \ldots, 1)$. We form the hypergeometric sheaf
	\[
	\SH:= i_{1!}\SL_\psi*_!i_{2!}\SL_\psi*_!\cdots *_! i_{n!}\SL_\psi.
	\]
	Define $$\pi\colon \BG_m^n \rightarrow \BG_m, (x_1, \ldots, x_n) \mapsto x_1x_2\cdots x_{n}$$ For a character $\chi = (\chi_1, \ldots, \chi_n) \in \mathscr{C}(\BG_m^n)$, the  \quotationMark{relative Fourier coefficient} is given by
	\[
	\pi_!(\SH_\chi) = \text{Hyp}(!, \psi; \chi_1,\chi_2, \ldots, \chi_n; \emptyset).
	\]
	The sheaf on the right is defined in \cite[{}8.2]{KatzESDE}. It is proven in \cite[Thm.~8.4.1]{KatzESDE}, that the Fourier coefficients are lisse. This is the type of behaviour we would like to generalize to arbitrary complexes. Note that the lisseness holds for all characters and no stratification of the character variety is required.
\end{example}
We recall the notion of a stratification and the notion of a tame family (see \cite[{}1.1,~1.3.1]{Orgogozo_2019}).
\begin{definition} 
	A \textit{stratification} of $X$ is a finite partition of $X$ into locally closed subsets. That is, we are given a finite set $\mathscr{X} = \{X_i\}$ of locally closed subsets $X_i \subseteq X$ such that
	\[
	X = \bigsqcup_{i = 1}^n X_i.
	\]
	Let $K \in \derCat{c}{X}{\Qbarl}$ be a complex. The complex $K$ is \textit{lisse} if the cohomology sheaves $H^i(K)$ are lisse for all $i \in \BZ$. The complex $K$ is called \textit{constructible along a stratification $\mathscr{X}$} if $K|_{X_i}$ is lisse for all strata $X_i \in \mathscr{X}$. \footnote{Note that any constructible $\Qbarl$-sheaf has a torsion-free integral $\CO$-model. This allows us to reduce to the case of sheaves with finite coefficients as in \cite{Orgogozo_2019}. The same principle applies to tame ramification.} 
\end{definition}
We now define a tamely ramified sheaf in our setting. The definition of a tamely ramified sheaf can also be given for an $\ell$-adic sheaf. Tame ramification will allow us to control the ramification locus of Fourier coefficients.
\begin{definition}
	A sheaf complex $K \in \derCat{c}{X}{\Qbarl}$ is called \textit{tamely ramified} (see also \cite[{}1.3.1]{Orgogozo_2019}) if the following holds: Let $Y$ be the strict henselization of a local ring of $X$. Consider a point $y \in Y$ with separable closure $\overline{y} \rightarrow y$. The action of $\pi_1(y, \overline{y})$ on $H^i(K)_{{y}}$ admits no $p$-part for all $i \in \BZ$. In other words, the image of $\pi_1(y, \overline{y})$ inside $\text{GL}(H^i(K)_y)$ is a profinite group, which can be written as a limit of finite groups of order coprime to $p$ for all $i \in \BZ$.
 \end{definition}
 The Fourier coefficients are naturally a family of complexes.
 \begin{definition}
 	A \textit{family of complexes} is a tuple $\SF = (K_\lambda)_\lambda$ of $K_\lambda \in \derCat{c}{X}{\Qbarl}$ indexed by a set $\lamdba \in I$. A family is called \textit{constructible along a stratification $\mathscr{X}$} if $K_\lambda$ is constructible along $\mathscr{X}$ for all $\lambda \in I$.  The family is called \textit{tame} if $K_\lamdba$ is tame for all $\lamdba \in I$.
 \end{definition}
We recall a general stratification theorem due to \cite{Orgogozo_2019}. Note that the family of Fourier coefficients is tame by Lemma \ref{LEM_charsabtame}. This is what will make the following theorem applicable to the formation of Fourier coefficients over a base.
\begin{theorem}\label{THM_StratificationOrgogozosTheorem}
	Let $\pi\colon X \rightarrow Y$ a morphism. Let $j\colon X \rightarrow \overline{X}$ be an open immersion such that $\pi$ factors through a proper map $\overline{\pi}\colon \overline{X} \rightarrow Y$. Let $\SF = (K_\lambda)_\lambda$ be a family of complexes on $X$. Suppose there is a stratification $\mathscr{X}$ of $X$ such that $(K_\lambda)_\lambda$ and $(D(K_\lambda))_\lambda$ are constructible along $\mathscr{X}$. Moreover, suppose there is an alteration $\alpha\colon \overline{X}' \rightarrow \overline{X}$   such that $(\alpha^*j_!(K_\lambda))_\lambda$ and $(\alpha^*j_!(D(K_\lambda)))_\lambda$ are tame families. There is a dense open subset $U \subseteq Y$ such that the complexes $\pi_?(K_\lambda)$ are lisse over $U$. Moreover, the formation of $\pi_?(K_\lambda)$ commutes with base change $Y' \rightarrow U$.
\end{theorem}
\begin{proof}
	From \cite[{}1.4.10]{Orgogozo_2019}, we obtain the stratification required to apply \cite[Thm.~5.1]{Orgogozo_2019}. We can shrink $U$ so that $U^{\text{red}}$ is smooth. We can take the dual to obtain the lisseness of $\pi_{X!}(K_\lambda)$. 
\end{proof}
We prepare the application of Orgogozo's theorem by constructing the tame families, that will yield the Fourier coefficients.
\begin{lemma}\label{LEM_StratificationPreparation}
	Consider the compactification $j\colon S \rightarrow \overline{S}$. Let  $K \in \derCat{c}{S\times X}{\Qbarl}$. There exists a finite morphism $\alpha\colon X' \rightarrow \overline{S}\times X$ and a stratification $\mathscr{X}$ of $S\times X$ such that the family $(K_\chi)_{\chi \in \mathscr{C}(S)}$ is construcible along $\mathscr{X}$ and $(\alpha^*j_!(K_\chi))_{\chi \in \mathscr{C}(S)}$ is a tame family.
\end{lemma}
\begin{proof}
	By \cite[Prop.~1.6.7]{Orgogozo_2019}, there exists a finite cover $\alpha\colon X' \rightarrow \overline{S}\times X$ such that $\alpha^* j_!K$ is tamely ramified. For each $\chi \in \mathscr{C}(S)$, the complex $\alpha^* j_!(K_\chi)$ is tamely ramified by \cite[{}5.2.5]{Orgogozo_2019}, because the sheaves $\SL_\chi$ are tamely ramified by \ref{LEM_charsabtame}. There exists a stratification $\mathscr{X}$ of $S\times X$ such that $K$ is constructible along this stratification. In particular, the complexes $K_\chi$ are constructible along this stratification by \cite[{}5.2.5]{Orgogozo_2019}, because $\SL_\chi$ is a local system on $S$.
\end{proof} 
We now apply Lemma \ref{LEM_StratificationPreparation} to Theorem \ref{THM_StratificationOrgogozosTheorem} to obtain lisse Fourier coefficients.
\begin{theorem}\label{THM_StratificationFourier}
	Let $K \in \derCat{c}{S\times X}{\Qbarl}$.
	There exists a dense open subset $U \subseteq X$ such that the formation of the Fourier coefficient $\pi_{X?}(K_\chi)$ commutes with base-change $X' \rightarrow U$ for all $\chi \in \cvar{S}$. Moreover, the complexes $\pi_{X*}(K_\chi)$ and $\pi_{X!}(K_\chi)$ are lisse over $U$ for all $\chi \in \mathscr{C}(S)$.
\end{theorem}
\begin{proof}
	We apply Lemma \ref{LEM_StratificationPreparation} to $K$ and $D(K)$. By Theorem \ref{THM_StratificationOrgogozosTheorem}, there exists a dense open subset $U \subseteq X$ such that the formation of $\pi_{X?}((K)_\chi)$ commutes with basechange and is lisse over $U$ for all $\chi \in \mathscr{C}(S)$.
\end{proof}
The following corollary is useful for the stratification theorems.
\begin{corollary}\label{COR_StratificationStrats}
	Let $K \in \derCat{c}{X\times S}{\Qbarl}$. There exists a stratification by closed subset
	\[
	X = Z_{0} \supseteq Z_1  \supseteq \ldots \supseteq Z_{n - 1} \supseteq Z_n = \emptyset
	\]
	such that 
	\begin{enumerate}
		\item The subset $Z_i$ has codimension $\geq i$.
		\item The complex $\pi_{X!}(M_\chi)$ is lisse on $Z_{i - 1} - Z_i$ for all $1 \leq i \leq n$ and $\chi \in \cvar{S}$
	\end{enumerate}
\end{corollary}
\begin{proof}
	We argue by induction on the dimension of $X$. If the dimension is zero, there is nothing to prove.  By Theorem \ref{THM_StratificationFourier}, there exists a nowhere dense closed subset $Z_0 \subseteq X$ such that $\pi_{X?}(M_\chi)$ are lisse for all $\chi \in \cvar{S}$. By induction and proper base change, there is a stratification of $Z_0$ such that $\pi_{X!}(M_\chi)$ is lisse along the strata.
\end{proof}
\subsection{Semiabelian Varieties} We recall the vanishing theorem for an abelian variety $A$.  
\begin{theorem}[{\cite[Vanishing theorem]{Weissauer16}}]\label{THM_VanishingAbelian}
	Let $M \in \Perv{}{A}$ be a perverse sheaf on $A$. There exists a  proper thin subset $\Delta \subset \mathscr{C}(A)$ such that the cohomology complex $H^*(A, K_\chi)$ is concentrated in degree zero.
\end{theorem}
We deduce the relative version: 
\begin{theorem}\label{THM_VanishingAbelianRel}
 Let $M \in \Perv{}{A\times X}$. There exists a proper thin subset $\Delta \subset \mathscr{C}(A)$ such that $\pi_{X*}(M_\chi)$ is perverse for $\chi \notin \Delta$. 
\end{theorem}
\begin{proof}We argue by induction on the dimension of $X$. If the dimension of $X$ is zero we can apply Theorem \ref{THM_VanishingAbelian}. In the general case, we remark that we can assume $X$ is affine by passing to an open covering of $X$. Noether normalization allows us to assume $X = \BA^{d + 1}$ for some $d \geq 0$.
	
	By Theorem \ref{THM_StratificationFourier} and Lemma \ref{LEM_PervRelativePerversity}, there exist a dense open subset $U \subseteq X$ such that $\pi_{X*}(M_\chi)$ is lisse over $U$ and $i_x^*(M)$ is concentrated in perverse degree $-d$ for all points $x \in U$. We can shrink $U$ so that the complement $D := X - U$ is a Cartier divisor. We denote by $i\colon D \rightarrow X$ the immersion of the reduced closed subscheme $D \subseteq X$.
	
	Artin's vanishing theorem implies $i^*(M)$ has perverse amplitude $[-1, 0]$. By the induction hypothesis, we can find a proper thin subset $\Delta_1 \subset \cvar{A}$ such that $\pi_{D*}((i^*M)_\chi)$ has perverse amplitude $[-1, 0]$ for all $\chi \notin \Delta_1$. 
	
	Let $x \in U$ be a closed point. By Theorem \ref{THM_VanishingAbelian}, there exists a closed subset $\Delta_2 \subset \cvar{A}$ such that $H^*(A, i_x^*(M_\chi))$ is concentrated in degree $-d$ for all $\chi \notin \Delta_2$. The subset $U$ is connected, therefore $\pi_{X*}(M_\chi)$ is lisse and concentrated in degree $-d$ over $U$. Since $U$ is smooth, proper base change implies $\pi_{X*}(M_\chi)$ has perverse amplitude $\leq 0$ for all $\chi\notin\Delta_1\cup\Delta_2$. 
	
	We can apply Poincare duality to the above statement to find a proper thin subset $\Delta_3 \subset \cvar{A}$ such that $\pi_{X*}(M_\chi)$ has perverse amplitude $\geq 0$ for all $\chi \notin \Delta_3$. For all $\chi\notin \Delta_1\cup\Delta_2\cup\Delta_3$, the complex $\pi_{X*}(M_\chi)$ is perverse. A finite union of proper thin subsets is proper and thin by Proposition \ref{PROP_thintop}.
\end{proof}
We recall the relative vanishing theorems for a torus $T$.
\begin{theorem}[{\cite[Thm.~2.11]{KowalskiTannaka}}]\label{THM_VnshTorus}
	Let $M \in \Perv{}{T\times X}$. There exists a proper thin subset $\Delta \subset \cvar{T}$ such that the natural morphism
	\[
	\pi_{X!} (M_\chi) \rightarrow \pi_{X*}(M_\chi)
	\]
	is an isomorphism. In particular, the complexes $\pi_{X!} (M_\chi)$ and $\pi_{X*}(M_\chi)$ are perverse. 
\end{theorem}
We can apply the stratification theorem and put these two relative vanishing theorems together to obtain the following relative vanishing theorem for a semiabelian variety $S$.
\begin{theorem}\label{THM_VanSemiAb}
	Let $M \in \Perv{}{S\times X}$ be a perverse sheaf on $S$. There exists a dense open subset $U \subseteq X$ and a proper thin subset $\Delta \subset \cvar{S}$ such that
	\begin{enumerate}
		\item The complexes $\pi_{X?}(M_\chi)$ are lisse over $U$ and their formation commutes with arbitrary base change $X' \rightarrow U$ for all $\chi \in \mathscr{C}(S)$.
		\item The forget supports map
		\[
		\pi_{X!}(M_\chi) \rightarrow \pi_{X*}(M_\chi)
		\]
		is an isomorphism for all $\chi \notin \Delta$.
		\item The complexes $\pi_{X?}(M_\chi)$ are perverse for all $\chi \notin\Delta.$
	\end{enumerate}
\end{theorem}
\begin{proof}
	We have already proven conclusion (1) in Theorem \ref{THM_StratificationFourier}. .
	
	There is an isogeny $q\colon S \rightarrow A\times T$ by Theorem \ref{THM_AppendixStructureTheorem}. Suppose there is a closed proper subset $\Delta' \subseteq \mathscr{C}(A\times T)$ such that (2) and (3) are satisfied for $q_*M$. The projection formula, i.e. Proposition \ref{PROP_ProjectionFormula2} implies all $\chi\notin q^*(\Delta')$ satisfy (2) and (3) for $M$. Because $q^*$ is closed and preserves proper thin subsets by Proposition \ref{PROP_IsogenyPreservesDimension}, we can assume $G = A\times T$. 
	
	We have a square
	\[
	\begin{tikzcd}
		A\times T\times X \arrow[r, "\pi_T"] \arrow[d, "\pi_A"'] & T \times X\arrow[d, "\pi_X"] \\
		A \times X\arrow[r, "\pi_X"]                   & X.            
	\end{tikzcd}
	\]
	Our goal is to prove (2) and (3) at once by applying the relative vanishing theorem to hold for $M$ with respect to $\pi_A$ and $\pi_T$. We then go around the square \quotationMark{both ways}. The relative vanishing theorem for $\pi_T$ makes the forget support map an isomorphism. The relative vanishing theorem for $\pi_A$ makes the pushforward to $T\times X$ perverse. We know by requiring the vanishing theorems for $A$ to hold that the natural map $\pi_{X!}(M_\chi)\rightarrow \pi_{X*}(M_\chi)$ is an isomorphism on $X$. Moreover, we know that the relative Fourier coefficient on $X$ is $\pi_{X?}(\pi_{T*}(M_\chi))$ hence it is the pushforward of a perverse sheaf under an affine map by the vanishing theorem for $\pi_T$. Hence it must satisfy condition (1) and (2) by Artin's vanishing theorem.
	
	By Theorem \ref{THM_VanishingAbelianRel}, there exists a proper thin subset $\Delta_A \subset \mathscr{C}(A)$ such that the pushforward $\pi_{T*}(M_\chi)$ is perverse for each $\chi \notin \Delta_A$. In particular, for each $\chi' \in \mathscr{C}(T)$, the pushforward $\pi_{T*}(M_{(\chi',\chi)}) = (\pi_{T*}(M_{\chi}))_{\chi'}$ is perverse by the projection formula from Proposition \ref{PROP_ProjectionFormula2}. Artin's vanishing theorem  implies $\pi_{X!}(M_{(\chi, \chi')})$ has perverse amplitude $\geq 0$ and $\pi_{X*}(M_{(\chi, \chi')})$ has perverse amplitude $\leq 0$ for all $\chi \notin \Delta_A$ and $\chi' \in \mathscr{C}(T)$.
	
	By Theorem \ref{THM_VnshTorus}, there exists a proper thin subset $\Delta_T \subset \mathscr{C}(T)$, such that 
	\[
	\pi_{A!}(M_\chi) \rightarrow \pi_{A*}(M_\chi)
	\]
	is an isomorphism for all $\chi \notin \Delta_T$. By the projection formula, the map 
	\[
	\pi_{A!}(M_{(\chi, \chi')}) \rightarrow \pi_{A*}(M_{(\chi, \chi')})
	\]
	is an isomorphism for all $\chi' \in \mathscr{C}(A)$ and $\chi \notin \Delta_T$. Therefore, the map
	\[
	\pi_{*}\pi_{A!}(M_{(\chi, \chi')})) \rightarrow \pi_*\pi_{A*}(M_{(\chi, \chi')}))
	\]
	is an isomorphism. Since $A$ is proper, the Leray spectral sequence identifies this map with the natural map
	\[
	\pi_{X!}(M_\chi) \rightarrow \pi_{X*}(M_\chi)
	\]
	This map is an isomorphism for all $\chi \notin \Delta_T$. Therefore, we have (2) and (3) for all characters $$(\chi, \chi') \notin \Delta_T\times\mathscr{C}(A)\cup \mathscr{C}(T)\times\Delta_A.$$
\end{proof}
We obtain the following relative variant for group morphisms from smooth descent. This is based on the vanishing theorem \cite[Thm.~2.1]{Kr_mer_2015}
\begin{corollary}\label{COR_VanishingRelGroups}
	We consider a surjective group morphism $\pi\colon G \rightarrow G'$ with kernel $S$. Denote by $i\colon S \rightarrow G$ the inclusion. Let $M \in \Perv{}{G}$. There is a proper thin subset $\Delta \subset \cvar{S}$ and a dense open subset $U \subseteq G$	Let $X$ be a point. In this case, the stratification can be identified with a stratification by thin subsets inside $\cvar{G}$ with the following properties.
	\begin{enumerate}
		\item The complexes $\pi_{?}(M_\chi)$ are lisse over $U$ for all $\chi \in \mathscr{C}(G)$ and their formation commutes with arbitrary base change $X \rightarrow U$. 
		\item The forget supports map
		\[
		\pi_!(M_\chi) \rightarrow \pi_*(M_\chi)
		\]
		is an isomorphism for all $\chi \in \mathscr{C}(G)$ such that $i^*(\chi) \notin \Delta$.
		\item The complex $\pi_{?}(M_\chi)$ is perverse for all $\chi \in \mathscr{C}(G)$ such that $i^*(\chi)\ \notin \Delta$.
	\end{enumerate}
\end{corollary}
\begin{proof}
	We have a Cartesian square
	\[
	\begin{tikzcd}
		S\times G \arrow[r, "\pi_{G}"] \arrow[d, "m"] & G \arrow[d, "\pi"] \\
		G \arrow[r, "\pi"]                            & G'.                
	\end{tikzcd}
	\]
	Note that, by the Künneth formula,
	\[
	m^*(\SL_\chi) = \SL_{i^*(\chi)}\boxtimes \SL_{\chi}.
	\]
	Let $d$ be the dimension of $S$. We apply Theorem \ref{THM_VanSemiAb} to $m^*M[d]$. Then (2) and (3) follow from smooth base change, proper base change, and $t$-exactness of $\pi^*(-)[d]$. Let $U' \subseteq G$ be a dense open where the Fourier coeffcients of $m^*M[d]$ are lisse and commute with base change. This implies that the Fourier coefficients of $M$ are lisse over $U := \pi(U')$. This set is open, because $\pi$ is open. By smooth base change, we can verify the basechange property over $U$ by proving the base change property over $U'$. 
\end{proof}

\subsection{Vanishing theorem} We now prove the vanishing theorem by taking a Fourier transform and applying the relative vanishing theorem. 
\begin{theorem}\label{THM_VanishingTHM}
	Let $M \in \Perv{}{G_{\overline{k}}}$. There exists a proper thin subset $\Delta \subset \cvar{G}$ such that $M$ is $\Delta$-unramified.
\end{theorem}
\begin{proof} We can assume $M$ is geometrically irreducible, hence arithmetic. 
	By Theorem \ref{THM_AppendixStructureTheorem}, we have an isogeny $q\colon G \rightarrow S\times U$ from $G$ to a product of a semiabelian variety $S$ and a unipotent group $U$. Suppose there is a closed proper subset $\Delta\subseteq \cvar{S\times U}$ such that $q_*M$ is $\Delta$-unramified. The subset $q^*(\Delta)\subset \cvar{G}$ is a proper thin subset by Proposition \ref{PROP_IsogenyPreservesDimension}. By Proposition \ref{PROP_RamifiedFiniteDescent}, the perverse sheaf $M$ is $q^*(\Delta)$-unramifed. Thus we can assume $G = S\times U$. 
	
	We apply Theorem \ref{THM_VanSemiAb} to the projection $\pi\colon S\times \widehat{U} \rightarrow \widehat{U}$ and the perverse sheaf $\text{FT}(M).$ We obtain an open subset $V\subseteq \widehat{U}$ and a closed arithmetic proper subset $\Delta \subseteq \mathscr{C}(S)$. Note that the Fourier-Mellin transform $\FMQ{M}$ can be written as a Fourier coefficient of $\text{FT}(M)$ by Proposition \ref{LEM_fmprojformula}. Unwinding the definitions, we find that $M$ is unramified for all $(\chi, \psi) \in \cvar{S}\times\cvar{U}$ with $\chi \notin \Delta$ and $\psi \in V$. Since the complement $Z\subset\cvar{U}$ of $V$ is closed, the subset $\Delta\times \cvar{U}\cup \cvar{S}\times Z$ is a proper thin subset.
\end{proof}
\subsection{Stratification theorems} The relative vanishing theorem for semiabelian varieties also applies to give a general stratification theorem for exponential sums indexed by a scheme over a finite field. We begin by stating the stratification theorem. We then reduce it to a sheaf theoretic stratification theorem, which we prove by induction on the dimension of the base space.
\begin{theorem}\label{THM_stratificationbounds}
	Let $G = S\times U$ and $X$ of pure dimension $d$. Let $K \in \derCat{c}{G\times X}{\Qbarl}$ be a complex of weight $\leq 0$ and with perverse amplitude $\leq d$. There exists:
	\begin{enumerate}
		\item A stratification by closed subsets \[\widehat{U}\times X = Z_{0} \supseteq Z_1 \supseteq  \ldots \supseteq Z_{n - 1} \supseteq Z_n = \emptyset \]
		such that for all $0 \leq i \leq n$ the subset $Z_i$ has codimension $\geq i$. 
		\item A stratification by thin subsets 
		\[
		\cvar{S} = \Delta_{0} \supset \Delta_1 \supseteq \ldots \supseteq \Delta_{n - 1} \supseteq \Delta_n = \emptyset
		\]
		such that for all $0\leq j \leq n$ the subset $\Delta_j$ has codimension $\geq j$.
	\end{enumerate}
	Let $m \geq 1$ and $0\leq i, j\leq n - 1$.  For all $\chi \in \widehat{G}(k_m)$ and $x \in X(k_m)$ with $(\pi_U^*(\chi), x) \notin Z_{i + 1}$ and $\pi_S^*(\chi) \notin \Delta_{j + 1}$, we have
	\[
	\sum_{t \in G(k_m)}\chi(t)t_K(t, x)\ll_K |k_m|^{(i + j)/2}.
	\]
\end{theorem}
\begin{remark}
	The number of arithmetic characters rational over a finite field $k_n$ can be estimated by Proposition \ref{PROP_thinsubsetsestimate}.
\end{remark}
By Deligne's generalized Riemann hypothesis \cite{DeligneWeilII}, this statement can be reduced to the following relative stratification theorem for Fourier coefficients.
\begin{theorem}\label{THM_stratcoh}
	Let $G = S\times U$ and $X$ of pure dimension $d$. Let $K \in \derCat{c}{G_{\overline{k}}\times X_{\overline{k}}}{\Qbarl}$ have perverse amplitude $[a + d, b + d]$.\footnote{Recall our convention that complexes are quasi-arithmetic. For example, a geometric complex is quasi-arithmetic.} We consider a morphism $i$ and $\pi$  as above. There exists:
	\begin{enumerate}
		\item A stratification by closed subsets \[\widehat{U}\times X = Z_{0} \supseteq Z_1 \supseteq  \ldots \supseteq Z_{n - 1} \supseteq Z_n = \emptyset \]
		such that for all $1 \leq i \leq n$ the subset $Z_i$ has codimension $\geq i$.
		\item A stratification by thin subsets
		\[
		\cvar{S} = \Delta_0 \supset \Delta_1 \supseteq \ldots \supseteq \Delta_{m - 1} \supseteq \Delta_m = \emptyset
		\]
		such that for all $1\leq j \leq m$ the subset $\Delta_j$ has codimension $\geq j$.
	\end{enumerate}
	such that:
	\begin{enumerate}
		\item For all $\chi \in \cvar{G}$ and $x \in X$ with $(\pi_U^*(\chi), x) \notin Z_0$ and $\pi_S^*(\chi) \notin \Delta_0$, the natural map
		\[
		H^*_c(G, (i_x^*K)_\chi) \rightarrow H^*(G, (i_x^*K)_\chi)
		\]
		is an isomorphism.
		\item  Consider $0\leq i \leq n - 1$ and $0 \leq j \leq m - 1$. For all $\chi \in \cvar{G}$ and $x \in X$ with $(\pi_U^*(\chi), x) \notin Z_{i + 1}$ and $\pi_S^*(\chi) \notin \Delta_{j + 1}$ the cohomology complex $H^*_c(G, (i_x^*K)_\chi)$ has amplitude $[-j + a, i + j + b]$. 
	\end{enumerate}
\end{theorem}
\begin{remark}
	The argument in Corollary \ref{COR_VanishingRelGroups} allows one to apply this stratification theorem to $G$-torsors. 
\end{remark}
\begin{proof} By the Leray spectral sequence, we can assume $K$ is a perverse sheaf $M$. Moreover, we can assume $M$ is geometrically irreducible. Hence $M$ is arithmetic. 
	
	Our first step is to consider the case when $G = S$ and $X$ is a point. In this case, Theorem \ref{THM_VanSemiAb} shows there exists a proper thin subset $\Delta_0 \subset \cvar{S}$ such that for all $\chi \notin \Delta_0$, the map
	\[
	H^*_c(G, M_\chi) \rightarrow H^*(G, M_\chi)
	\]
	is an isomorphism and $H^*_?(G, M_\chi)$ is concentrated in degree zero. In particular, we obtain (1) and the first stratum of (2). We prove (2) by induction on the dimension of $S$. The thin subset $\Delta_0$ admits an irreducible decomposition in the thin topology, by Proposition \ref{PROP_thintop},
	\[
	\Delta_0 = \Delta_{00}\cup\Delta_{01}\cup\ldots\cup \Delta_{0m}.
	\]
	For each $0 \leq j \leq m$, there is a character $\chi_i \in \cvar{S}$ and a quotient map $\pi_i\colon S_{\overline{k}}\rightarrow S_i$ with connected fibers such that
	\[
	\Delta_{0i} = \chi_i\cdot \pi_i^*(S_i).
	\]
	Consider a character $\chi \in \Delta_{0i}$ for some $0 \leq i \leq m$. There is a character $\chi' \in \cvar{S_i}$ such that $\chi = \chi_i\pi_i^*(\chi')$. In particular, the projection formula recorded in Proposition \ref{PROP_ProjectionFormula2} implies
	\[
	H^*_?(S, M_\chi) = H^*_?(S_i, (\pi_?(M))_{\chi'}).
	\]
	The induction hypothesis implies the existence of the required stratifications by standard bounds on the cohomological amplitude of $\pi_?$. 
	
	We turn to the general case. By Theorem \ref{THM_AppendixStructureTheorem}, there is an isogeny $q\colon G \rightarrow S\times U$. The projection formula in Proposition \ref{PROP_ProjectionFormula2} allows us to replace $M$ by $q_*(M)$ and assume $G = S\times U$. 
	
	We put
	\[
	M' := \text{FT}(M).
	\]
	We can stratify the perversity of the fibers of $M'$ over $\widehat{U}\times X$ by Corollary \ref{COR_StratificationStrats} and we can stratify the Fourier coefficients of $M'$ by Proposition \ref{PROP_PervRelativePerversity} and Theorem \ref{THM_StratificationFourier}. We take the union to obtain a stratification by closed subset
	\[\widehat{U}\times X = Z_{0} \supseteq Z_1 \supseteq  \ldots \supseteq Z_{n - 1} \supseteq Z_n = \emptyset \]
	such that:
	\begin{enumerate}
		\item The closed subset $Z_i$ has codimension $\geq i$.
		\item The Fourier coefficient $\pi_{X!}(M'_\chi)$ is lisse along this stratification for all $\chi \in \cvar{S}$ .
		\item The Fourier coefficient $\pi_{X*}(M'_\chi)$ is lisse on $X - Z_1$ for all $\chi \in \cvar{S}$.
		\item The formation of $\pi_{X?}(M'_\chi)$ commutes with basechange $X' \rightarrow X - Z_1$.
		\item The complex $i_x^*(M')$ has perverse amplitude $[a, i + b - 1]$ for all $0 \leq i \leq n$ and closed points $x \notin Z_i$.
	\end{enumerate}
	For all $1 \leq i \leq n$, we choose finite subsets of closed points $X_i \subseteq Z_{i + 1} - Z_{i}$ such that the map on connected components $X_i \rightarrow \pi_0(Z_{i + 1} - Z_i)$ is surjective. We will test the vanishing theorems for each stratum along these points. We can stratify $\cvar{S}$
	\[
	\cvar{S} = \Delta_0 \supset \Delta_1 \supseteq \ldots \supseteq \Delta_{n - 1} \supseteq \Delta_n = \emptyset
	\]
	so that the cohomology of the fibers at $X_i$ is controlled as follows:
	\begin{enumerate}
		\item For all $\chi \notin \Delta_1$ and $(\psi, x) \in X_1$, the natural amp
		\[
		H^*_c(G, (i_{(\psi, x)}^*(M'))_\chi) \rightarrow H^*(G, (i_{(\psi, x)}^*(M'))_\chi)
		\]
		is an isomorphism.
		\item For all $1 \leq i \leq n$, $0\leq j \leq m - 1$, $\chi \notin \Delta_{j + 1}$, and $(\psi, x)\in X_i$, the complexes $$H^*_c(G, (i_{(\psi, x)}^*(M'))_\chi)$$ have amplitude $[- j + a, i + j + b]$ 
	\end{enumerate}
	We spread the vanishing obtained from stratifying $\cvar{S}$ out to the other fibers from the lisseness of the Fourier coefficients along the stratifications constructed on $\widehat{U}\times X$. Let $\chi \in \cvar{S}$ and $(\psi, x) \notin Z_1$ be a closed point. The stalk of the natural morphism
	\[
	\pi_{X!}(M'_\chi)\rightarrow \pi_{X*}(M'_\chi)
	\]
	at $(\psi, x)$ is given by
	\[
	H^*_c(G, (i_{(\psi, x)}^*(M))_\chi) \rightarrow H^*(G, (i_{(\psi, x)}^*(M))_\chi)
	\]
	because the formation of $\pi_{X*}(M'_\chi)$ commutes with arbitrary base change. Since the complexes $\pi_{X?}(M_\chi)$ are lisse and their formation commutes with arbitrary basechange,  there exists a point $x_0 \in X_0$ such that this map is isomorphic to
	\[
	H^*_c(G, (i_{(\psi, x_0)}^*(M))_\chi) \rightarrow H^*(G, (i_{(\psi, x_0)}^*(M))_\chi)
	\]
	In particular, the stratum given by $Z_1$ and $\Delta_1$ satisfies the first conclusion of the theorem. Similarly, by proper base change the strata satisfy the second conclusion of the theorem with supports.
\end{proof}



%% file: Classification.tex
		\section{Classification of negligible sheaves}
		\subsection{Semiabelian varieties} We classify the geometrically irreducible negligible sheaves for semiabelian varieties $S$ over finite fields by the following theorem, which generalizes \cite[Main~Theorem]{Weissauer16} and \cite[Thm.~5.1.1]{GabberLoeserTore} over finite fields to constant semiabelian varieties over a base scheme. Our arguments draw in an essential way from the classification of characters from \cite{GabberLoeserTore}.
	\begin{theorem}\label{THM_Classification}
		Let $M \in \Perv{}{S}$ be geometrically irreducible with $\chi(S, M) = 0$. There is a semiabelian variety $S'$ over $k$, a surjective group morphism $\pi\colon S \rightarrow S'$ with connected fibers of dimension $d > 0$, $N \in \Perv{}{G}$ with $\chi(S', N) \neq 0$, and a character $\chi \in \widehat{S}(k)$ such that
		\[
		M = \pi_{}^*N\otimes \SL_\chi[d].
		\]
	\end{theorem}
	There is also the following relative variant.
	\begin{theorem}\label{THM_ClassCharactersRelative}
		 Let $M \in \Perv{}{S\times X}$ be a geometrically irreducible perverse sheaf, such that there is a character $ \nu \in \mathscr{C}(S)$ with\footnote{We have $[\pi_{X!}(M_\chi)] = [\pi_{X*}(M_\chi)]$ in the $K$-groug, hence the two conditions are equivalent.} $$[\pi_{X?}(M_\chi)] = 0$$ in the $K$-group of $X$. There is a quotient $\pi\colon S \rightarrow S'$ with connected fibers of dimension $d > 0$, a perverse sheaf $N \in \Perv{}{S'\times X}$, and a character $\chi \in \widehat{S}(k)$ such that
		\[
		M = \pi^*N\otimes\SL_\chi[d].
		\]
		Moreover, for each $\chi \in \cvar{S}$ we have $\pi_{X!}(M_\chi) \neq 0$ if  and only if $\chi \in \overline{\chi}\cdot\pi^*(\cvar{S'})$. Dually, we also have $\pi_{X*}(M_\chi) \neq 0$ if  and only if $\chi \in \overline{\chi}\cdot\pi^*(\cvar{S'})$
	\end{theorem}
	The relative version is \cite[Thm~5.1.2]{GabberLoeserTore} for the semiabelian case. We then deduce the following relative variant.
	\begin{corollary}\label{COR_ClassificationRelativeGroup}
		Let $\pi\colon G \rightarrow G'$ a surjective group morphism. We assume the kernel of $\pi$ is a semiabelian variety $S$. Let $M \in \Perv{}{G\times X}$ be a geometrically irreducible perverse sheaf. Suppose there is $\nu \in \mathscr{C}(G)$ with
		\[
		[\pi_{?}(M_\nu)] = 0 
		\]
		in the $K$-group of $G'\times X$. There is a semiabelian subvariety $i\colon S' \rightarrow S$ of dimension $d > 0$ with quotient $\pi'\colon G \rightarrow G''$, a perverse sheaf $N \in \Perv{}{G''\times X}$, and a character $\chi \in \widehat{G}(k)$ such that
		\[
		N = \pi^*(N)\otimes\SL_\chi[d].
		\]
		Moreover, for each $\chi \in \cvar{G}$ we have $\pi_{!}(M_\chi) \neq 0$ if and only if $i^*(\chi) \in \overline{\chi}\cdot{\pi'}^*(\cvar{S'})$. Dually, for each $\chi \in \cvar{G}$ we have $\pi_{*}(M_\chi) \neq 0$ if and only if $i^*(\chi) \in \overline{\chi}\cdot{\pi'}^*(\cvar{S'})$.
	\end{corollary}
	We give an overview of the argument. The argument proceeds by induction of the rank of the maximal torus in $S$. The induction start is given to us by (an arithmetic extension of) the classfication given in \cite{Weissauer16}. We begin by proving that the absolute version implies the relative version in Lemma \ref{LEM_ClassificationAbsoluteImpliesRelative}. We then deduce the relative variant for group morphisms in Lemma \ref{LEM_classabsimplrelgrp}. We can now begin the induction following \cite{GabberLoeserTore}. We distinguish two cases in Lemma \ref{LEM_ClassificationInductionStep1}. The first case is treated by the case distinction and the second case is then treated in more detail in Lemma \ref{LEM_ClassificationInductionStep2}.

	\subsubsection{Abelian varieties} We begin by discussing a minor extension of the classification in the abelian case. We then introduce a descent lemma, which permits us to deduce the arithmetic classification stated here from the geometric version in \cite{Weissauer16}.
	\begin{theorem}[{\cite[Main Theorem]{Weissauer16}}]\label{THM_ClassificationNegligibleAbelian}
		Let $M \in \Perv{}{A}$ be a geometrically irreducible perverse sheaf with $\chi(M) = 0$. There is a quotient $\pi\colon A \rightarrow A'$ with connected fibers of dimension $d$, a perverse sheaf $N \in \derCat{c}{A'}{\Qbarl}$ with $\chi(A, N) \neq 0$, and a character $\chi \in \widehat{A}(k)$ such that
		\[
		M =  \pi^* N\otimes\SL_\chi[d]
		\]
	\end{theorem}
	We recall a well-known theorem on perverse sheaves which will play a crucial role.
	\begin{lemma}\label{LEM_pullbackcrit}
		Let $\pi\colon X \rightarrow Y$ be a map with geometrically connected fibers of dimension $d$. Let $M \in \Perv{}{X}$ be an irreducible perverse sheaf such that
		\[
		\pervCoh{d}{\pi_!(M)} \neq 0.
		\]
		Then $M = \pi^*N[d]$ for a perverse sheaf $N \in \Perv{}{Y}$. 
	\end{lemma}
	\begin{proof}
		See \cite[Cor.~4.2.6.2]{BBD}.
	\end{proof}
	The following lemma descends the tac by proving that it is unique. 
	\begin{lemma}\label{LEM_UniquenessOfInvariantSubgroup}
		Let $S$ be a semiabelian variety over a finite field $k$ and $M \in \Perv{}{S}$ a perverse sheaf. 
		Suppose there is a quotient $\pi\colon S_{\overline{k}} \rightarrow S'_{\overline{k}}$ with connected fibers of dimension $d$, a perverse sheaf $N\in\Perv{}{S'_{\overline{k}}}$ with $\chi(S'_{\overline{k}}, N) \neq 0$, and a character $\chi \in \mathscr{C}(S)$ such that
		\[
		M = \pi^*N\otimes\SL_{\chi}[d].
		\]
		The map $\pi$ descends to a quotient $\pi\colon S \rightarrow S'$ over $k$. There are $\chi' \in \widehat{S}(k)$ and  $N' \in \Perv{}{S'}$ with $\chi(S', N') \neq 0$ and
		\[
		M  = \pi^*N'\otimes\SL_{\chi'}
		\]
	\end{lemma}
	\begin{proof} We begin by defining the set
		\[
		\Delta := \chi\cdot \pi^*\mathscr{C}(S_{\overline{k}}')
		\]
		Suppose $\nu \in \mathscr{C}(S)$. Then we have 
		\[
		H^*(S, M_\nu) \neq 0
		\]
		if and only if $\nu \in \Delta$ by the projection formula recorded in Proposition \ref{PROP_ProjectionFormula1} and the non-vanishing of the Euler-Poincare characteristic of $N$. Therefore, the subset $\Delta$ only depends on the non-vanishing of certain cohomology complexes. This condition is Frobenius equivariant. In particular, Theorem \ref{THM_DescendingTACS} implies that the set $\chi\cdot \pi^*\mathscr{C}(S'_{\overline{k}})$ is of the form $\chi'\cdot \pi^*\mathscr{C}(S')$. In particular, the quotient $\pi$ descends to $k$ and there is $\chi' \in \widehat{S}(k)$ with $\chi' \in \chi\cdot\pi^*(\cvar{S'_{\overline{k}}})$. In particular, we obtain
		\[
		\pervCoh{d}{\pi_{!}(M_{\overline{\chi}})} \neq 0
		\]
		from the projection formula recorded in Proposition \ref{PROP_ProjectionFormula1}. We can construct the perverse sheaf $N'$ from Lemma \ref{LEM_pullbackcrit}. If the Euler-Poincare characteristic of $N'$ would vanish, then Theorem \ref{THM_VanSemiAb} and the projection formula would imply that the characters with non-vanishing cohomology are contained in a proper thin subset inside $\Delta$. 
	\end{proof}

	\subsubsection{Absolute implies relative} We can now prove that the absolute version implies the relative version of the classification theorem.
	\begin{lemma}\label{LEM_ClassificationAbsoluteImpliesRelative}
		Let $k$ be a finite field and $S$ a semiabelian variety over $k$. Suppose Theorem \ref{THM_Classification} is true for perverse sheaves on $S_{k'}$ for all finite extensions $k'/k$. Then Theorem \ref{THM_ClassCharactersRelative} is true for all finite type schemes $X$ over $k$ and all perverse sheaves on $S\times X$.
	\end{lemma}
	To prove this lemma, we will utilize certain properties of the Fourier support. The Fourier support will be introduced later (see Definition \ref{DEF_fsupp}). The first point of the following Lemma will be proven in full generality in Theorem \ref{THM_SupportLemma}.
	\begin{lemma}\label{LEM_ClassificationFourierSupp}
		Let $S$ be a semiabelian variety over a finite field $k$ such that Theorem \ref{THM_Classification} holds for perverse sheaves on $S_{k'}$ for all finite extensions $k'/k$. Let $M \in \Perv{}{S}$ be a perverse sheaf. Let $k'/k$ be an extension such that all composition factors of $M_{k'}$ are geometrically irreducible. Let $M_1, \ldots, M_n \in \Perv{}{S}$ be the geometrically irreducible factors of $M_{k'}$ in a composition series of $M$.
		\begin{enumerate}
			\item The Fourier support satisfies
			\[
			\text{FSupp}_!(M) = \bigcup_{i = 1}^n \text{FSupp}_!(M_i).
			\]
			\item Let $1 \leq i \leq n$. The Fourier support $\text{FSupp}_!(M_i)$ is a tac. Let $\pi\colon S_{k'} \rightarrow S'$ be a quotient  with connected fibers of dimension $d$ and $\chi \in \widehat{S}(k')$. There is $N \in \Perv{}{S'}$ with $\chi(S', N) \neq 0$ and
			\[
			M_i = \pi^*N\otimes \SL_\chi
			\]
			if and only if the Fourier support is given by
			\[
			\text{FSupp}_!(M_i) = \chi\cdot\pi^*(\cvar{S'}).
			\]
			In particular, the Fourier support is not empty.
			\item Suppose $\text{FSupp}_!(M) \neq \cvar{S}$. There is a quotient $\pi\colon S_{k'} \rightarrow S'$ with connected fibers of dimension $d > 0$ and a character $\chi \in \cvar{S}$ such that  
			\[
			\pervCoh{d}{\pi_!(M_\chi)} \neq 0.
			\]
			We can choose the quotient and the character so that there is $1 \leq i \leq n$ with $\chi_i\cdot\cvar{S'} = \text{FSupp}_!(M_i)$.
		\end{enumerate} 
	\end{lemma}
	\begin{proof}
		The first point is going to be proven in Theorem \ref{THM_SupportLemma}. 
		
		The second point follows from the classification of sheaves with $\chi(S', M) = 0$. The classification allows one to compute the support of a sheaf with vanishing Euler-Poincare characteristic from the projection formula (as in Lemma \ref{LEM_UniquenessOfInvariantSubgroup}).
		
		For the third point, note that the support is a finite union of tacs. In particular, the Fourier support of each $M_i$ is a proper tac. There exists $1 \leq i \leq n$ such that there is a surjection $M_{\overline{k}} \rightarrow M_i$. Note that $\pervCoh{d}{\pi_!((M_{i})_{\overline{\chi}})} \neq 0$ by the projection formula in Proposition \ref{PROP_ProjectionFormula1} and the computation of cohomology in Proposition \ref{PROP_SemiAbCohomology} The bound on the cohomological amplitude of $\pi_!$ implies that the map
		\[
		\pervCoh{d}{\pi_!((M_{k'})_{\overline{\chi}})}  \rightarrow \pervCoh{d}{\pi_!((M_{i})_{\overline{\chi}})} 
		\]
		is surjective. Thus $\pervCoh{d}{\pi_!(M_{\overline{\chi}})}\neq 0$. 
	\end{proof}
	We can now turn to proving the induction step.
	\begin{proof}[Proof~({Lemma~\ref{LEM_ClassificationAbsoluteImpliesRelative}})]
		
		Let $Y \subseteq X$ be the schematic image of the (reduced) support of $M$ in $X$. The scheme $Y$ is irreducible. The formation of the schematic image commutes with field extensions, hence $Y$ is geometrically irreducible. We can base change $X$ to $Y$ and assume $X$ is geometrically irreducible. Let $\eta$ be the generic point of $X$ and $c$ the dimension of $X$.
		
		We have
		\[
		[\pi_{X*}(M_\nu)] = [\pi_{X!}(M_\nu)]
		\]
		by \cite{virk2014eulerpoincarecharacteristics}. Therefore
		\[
		[\pi_{X!}(M_\nu)]  = 0.
		\]

		Suppose there exists a character $\nu \in \cvar{S}$ with
		\[
		[\pi_{X!}(M_\nu)] = 0.
		\]
		For all $\chi \in \cvar{S}$,  Proposition \ref{PROP_eulerpoincaresav} and proper base change imply that the stalks of $\pi_{X!}(M_\chi)$ have Euler-Poincare characteristic zero. If $\pi_{X!}(M_\chi)$ is a perverse sheaf, this implies $\pi_{X!}(M_\chi)$ vanishes. By Theorem \ref{THM_VanSemiAb}, there exists a proper thin subset $\Delta \subset \cvar{S}$ such that
		\[
		\pi_{X!}(M_\chi) = 0
		\]
		for all $\chi \notin \Delta$.

		For each point $x \in X$, we consider the Fourier support
		\[
		\text{FSupp}(i_x^*M) := \{\chi \in \cvar{S}~|~ \pi_{Y!}(M_\chi)_x \neq 0\}^{\text{cl}}.
		\]
		By Theorem \ref{THM_StratificationFourier}, there exists a dense open subset $U \subseteq X$ such that $\pi_{X?}(M_\chi)$ is lisse over $U$ for all $\chi \in \cvar{S}$ and the formation of $\pi_{X?}(M_\chi)$ commutes with base change. By Lemma \ref{LEM_PervRelativePerversity}, if we shrink $U$ we can assume $i_x^*M[-c]$ is perverse for all closed points $x\in U$. For a closed point $x \in U$ we have
		\[
		\text{FSupp}(i_x^*M) = \text{FSupp}(i_\eta^*(M)).
		\]
		Note that
		\[
		\text{FSupp}(i_x^*M) \subseteq \Delta
		\]
		for all $x \in X$. By Lemma \ref{LEM_ClassificationFourierSupp}, the set $\text{FSupp}(i_\eta^*M)$ is a finite union of tacs. Thus there is a finite extension $k'/k,$ quotients $\pi_i\colon S_{k'} \rightarrow S'_i$ with connected fibers of dimension $d_i$, and characters $\chi_1, \ldots, \chi_n \in \widehat{S}(k')$ such that
		\[
		\text{FSupp}(i_\eta^*M ) = \bigcup_{i = 1}^n \chi_i\cdot \pi_i^*(\cvar{S'_i}) .
		\]
		Consider the complexes 
		\[
		K_i := \pi_{i!}(M_{\overline{\chi}_i}).
		\]
		There exists a dense open subset $V \subseteq X$ such that for a closed point $x \in V$ and  $1\leq i \leq n$, we have
		\[
		i_x^*\big(\pervCoh{d_i}{K_i}\big) = \pervCoh{d_i - c}{i_x^*(K_i)}
		\]
		by Lemma \ref{LEM_PervRelativePerversity}. Let $x \in V\cap U$ be a closed point. By Lemma \ref{LEM_ClassificationFourierSupp}, there exists $1\leq i \leq n$ such that
		\[
		\pervCoh{d_i - c}{i_x^*(K_i)} \neq 0.
		\]
		Thus
		\[
		\pervCoh{d_i}{K_i} \neq 0.
		\]
		In particular, there exists a non-zero perverse sheaf $N \in \Perv{}{S'_i\times X}$ and a surjective morphism by \cite[Cor.~4.2.6.2]{BBD}
		\[
		M_{k'} \rightarrow \pi_i^*N\otimes\SL_{\chi_i}.
		\]
		Since $M_{k'}$ is irreducible, this map is an isomorphism. We have
		\[
		\text{FSupp}(i_\eta^*M) \subseteq \chi_i\cdot\pi_i^*(\cvar{S'_i})
		\]
		by the projection formula in Proposition \ref{PROP_ProjectionFormula1}. Hence
		\[
		\text{FSupp}(i_\eta^*M)  = \chi_i\cdot\pi_i^*(\cvar{S'_i})
		\]
		The geometric irreducibility of $X$ implies that there is $N \gg 1$ such that $U(k_n) \neq \emptyset$ for all $n \geq N$. In particular, $\text{FSupp}(i_\eta^*(M))$ is invariant under $\text{Fr}_{k_n}$ for all $n\geq N$. Hence it is invarianted under $\text{Fr}_k$. By Theorem \ref{THM_DescendingTACS}, the tac $\text{FSupp}(i_\eta^*M)$ is defined over $k$. In particular, we can find a perverse sheaf $N \in \Perv{}{S'\times X} $ such that 
		\[
		M = \pi_i^*(N)\otimes\SL_{\overline{\chi}_i}
		\]
		by \ref{LEM_pullbackcrit}.  
		
		We now go on to prove the statement on the support.  Let $x \in U$. Then we have
		\[
		\text{FSupp}_!(i_x^*(N)) = \cvar{S'}.
		\]
		for all $x \in U$. In particular, we have $\chi(S, i_x^*(N')) \neq 0$ for all $x \in U$ by Lemma \ref{LEM_ClassificationFourierSupp}. The formation of $\pi_{X*}(N'_\chi)$ commutes with base change over $U$ for all $\chi \in \cvar{S'}$. This means we obtain $i_x^*(\pi_{X*}(N_\chi)) \neq 0$ for all $x \in U$ because the stalk of a point $x \in U$ of $\pi_{X*}(M_\chi)$ has Euler-Poincare characteristic equal to the Euler-Poincare characteristic of $i_x^*(M)$ by Proposition \ref{PROP_eulerpoincaresav}. To prove the vanishing we can apply Proposition \ref{PROP_ProjectionFormula1}.
	\end{proof}
	\begin{lemma}\label{LEM_classabsimplrelgrp}
		Let $S$ be a semiabelian variety over a finite field $k$. Suppose Theorem \ref{THM_ClassCharactersRelative} is true for $S$. Corollary \ref{COR_ClassificationRelativeGroup} is true for all surjective group morphisms $G \rightarrow G'$ with kernel isomorphic to $S$.
	\end{lemma}
	\begin{proof}
		We have a cartesian square
		\[
		\begin{tikzcd}
			S\times G \arrow[r, "\pi_{G}"] \arrow[d, "m"] & G \arrow[d, "\pi"] \\
			G \arrow[r, "\pi"]                            & G'.            
		\end{tikzcd}
		\]
		Let $d$ be the dimension of $S$. Note that we have $m^*(\SL_\nu) = \SL_{i^*(\nu)}\boxtimes\SL_{\nu}$ by Proposition \ref{PROP_CharKunneth}. Smooth base change implies
		\[
		[(\pi_{G*}(m^* (M_{i^*(\nu)})[d]))_\nu]= [\pi^*\pi_*(M_\nu)[d]] = 0.
		\]
		By Theorem \ref{THM_ClassCharactersRelative}, there is a non-trivial semiabelian subvariety $S' \subseteq S$ of dimension $d$ and a character $\chi \in \widehat{S'}(k)$ with the following property. Denote by $\pi'\colon G \rightarrow G'$ the quotient, then we can find another cartesian square
		\[
		\begin{tikzcd}
			S'\times G \arrow[r, "\pi'_{G}"] \arrow[d, "m"] & G \arrow[d, "\pi'"] \\
			G \arrow[r, "\pi'"]                            & G/S'.               
		\end{tikzcd}
		\]
		Define the lisse sheaf $\SL:= \SL_\chi\boxtimes \Qbarl$ on $S'\times G$. Then we have
		\[
		\pervCoh{d'}{\pi'_{G!}(m^*(M)\otimes \SL)[d]} \neq 0
		\]
		by proper base change. There exists a character $\chi' \in \widehat{G}(k)$ which restricts to $\chi$ on $S'(k)$ because the map $S'(k) \rightarrow G(k)$ is injective. The $t$-exactness of $\pi^*(-)[d]$ and smooth base change as above imply
		\[
		\pervCoh{d'}{\pi'_*(M_{\chi''})} \neq 0.
		\]
		The theorem follows from \ref{LEM_UniquenessOfInvariantSubgroup}.
	\end{proof}
	\subsubsection{Induction step}
	We turn to classifying the negligible sheaves on a semiabelian variety. We follow the argument in Gabber. We begin with the induction step:
	\begin{lemma}[{\cite[Lem.~5.2.1.]{GabberLoeserTore}}]\label{LEM_ClassificationInductionStep1}
		Let $M \in \Perv{}{S}$ an irreducible perverse sheaf with Euler-Poincare characteristic $\chi(S, M) = 0$. Exactly one of the following two statements is true\footnote{XOR}:
		\begin{enumerate}
			\item There is a quotient $\pi\colon S \rightarrow S'$ with connected fibers of dimension one, a character $\chi \in \widehat{S}(k)$, and a perverse sheaf $N \in \Perv{}{S'}$ with $\chi(S', N) \neq 0$ and 
			\[
			M = \pi^*N\otimes\SL_\chi[1].
			\]
			\item There is a closed subset $\Delta \subset \mathscr{C}(S)$ of codimension $\geq 2$, such that 
			\[
			H^*(S, M_\chi) = H^*_c(S, M_\chi) = 0
			\]
			for all $\chi \notin \Delta$. 
		\end{enumerate}
	\end{lemma}
	\begin{proof} 
		If (1) is true, then the projection formula in Proposition \ref{PROP_ProjectionFormula1} and Theorem \ref{THM_VanSemiAb} implies 
		\[
		H^*(S, M_{\overline{\chi}\pi_{}^*{(\chi})}) = H^*_c(S, M_{\overline{\chi}\pi^*(\chi)}) \neq 0
		\]
		for all $\chi' \in \mathscr{C}(S')$ outside a proper thin subset. In particular, there is a tac $\Delta \subset \mathscr{C}(S)$ of codimension $1$ such that the Fourier coefficients of $M$ for all characters outside a thin subset inside $\Delta$ do not vanish. This contradicts (2). Therefore, if (1) is true, then (2) is false.
		
		Suppose (2) is false. Theorem \ref{THM_VanSemiAb} implies that there is a finite union of tacs $\Delta \subset \mathscr{C}(S)$ with
		\[
		H^*(S, M_\chi) = H^*_c(S, M_\chi) = 0
		\]
		for all $\chi \notin \Delta$. Consider the set
		\begin{align*}
			Z &:= \{\chi \in \Delta : H^*(S, M_\chi) \neq 0\}.
		\end{align*}
		Note that this subset is the support of a perfect complex on $\mathscr{C}(S)$, so $Z$ is a closed subset inside $\mathscr{C}(S)$. Since (2) is false and $\Delta$ contains the support of the Fourier coefficients of $M$, the subset $Z$ can not be contained in a finite union of tacs of codimension $\geq 2$. We have the inclusion
		\[
		Z \subseteq \Delta.
		\]
		Therefore, $\Delta$ has an irreducible component $\Delta_0 \subseteq \Delta$ of codimension $1$ with 
		\begin{equation}\label{EQN_ZDelta}
		\Delta_0 \subseteq Z.
		\end{equation}		
		We can write the irreducible component of the tac of codimension $1$ as
		\[
 		\Delta_0 = \chi \cdot\pi_{S'}^*(\mathscr{C}^{\ell}(S')).
		\]
		for a quotient $S_{k'} \rightarrow S'$ with geometrically connected fibers of dimension $1$ and a character $\chi \in \mathscr{C}(S_k')$ over a finite extension $k'/k$. Consider
		\[
		M' := \pi_*(M_\chi).
		\]
		Suppose $M'$ is perverse. The vanishing theorems imply
		\[
		H^*(S', M'_\chi) = 0
		\]
		for all $\chi \in \mathscr{C}(S')$ outside a proper thin subset because $\chi(S', M')= 0$. This contradicts the inclusion (\ref{EQN_ZDelta}) by the projection formula in Proposition \ref{PROP_ProjectionFormula2}.
	\end{proof}
	Then induction step is based on the following theorem on existence of tac's avoiding certain bad loci.
	\begin{proposition}[{\cite[Lem.~5.2.5.]{GabberLoeserTore}}]\label{PROP_ClassificationAvoidance}
		Let $k$ be a finite field and $S$ a semiabelian variety over $k$ of dimension $d = d_a + d_t$. Let $\Delta \subset \mathscr{C}(S)$ be a closed subset of dimension $r < d$. Let $i\colon S' \rightarrow S$ be the inclusion of a closed, connected subgroup of dimension $d' = d'_a + d'_t$. For all but countably many $\chi \in \mathscr{C}(S')$, the intersection 
		\[
		Z\cap  (i^*)^{-1}(\chi)
		\]
		has dimension $$\leq \max(r - d'_t - 2d'_a, 0).$$
	\end{proposition}
	This remains true in our context, since the connected components of the character varieties satisfy the same geometric properties as the character variety considered in \cite{GabberLoeserTore}, up to computing the dimension. We have computed the dimension in Lemma \ref{LEM_ClosedDimensionTAC}.
	\begin{lemma}\label{LEM_ClassificationInductionStep2}
		Let $S$ be a semiabelian variety over an arithmetic field $k$ of dimension $d = d_a + d_t$ with $d_t > 0$. Let $M \in \derCat{c}{S}{\Qbarl}$ be an irreducible perverse sheaf on $S$ with $\chi(S, M) \neq 0$. Suppose there is a closed subset $\Delta\subset \mathscr{C}(S')$ of codimension $\geq 2$ such that 
		\[
		H^*(S, M_\chi) = H^*_c(S, M_\chi) =0
		\] 
		for all $\chi \notin \Delta$. There exists a finite extension $k'/k$ and a quotient $$\pi\colon S_{k'} \rightarrow \BG_m$$ with connected fibers and a character $\nu \in \mathscr{C}(S)$ such that $$\pi_{*}(M_\nu) = 0.$$
	\end{lemma}
	\begin{proof}
		By Corollary \ref{COR_AppendixMapToCircle}, there exists a finite extension $k'/k$ and a surjective map with connected fibers
		\[
		\pi\colon S_{k'} \rightarrow \BG_m.
		\]
		Let $S'$ be the kernel of this group morphism. By Proposition \ref{PROP_ClassificationAvoidance} and our assumption, there exists a character $\nu \in \mathscr{C}(S)$ such that $\chi\cdot \mathscr{C}(\BG_m)\cap \Delta$ is empty. Put
		\[
		M' := \pi_{*}(M_\nu).
		\]
		All Fourier coefficients of $M'$ vanish, therefore \cite[Prop.~3.4.5]{GabberLoeserTore} implies
		\[
		M' = 0.
		\]
	\end{proof}
	\begin{proof}[Proof (Theorem \ref{THM_Classification}).]
		We induct on the dimension $d_t$ of the maximal torus inside $S$. If the maximal torus is trivial, then $S$ is an abelian variety. This case is treated in Theorem \ref{THM_ClassificationNegligibleAbelian}. Suppose $S$ is a semiabelian variety such that $d_t > 0$ and we have proven Theorem \ref{THM_Classification} for all semiabelian varieties over finite extension of $k$ whose maximal torus has dimension $< d_t$. Let $M \in \Perv{}{S}$ be an irreducible perverse sheaf on $S$ with $\chi(S, M) = 0$. By Lemma \ref{LEM_ClassificationInductionStep1}, we can assume that the Fourier coefficients of $M$ vanish outside a closed subset of codimension $\geq 2$. By Lemma \ref{LEM_ClassificationInductionStep2}, there exists a finite extension $k'/k$, a quotient $\pi\colon S_{k'} \rightarrow \BG_m$ with connected kernel $S''$, and a character $\nu \in \mathscr{C}(S)$ such that
		\[
		\pi_*(M_\nu) = 0.
		\]
		By Lemma \ref{LEM_ClassificationAbsoluteImpliesRelative}, we can assume Theorem \ref{THM_ClassCharactersRelative} for the quotient morphism $\pi$. By Theorem \ref{THM_ClassCharactersRelative}, there exists a connected, non-trivial quotient $\pi_1\colon S_{k'} \rightarrow S'_1$ with connected fibers of dimension $d_1$, a character $\chi_1 \in \mathscr{C}(S)$, and a perverse sheaf $N_1 \in \Perv{}{S'}$ such that
		\[
		M = \pi^*N_1\otimes\SL_{\chi_1}[d_1].
		\]
		If $\chi(S', N_1) = 0$, then we can apply Theorem \ref{THM_Classification} to $N$ by the induction hypothesis to find a quotient $S'_1 \rightarrow S'_2$ with connected fibers, a perverse sheaf $N_2 \in \Perv{}{G}$, and a character $\chi_2 \in \mathscr{C}(S'_1)$ such that
		\[
		N_1 = \pi_2^*N_2\otimes\SL_{\chi_2}
		\]
		and
		\[
		\chi(S'_2, N_2) \neq 0.
		\]
		We put $\pi := \pi_2\circ\pi_1$ for the quotient, $N := N_2$ for the perverse sheaf, and $\chi := \chi_1\pi_1^*(\chi_2)$ for the character. This satisfies the assumption of Lemma \ref{LEM_UniquenessOfInvariantSubgroup}. 
	\end{proof}
	\subsection{General classification}
	We can now classify the negligible sheaves on an arbitrary connected commutative group. We only state the classification for a product of a semiabelian variety with a unipotent group.
	\begin{theorem}\label{THM_ClassificationFull}
		Let $G = S\times U$. Let $M \in \text{Perv}(G)$ be a geometrically irreducible perverse sheaf and put $M' := \text{FT}(M)$. Let $Z \subseteq \widehat{U}$ be the closure of the image of the support of $M'$ in $\widehat{U}$. There exists a tac $\Delta_S \subseteq \mathscr{C}(S)$ such that
		\begin{enumerate}
			\item We have
			\[
			\text{FSupp}_!(M) = \text{FSupp}_*(M) = \Delta_S\times Z.
			\]
			\item The tac $\Delta_S$ can be uniquely characterized by the following conditions. There exists a unique quotient $\pi\colon S \rightarrow S'$ with connected fibers of dimension $d$, a character $\chi \in \mathscr{C}(S)$, and a perverse sheaf $N \in \Perv{}{S\times U}$ such that
			\begin{enumerate}
				\item Let $\eta$ be the generic point of $\Delta_U$. The Euler-Poincare characteristic of $N' := \text{FT}(N)$ does not vanish on the generic fiber over $\Delta_U$, i.e.
				\[
				\chi(S', i_\eta^*N) \neq 0.
				\]
				\item We have
				\[
				M = \SL_{\chi}\otimes\pi^*{N}[d].
				\]
				where $d \in \BN$ is the relative dimension of $S \rightarrow S'$. 
				\item The tac is given by $\Delta_S = \chi\cdot \pi^*\mathscr{C}(S')$.
			\end{enumerate}
		\end{enumerate}
	\end{theorem}
	\begin{proof}
		 For all $\psi \in \widehat{U}$ and $\chi \in \cvar{S}$, we have by Lemma \ref{LEM_fmstalks}
		\begin{align*}
			i_\psi^*(\pi_{\widehat{U}!}(M'_\chi)) = H^*_c(G, M_{(\chi, \psi)})\\
			i_\psi^!(\pi_{\widehat{U}*}(M'_\chi)) = H^*(G, M_{(\chi, \psi)})
	\end{align*}
		This yields the inclusion
		\[
		\text{FSupp}_?(M) \subseteq \cvar{S}\times \Delta_U.
		\]
		
		We distinguish two cases. Let $\eta$ be the generic point of $\Delta_U$ and suppose
		\[
		\chi(S, i_\eta^*M') \neq 0
		\]
		By Theorem \ref{THM_StratificationFourier}, there is a dense open subset $V \subseteq \Delta_U$ such that $\pi_{\widehat{U}?}(M'_\chi)$ commutes with base change and is lisse over $V$ for all $\chi \in \cvar{S}$. Let $x \in V$ be a closed point. Lisseness of the Fourier coefficient with $\chi = 1$ implies
		\[
		\chi(S_\eta, i_\eta^*(M'))  = \chi(S, i_x^*(M')).
		\]
		For all $\psi \in V$ and all $\chi := (\chi', 0) \in \mathscr{C}(S)\times\mathscr{C}(U)$, we have an  isomorphism by base change
		\[
		i_\psi^*(\pi_{\widehat{U}?}(M'_\chi)) = H^*_?(G,  H^*_?(G, M_{(\chi', \psi)})).
		\]
		This cohomology complex is non-zero because base change and Proposition \ref{PROP_eulerpoincaresav} implies
		\begin{align*}
			\chi(i_\psi^*\pi_{\widehat{U}?}(M'_\chi)) &= \chi(i_\psi^*\pi_{\widehat{U}?}(M'_\chi))\\
			& = \chi(i_\eta^*\pi_{\widehat{U}?}(M'_\chi))\\
			&= \chi(S_\eta, i_\eta^*(M'_\chi)) \\&= \chi(S, i_x^*(M'_\chi)) \\&= \chi(S, i_x^*(M')) \\&\neq 0.
		\end{align*}
		Thus we obtain
		\[
		\text{FSupp}_?(M) = \cvar{S}\times\Delta_U
		\]
		in this case.
		
		If 
		\[
		\chi(S, i_\eta^*M') = 0,
		\]
		then we can find a character $\chi \in \mathscr{C}(S)$ such that
		\[
		\pi_{\widehat{U?}}(M_\chi) = 0
		\]
		over $V$. Theorem \ref{THM_ClassCharactersRelative} implies the existence of $N$ as in the Lemma. The projection formula in Proposition \ref{PROP_ProjectionFormula2} and Proposition \ref{PROP_ProjectionFormula1} and the argument in the first case, applied to $N$, imply the statement on the support. 
	\end{proof}
	From the classification, we obtain the following strengthening of the vanishing theorems. The following corollary is maybe the only reason for our classification of negligible sheaves.
	\begin{corollary}\label{COR_ClassificationWeaklyUnramifiedChars}
		Let $G = S\times U$ and $M \in \Perv{}{G}$ a geometrically irreducible perverse sheaf on $G$. Let $\Delta_S\times Z \subseteq \cvar{G}$ be the Fourier support of $M$. Let $2d_a + d_t$ be the codimension of the closed subset $\Delta_S$ and $d_u$ be the codimension of $Z$. There exists a a proper thin subset $\Delta'_S \subset \Delta_S$ and a proper closed subset $Z' \subset Z$ such that
		\[
		H^n_c(G, M_\chi) = \begin{cases}
			\neq 0& -d_a - d_u \leq n \leq d_a + d_t - d_u \\
			0 & \text{ else} 
		\end{cases}
		\]
		and
		\[
		H^n(G, M_\chi) = \begin{cases}
			\neq 0 & -d_a - d_t + d_u \leq n \leq d_a + d_u\\
			0 & \text{ else} 
		\end{cases}
		\]
		for all $\chi \in \Delta_S\times Z$ with $\pi_S^*(\chi) \notin \Delta'_S$ and $\pi_U^*(\chi) \notin Z'$. 
	\end{corollary}
	\begin{proof}
		After a finite extension of $k$, we can assume there exists a quotient $\pi\colon S \rightarrow S'$ and a character $\chi \in \widehat{S}(k)$ such that
		\[
		\Delta_S = \chi\cdot\pi_S^*(\cvar{S'}).
		\]
		Let $M' := \text{FT}(M)$. By Theorem \ref{THM_ClassificationFull}, there exists a perverse sheaf $N \in \Perv{}{S'\times \widehat{U}}$ such that $N$ has support concentrated over $Z$, we have $M' = \pi^*(N)\otimes\SL_\chi[d_a+d_t]$. Moreover, if $\eta$ is the generic point of $Z$, then $\chi(S', N_\eta) \neq 0$.
		
		Let $d$ be the dimension of $U$. By Theorem \ref{THM_StratificationFourier} and Proposition \ref{PROP_PervRelativePerversity}, there exists an open subset $V \subseteq Z$ such that the formation of the relative Fourier coefficient $\pi_{\widehat{U}?}(N_\chi)$ commutes with basechange $X \rightarrow V$, the relative Fourier coefficient is lisse over $V$, and for all points $x \in V$ the perverse sheaf $i_x^*(N)$ is concentrated in perverse degree $d_u - d$. 
		
		For all closed points $\psi \in V$, we have
		\[
		\chi(S, i_\psi^*(N)) = \chi(S, i_\eta^*(N)) \neq 0.
		\]
		The vanishing theorems and the liseness implies that there is a proper thin subset $\Delta' \subseteq \cvar{S'}$ such that
		\[
		H^*_c(S', i_\psi^*(N)_\chi)
		\]
		is non-zero precisely in degree $d - d_u$ and non-zero for all points $\psi \in V$ and all $\chi \notin \Delta'$. Proper base change implies
		\[
		H^*_c(S', i_\psi^*(N)_\chi) = H^*_c(S'\times U, \text{FT}(N)_{(\chi, \psi)})[d].
		\]
		Thus the cohomology group
		\[
		H^*_c(S'\times U, \text{FT}(N)_{(\chi, \psi)})
		\]
		is non-zero precisely in degree $-d_u$ for all $\chi \notin \Delta'$ and $\psi \in V$. We can now apply the projection formula from Proposition \ref{PROP_ProjectionFormula2} and Proposition \ref{PROP_ProjectionFormula1} to obtain
		\[
		H^*_c(S\times U, M_{(\pi^*(\chi), \psi)}) = H^*_c(S'\times U, \text{FT}(N)_{(\chi, \psi)})\otimes H^*_c(S'', \Qbarl)[d_a + d_t],
		\]
		where $S''$ is the kernel of $S \rightarrow S'$, for all $\chi \notin \Delta'$ and $\psi \in V$. We can finish the computation by Proposition \ref{PROP_SemiAbCohomology}.
		
		Poincare duality implies
		\[
		\text{FSupp}_*(M) = \text{FSupp}_!(\text{inv}^*D(M)).
		\]
		Since $\text{inv}^*D(M)$ is geometrically irreducible, Theorem \ref{THM_ClassificationFull} implies
		\[
		\text{FSupp}(M) = \text{FSupp}(\text{inv}^*D(M)).
		\]
		Thus we can derive the statement on cohomology without supports by applying the above argument to the dual $\text{inv}^*D(M)$. 
	\end{proof}


%% file: supports.tex
	\section{Supports} In this section, we study supports of perverse sheaves. We do this first in the absolute case over an arbitrary group. We introduce the Fourier support and establish its basic properties. These basic properties lay the foundation for the theory of $\Delta$-localization because they allow us to define a perverse $t$-structure and a duality theory on the localized categories.
	
	We then study supports in the relative setting. Here, we narrow our focus to tori and prove two propagation theorems. These propagation theorems are central to our construction of Tannakian categories because they will allow us to prove a certain criterion for a morphism to be an isomorphism.
	
	In this section, we begin deviating from the convention of the previous sections with respect to the base field. In particular, this section is written for groups over a finite field or the algebraic closure of a finite field $k$. Strictly speaking, we never consider an algebraic group over the algebraic closure of a finite field but only the base change of an algebraic group to $\overline{k}$ while implicitly remembering the group over $k$. This is so we obtain a well-defined notion of quasi-arithmetic complexes. 
	
	If the group $G$ is defined over the algebraic closure of a finite field, following our convention made at the beginning, the category $\derCat{c}{G}{}$ denotes the category of {quasi-arithmetic} complexes on $G$. Note that the irreducible constituents of a decomposition series of a perverse sheaf $M \in \Perv{}{G}$ descend to geometrically irreducible perverse sheaves defined over a finite field by definition of quasi-arithmeticity. This makes all the results of the previous section applicable to irreducible perverse sheaves. 
	\subsection{Fourier support} We now introduce the notion of Fourier support. The Fourier support plays a central role to $\Delta$-localization. The negligible complexes we introduce are defined by a condition on the Fourier support.
	\begin{definition}\label{DEF_fsupp}
		Let $K \in \derCat{c}{G}{\Qbarl}$ be a complex. We define the \textit{Fourier support} without supports, or the $*$-support, to be
		\[
		\text{FSupp}_*(K) := \{\chi \in \mathscr{C}(G)~|~ H^*(G, M_\chi) \neq 0\}^{\text{cl}}.
		\]
		We define the Fourier support with supports to be
		\[
		\text{FSupp}_!(K) := \{\chi \in \mathscr{C}(G)~|~ H^*_c(G, M_\chi) \neq 0\}^{\text{cl}}.
		\]
		We take the closure in the standard topology. We prove in Theorem \ref{THM_EqualitySupport} that these two supports agree. From then on, we will write $\text{FSupp}(K)$ to denote either support.
	\end{definition}
	We begin by recording a basic lemma on the Fourier support.
	\begin{lemma}\label{LEM_fsuppisogeny}
		Let $q\colon G \rightarrow G'$ be an isogeny and $K \in \derCat{c}{G}{}$. Then we have
		\[
		\text{FSupp}_?(K) = q^*(\text{FSupp}(q_*(K))).
		\]
		and
		\[
		(q^*)^{-1}(\text{FSupp}_?(K)) = \text{FSupp}(q_*(K)).
		\]
	\end{lemma}
	\begin{proof} For all $\chi \in \cvar{G'}$, we have
		\[
		H^*_?(G, q_*(K)_\chi) = H^*_?(G, K_{q^*(\chi)})
		\]
		by Proposition \ref{PROP_ProjectionFormula1}.The claim follows because $q^*$ is closed and surjective, i.e. Proposition  Proposition \ref{PROP_CharDescent} and Proposition \ref{PROP_CharIsogenyClosed}.		
	\end{proof}
	\subsubsection{Support lemma} We begin by proving the support lemma. For tori, the following theorem may also be found in the proof of \cite[Thm.~6.1.1]{GabberLoeserTore}
	\begin{theorem}\label{THM_SupportLemma}
		Let $K \in \derCat{c}{G}{\Qbarl}$. We have
		\[
		\text{FSupp}_? = \bigcup_{n \in \BZ} \text{FSupp}_?(\pervCoh{n}{K}).
		\]
		
		Let $M \in \text{Perv}(G)$ be a perverse sheaf with a fitration $$0 = M_0 \subseteq M_1 \subseteq \ldots \subseteq M_{n - 1} \subseteq M_n = M.$$ Then we have
		\[
		\text{FSupp}_?(M) = \bigcup_{i = 1}^n \text{FSupp}_?(M_i/M_{i - 1}).
		\]
	\end{theorem}
	\begin{proof} There exists an isogeny $q\colon G \rightarrow S\times U$ by Theorem \ref{THM_AppendixStructureTheorem}. By Lemma \ref{LEM_fsuppisogeny} and Proposition \ref{PROP_CharDescent}, we can assume $G = S\times U$. 
		Note that the Fourier support is unchanged by base change to the algebraic closure of $k$, hence we can assume $k$ is algebraically closed. Let $K \in \derCat{c}{G}{\Qbarl}$ be a complex. For each $n \in \BZ$, we put $M_n := \pervCoh{n}{M}$ and we write down a decomposition series
		\[
		0 = M_{n, 0} \subset M_{n, 1} \subset\ldots \subset M_{n, i_n - 1}\subset M_{n, i_n} = M_n.
		\] 
		By Jordan-Hölder, it is sufficient if we prove
		\begin{equation*}
			\text{FSupp}_?(K) = \bigcup_{n \in \BZ}\bigcup_{i = 1}^{i_{n}} \text{FSupp}_?(M_{n, i}/M_{n, i - 1}).
		\end{equation*}
		
		We can prove statements by induction on the number of factors $M_{n, i}$. It is set up as follows. Let $n \in \BZ$ be a maximal integer such that $\pervCoh{n}{K} \neq 0$. Then we have a distinguished triangle
		\[
		\pvTRC{< n} K \rightarrow K \rightarrow \pervCoh{n}{K}[-n] \xrightarrow{+}.
		\]
		We have a surjective map
		\[
		\pervCoh{n}{K} \rightarrow M_{n, i_n}/M_{n,i_n -1}.
		\]
		We can obtain a map $K \rightarrow M_{n, i}$ which induces this surjection of perverse sheaves in the $n$'th perverse degree. If we shift the triangle associated to the cone of this morphism, then we find a complex $K' \in \derCat{c}{G}{\Qbarl}$ which fits into a triangle
		\[
		K' \rightarrow K \rightarrow M_{n, i_n}/M_{n, i_n - 1}.
		\]
		The complex $K'$ has "one decomposition factor" less than $K$, so we could now proceed with induction.
		
		We apply this to inclusion $\subseteq$ in Equation (1). Let $\chi \in \mathscr{C}(G)$ satisfy
		\[
		H^*(G, (M_{n, i})_\chi) = 0
		\]
		for all factors $M_{n, i}$. If there is only one factor $M_{n, i}$, then $K = M_{n, i}[-n]$. Thus the support of $K$ is contained in the support of $M_{n, i}$. Suppose there is more than one factor. We have a distinguished triangle
		\[
		H^*_?(G, K'_\chi) \rightarrow H^*_?(G, K_\chi) \rightarrow H^*_?\big(G, (M_{n, i_n}/M_{n, i_n - 1})_\chi\big) \xrightarrow{+}.
		\]
		We have
		\[
		H^*_?(G, K'_\chi) = H^*_?(G, (M_{n, i})_\chi) = 0
		\]
		by our induction hypothesis, so 
		\[
		H^*_?(G, K_\chi) = 0.
		\]
		
		Note that the set on the right of the above equation is closed. Let 
		\[
		Z \subset \bigcup_{n \in \BZ}\bigcup_{i = 1}^{i_n} \text{FSupp}_?(M_{n, i}/M_{n, i - 1})
		\]
		be an irreducible component. We begin by recording the consequences of Corollary \ref{COR_ClassificationWeaklyUnramifiedChars}. Note that $Z$ is of the form $\Delta_S\times Z'$ for an irreducible closed subset $\Delta_S \subseteq\cvar{S}$ and an irreducible closed subset $Z' \subseteq \cvar{U}$. The subset $\Delta_S$ is an irreducible component of a tac. Let $(n, i)$ be an index with $i \geq 1$ and $M_{n, i}/M_{n, i - 1}\neq 0$. We distinguish two cases.
		
		Suppose $Z$ is not an irreducible component of the Fourier support of $M_{n, i}/M_{n, i - 1}$. The intersection of $Z$ with the Fourier support of $M_{n, i}/M_{n, i - 1}$ is a proper closed subset in $Z$, because no irreducible component of the Fourier support of $M_{n, i}/M_{n, i - 1}$ can contain $Z$ (because $Z$ is an irreducible component). Thus we obtain
		\[
		H^*_?(G,(M_{n, i}/M_{n, i - 1})_\chi) = 0
		\]
		for a generic character $\chi \in Z$, i.e. all characters in a dense open subset.
		
		Suppose $Z$ is an irreducible component of the Fourier support. The tac $\Delta_S$ is uniquely characterized by $Z$ by Proposition \ref{PROP_irredcomptequal}. Thus, Corollary \ref{COR_ClassificationWeaklyUnramifiedChars} shows that there exists an integer $d_? \geq 0$, which only depends on $Z$ by Lemma \ref{LEM_CharactersFindSubgroups}, with
		\begin{align*}
			H^{d_?}_?(G, (M_{n, i}/M_{n, i - 1})_\chi) &\neq 0\\
			H^{d_? + 1}_?(G, (M_{n, i}/M_{n, i - 1})_\chi) &= 0
	\end{align*}
	for a generic character $\chi \in Z$. Note that this is true for a generic character by Proposition \ref{PROP_irredcompinctacs}, because this proposition says that no irreducible component of the thin subset arising from Corollary \ref{COR_ClassificationWeaklyUnramifiedChars} can be an irreducible component of $\Delta_S$.

We return to the argument. Note that $Z$ is an irreducible component of a support $\text{Supp}(M_{n, i}/M_{n, i - 1})$ of one of the perverse sheaves $M_{n, i}/M_{n, i - 1}$. Let $(m, i)$ be the maximal index in the lexicographic ordering such that $Z \subset \text{Supp}(M_{m, i}/M_{m, i - 1})$. We have a distinguished triangle
		\[
		\pvTRC{\leq m} K \rightarrow K \rightarrow \pvTRC{> m}K \xrightarrow{+}.
		\]
		By the inclusion proven before and the remark on vanishing above, we have
		\[
		H^*_?(G, (\pvTRC{> m}K)_\chi) = 0
		\]
		for a generic character $\chi \in Z$. Hence
		\[
		H^*_?(G, (\pvTRC{\leq m}K)_\chi) = H^*_?(G, K_\chi)
		\]
		for a  generic character $\chi \in Z$. We have a distinguished triangle
		\[
		K' \rightarrow \pvTRC{\leq m} K \rightarrow M_{m, i_m}/M_{m, i}[-m] \xrightarrow[]{+}.
		\]
		The inclusion already proven implies, just as before,
		\[
		H^*_?(G, K'_\chi) = H^*_?(G, K_\chi)
		\]
		for a generic character $\chi \in Z$. We have a distinguished triangle
		\[
		K'' \rightarrow K' \rightarrow M_{m, i}/M_{m, i - 1}[-m]
		\]
		By the consequences drawn from Corollary \ref{COR_ClassificationWeaklyUnramifiedChars} above, there is an integer $d_? \geq 0$ such that
		\begin{align*}
			&H^{d_? + 1}_?(G, K''_\chi) = 0\\
			&H^{d_?}_?(G, (M_{m, i}/M_{m, i - 1})_\chi)  \neq 0.
		\end{align*}
		for a generic character $\chi \in Z$. The above triangle implies
		\[
		H^{d_?}_?(G, K'_\chi) \neq 0.
		\]
		for a generic character $\chi \in Z$. We have 
		\[
		H^{d_?}_?(G, K_\chi) = H^{d_?}_?(G, (\pvTRC{\leq m}K)_\chi) =  H^{d_?}_?(G, K'_\chi) \neq 0
		\]
		for a generic character $\chi \in Z$. Thus
		\[
		Z \subseteq \text{FSupp}_?(K).
		\]
	\end{proof}
	This has the following important corollary: 
	\begin{theorem}\label{THM_EqualitySupport}
		Let $K \in \derCat{c}{G}{\Qbarl}$. The Fourier supports 
		\[
		\text{FSupp}_!(K) = \text{FSupp}_*(K)
		\]
		agree. 
	\end{theorem} 
	\begin{proof}
		Note that the Fourier support is unchanged by base change to the algebraic closure of $k$, hence we can assume $k$ is algebraically closed. We can assume $G = S\times U$ by the same argument as in Theorem \ref{THM_SupportLemma}. We filter the complex $K$ as in the proof of Theorem \ref{THM_SupportLemma}. Theorem \ref{THM_SupportLemma} says
		\[
		\text{FSupp}_?(K) = \bigcup_{n \in \BZ}\bigcup_{i = 1}^{i_n} \text{FSupp}_?(M_{n, i}/M_{n, i - 1}).
		\]
		By Theorem \ref{THM_ClassificationFull}, we have
		\[
		\text{FSupp}_!(M_{n, i}/M_{n, i - 1}) = \text{FSupp}_*(M_{n, i}/M_{n, i - 1}).
		\]
		for all indices $(n, i)$. Thus
		\[
		\text{FSupp}_!(K) = \text{FSupp}_*(K).
		\]
	\end{proof}
	\begin{corollary}\label{COR_SupportDuality}
		Let $K \in \derCat{c}{G}{}$. We have
		\[
		\text{FSupp}(K)= \text{FSupp}(\text{inv}^*D(K)).
		\]
	\end{corollary}
	\begin{proof}
		This follows from Poincare duality and Theorem \ref{THM_EqualitySupport}.
	\end{proof}
	\subsection{Propagation}  We go on to prove the propagation theorems over tori. The idea for these theorems is based on \cite[Thm.~1.2]{liumaximwangpropagation}. Our first relative propagation theorem is the following.
	\begin{theorem}\label{THM_Propagation1}
		Let $M \in \Perv{}{T\times X}$ and $\Delta \subseteq \cvar{T}$ a closed subset. Suppose
		\[
		\pervCoh{0}{\pi_{X!}(M_\chi)} = 0 
		\]
		for all $\chi \notin \Delta$. Each perverse subquotient $M'$ of $M_{\overline{k}}$ satisfies
		\[
		\pi_{X?}(M'_\chi) = 0
		\]
		for all $\chi \notin \Delta$. 
		
		Dually, suppose
		\[
		\pervCoh{0}{\pi_{X*}(M_\chi)} = 0 
		\]
		for all $\chi \notin \Delta$. Each perverse subquotient $M'$ of $M_{\overline{k}}$ satisfies
		\[
		\pi_{X?}(M'_\chi) = 0
		\]
		for all $\chi \notin \Delta$. 
	\end{theorem}
	\begin{proof}
		The proofs of the two statements are dual to one another. We only give the proof for the first theorem. Note that the statement is unchanged by base change to the algebraic closure of $k$, hence we can assume $k$ is algebraically closed.
		
		We have a decomposition series
		\[
		0 = M_0 \subset M_1 \subset \ldots \subset M_n = M_{k'}
		\] 
		with $M_{i}/M_{i - 1}$ irreducible. By Theorem \ref{THM_VnshTorus}, there exists $\chi \notin \Delta$ such that $\pi_{X!}((M_i)_\chi)$ is perverse for $1 \leq i \leq n$. Hence $\pi_{X!}((M_i)_\chi) = 0$ for all $1 \leq i \leq n$.  Let $k$ be a finite field such that $M_i/M_{i - 1}$ descends to $k$ for all $1 \leq i \leq n$.
		
		By Theorem  \ref{THM_ClassCharactersRelative}, there are surjective morphisms $\pi_i\colon T \rightarrow T_i$ with connected fibers of dimension $d_i$, characters $\chi_i \in \widehat{T}(k)$, and perverse sheaves $N_i \in \Perv{}{S_i\times X}$ such that
		\[
		M_i/M_{i - 1} = \pi_i^*(N_i)\otimes\SL_{\chi_i}[d_i].
		\]
		We put $\Delta_i := \overline{\chi}_i\cdot \pi_i^*(\cvar{S_i})$. Moreover, we have 
		\[
		\pi_{X!}(M_\chi) \neq 0
		\]
		if and only if $\chi \in \Delta_i$. By Theorem \ref{THM_VanSemiAb} applied to $N_i$, there is a thin proper subset $\Delta'_i \subset \Delta_i$ such that $\pi_{X!}((N_i)_\chi)$ is perverse for all $1 \leq i \leq n$ and $\chi \in \cvar{T_i}$ with $\overline{\chi_i}\pi^*(\chi) \notin\Delta_i'$. In particular, the projection formula in Proposition \ref{PROP_ProjectionFormula1} and the cohomology computation in \ref{PROP_SemiAbCohomology} implies for all $\chi \in \Delta_i$ with $\chi \notin \Delta_i'$
		\[
		\pervCoh{n}{\pi_{X!}((M_i)_{\chi})} = \begin{cases}
			\neq 0 & \text{ if } 0\leq n \leq  d \\
			0 & \text{ else} .
		\end{cases}	
		\]
		Consider
		\[
		\Delta_0 := \bigcup_{i = 1}^n \Delta_i.
		\] 
		Suppose $\Delta_0 \not\subseteq \Delta$. There is an irreducible component of $\Delta_0$ which is not contained in $\Delta$. Let $\Delta'$ be such an irreducible component. Let $1 \leq i \leq n$ be the smallest integer such that $\Delta'$ is an irreducible component of $\Delta_i$. By Artin's vanishing theorem, we have an injective map
		\[
		\pervCoh{0}{\pi_{X!}((M_i)_\chi)} \hookrightarrow \pervCoh{0}{\pi_{X!}(M_\chi)}. 
		\]
		Hence
		\[
		\pervCoh{0}{\pi_{X!}((M_i)_\chi)} = 0
		\]
		for all $\chi \notin \Delta$. Suppose $\chi \in \Delta'$, $\chi \notin \Delta$, $\chi \notin \Delta_i'$, and $\chi \notin \Delta_j$ for any $1 \leq j \leq i - 1$. The evaluation of supports implies
		\[
		\pi_{X!}((M_i)_\chi) = 	\pi_{X!}((M_i/M_{i - 1})_\chi)
		\]
		because such a character does not lie in the support of $M_{j}/M_{j - 1}$ for $1 \leq j \leq i - 1$. Now we reach a contradiction because such a character exists: the character lies in $\Delta_i$ but not in $\Delta_i'$, hence it satisfies the above computation of the cohomology. In particular, the zeroth perverse cohomology sheaf does not vanish. 
	\end{proof}
	We can now utilize a spectral sequence and an induction to obtain another propagation theorem. 
	\begin{theorem}\label{THM_Propagation2}
		Let $M \in \Perv{}{T\times X}$ be a perverse sheaf. Suppose there is a closed subset $\Delta \subseteq \mathscr{C}(T)$ such that
		\[
		\pervCoh{1}{\pi_{X!}(M_\chi)} = 0 
		\]
		for all $\chi \notin \Delta$. Then 
		\[
		\pervCoh{n}{\pi_{X!}(M_\chi)} = 0 
		\]
		for all $n \geq 1$ and $\chi \notin \Delta$. 
		
		Dually, suppose 
		\[
		\pervCoh{-1}{\pi_{X*}(M_\chi)} = 0 
		\]
		for all $\chi \notin \Delta$. Then 
		\[
		\pervCoh{-n}{\pi_{X*}(M_\chi)} = 0 
		\]
		for all $n \geq 1$ and $\chi \notin \Delta$. 
	\end{theorem}
	\begin{proof}
		We prove the statement by induction on the dimension of the torus $T$. If the dimension is $\leq 1$, the statement follows from bounds on the cohomological amplitude of the pushforward.
		
		Suppose the dimension of $T$ is $> 1$. We can basechange to an extension of $k$ and assume $T$ is split. We take a quotient
		\[
		\pi\colon T \rightarrow T'
		\]
		with connected fibers of dimension $1$.  Let $\chi \notin \Delta$, then we put
		\[
		\Delta' :=  \cvar{T'} - (\pi^*)^{-1}(\overline{\chi}\cdot \Delta)
		\] then we consider the spectral sequence
		\[
		E_2^{pq} := \pervCoh{p}{\pi_{X!}(\pervCoh{q}{\pi_!(M_\chi)_{\chi'}})} \Rightarrow \pervCoh{p + q}{\pi_{X!}(M_{\chi\pi^*(\chi')})}
		\]
		for all $\chi' \notin \Delta'$. The term $E^{10}_2$ has converged by Artin's vanishing theorem, hence we obtain
		\[
		E^{10}_2 = 0.
		\]
		for all $\chi' \notin\Delta'$.  By our induction hypothesis, this implies $E^{p0}_2 = 0$ for all $p \geq 1$ and all $\chi' \notin \Delta$.  
		\[
		E^{01}_2 = 0
		\]  
		by assumption. By Theorem \ref{THM_Propagation1}, this implies $E^{p1} = 0$ for all $p \geq 0$ and all $\chi' \notin \Delta$. We obtain the statement from the above spectral sequence with $\chi' = 1$. 
	\end{proof}
\begin{remark}
	The proofs of the two propagation theorems work over arbitrary base fields by the results of \cite{GabberLoeserTore}. We do not prove this here because we would have to deviate too far from our conventions.
\end{remark}

%% file: TannakianCategories.tex
	\section{Tannakian categories} In this section, we construct the Tannakian categories with fiber functors. The general strategy is to imitate the construction of localization outside closed subsets from \cite{GabberLoeserTore}. We begin by introducing the localized categories. We prove a criterion for isomorphism. This criterion for isomorphism is based on the propagation theorems. We then use it to prove that the localized convolution product preserves perversity. 
	
	The following deviation from the notation is important for this section: we would like to introduce geometric and arithmetic Tannakian categories. In particular, we have to construct localizations on the derived category of complexes on a group $G$ and on the category of quasi-arithmetic complexes on the group $G_{\overline{k}}$. For this section, we use $G$ to denote a connected commutative group over a finite field or the algebraic closure of a finite field, as in the previous section. 
\subsection{$\Delta$-localization} We introduce the localizations. We remark that, ultimately, we want to prove that a perverse sheaf is generically unramified for the localizations introduced in \cite{KowalskiTannaka}. For this, we want to compare the Tannakian categories constructed here to the categories constructed in \cite{KowalskiTannaka}. Note that once this is achieved, we also gain access to all the results proven in \cite{KowalskiTannaka}. We approach this problem by allowing for localizations along families of closed subsets. The localization introduced in \cite{KowalskiTannaka} then arises naturally as a special case of our construction. This, however, means that we have to incorporate families of closed subsets from the very beginning of our theory of localized categories. 
\begin{definition}
	A \textit{family of closed subsets $\mathbf{\Delta}$} is a set of closed subsets in $\cvar{G}$ such that for all $\Delta, \Delta' \in \mathbf{\Delta}$, we have $\Delta\cup\Delta' \in \mathbf{\Delta}$. We say $\chi \notin \mathbf{\Delta}$ if and only if $\chi \notin \Delta$ for all $\Delta \in \mathbf{\Delta}$. 
\end{definition}
\begin{example}
	Let $\Delta \subseteq \cvar{G}$ be a closed subset. Then $\mathbf{\Delta} := \{\Delta\}$ is a family of closed subsets. We write $\Delta$ for this family of closed subsets. 
\end{example}
\begin{example}
	Let $\Delta \subseteq \cvar{G}$ be a subset. Then we can define the family of proper thin subsets contained in $\Delta$
	\[
	\mathbf{\Delta} := \{\Delta'\subset\cvar{G} ~|~ \Delta' \text{ thin and } \Delta' \subseteq \Delta\}
	\]
	Then $\mathbf{\Delta}$ is a family of closed subsets by Proposition \ref{PROP_thintop}. We denote the family of closed subsets arising from this construction from $\Delta := \cvar{G}$ by $\mathbf{\Delta}_{\text{FFK}}$.
\end{example}
\begin{definition}
	Let $\mathbf{\Delta}$ be a family of closed subsets. The category $S_{\mathbf{\Delta}}$ is defined to be the full subcategory of all perverse sheaves $M \in \Perv{}{G}$ such that there is $\Delta \in \mathbf{\Delta}$ with
	\[
	\text{FSupp}(M) \subseteq \Delta.
	\]
\end{definition}
\begin{proposition}
	Let $\mathbf{\Delta}$ be a family of closed subsets. We have that the subcategory $S_{\mathbf{\Delta}} \subseteq \Perv{}{G}$ is Serre.
\end{proposition}
\begin{proof}
	Let 
	\[
	0 \rightarrow M_1 \rightarrow M_2 \rightarrow M_3 \rightarrow 0
	\]
	be a short exact sequence of perverse sheaves $M_i \in \Perv{}{G}$. Theorem \ref{THM_SupportLemma} implies
	\[
	\text{FSupp}(M_2) = \text{FSupp}(M_1)\cup\text{FSupp}(M_3).
	\]
	Thus we obtain the claim. 
\end{proof}
\begin{definition}
	Ler $\mathbf{\Delta}$ be a family of closed subsets. We define the \textit{$\mathbf{\Delta}$-localized category of perverse sheaves} to be the Serre quotient 
	\[
	\text{Perv}_\Delta(G) := \text{Perv}(G)/S_\Delta.
	\]
	We define the full, saturated subcategory $N_\mathbf{\Delta} \subseteq \derCat{c}{G}{\Qbarl}$ to be the category where $K \in N_\mathbf{\Delta}$ if and only if $\pervCoh{n}{K} \in S_\mathbf{\Delta}$ for all $n \in \BZ$. We then define the \textit{$\mathbf{\Delta}$-localized derived category} to be the Verdier quotient
	\[
	\derCat{\mathbf{\Delta}}{G}{\Qbarl} := \derCat{c}{G}{\Qbarl}/N_\mathbf{\Delta}.
	\]
	It is a triangulated category with a $t$-structure (see \cite[Prop.~3.6.1]{GabberLoeserTore}). Loc. cit. implies that the heart of the $t$-structure is naturally equivalent to $P_\mathbf{\Delta}$. We write $\pervCoh{}{-}$ for the resulting cohomology functors. 
	
	Furthermore, we have that the quotient functor
	\[
	\derCat{c}{G}{\Qbarl} \rightarrow \derCat{\mathbf{\Delta}}{G}{\Qbarl}
	\]
	is $t$-exact. A complex $K \in N_\mathbf{\Delta}$ is called a \textit{$\mathbf{\Delta}$-negligible complex} or just \textit{negligible complex.}
	
	We also introduce a category of $\mathbf{\Delta}$-unramified perverse sheaves as follows. Let $M \in \Perv{\mathbf{\Delta}}{G}$. We say $M$ is \textit{$\mathbf{\Delta}$-unramified}, if for any lift\footnote{This means any $M' \in \Perv{}{G}$ such that $M' \cong M$ in $\Perv{\mathbf{\Delta}}{G}$.} $M'$ of $M$ to $\Perv{}{G}$ there exists a $\Delta \in \mathbf{\Delta}$ such that $M'$ is $\Delta$-unramified. 
	
	We say a perverse sheaf $M \in \Perv{\mathbf{\Delta}}{G}$ is \textit{strongly $\mathbf{\Delta}$-unramified} if all subquotients of $M_{\overline{k}}$ are $\mathbf{\Delta}$-unramified. We denote the category of $\mathbf{\Delta}$-unramified perverse sheaves by $\PervUnr{\mathbf{\Delta}}{G}$. Note that a perverse sheaf $M$ is strongly $\mathbf{\Delta}$-unramified if and only if all irreducible constituents in a decomposition series of $M_{\overline{k}}$ are $\mathbf{\Delta}$-unramified. The category $\PervUnr{\mathbf{\Delta}}{G}$ is abelian. 
\end{definition}
\begin{remark}
	It follows from Theorem \ref{THM_VanishingTHM} that any perverse sheaf is strongly $\mathbf{\Delta}_{\text{FFK}}$-unramified. 
\end{remark}
\begin{remark}
	Let $\mathbf{\Delta}$ be a family of closed subsets and $M \in \Perv{}{G}$ a perverse sheaf. If $M$ is pure or just $\iota$-pure, then $M$ is strongly $\mathbf{\Delta}$-unramified if and only if it is $\mathbf{\Delta}$ unramified. Indeed, when $M$ is $\iota$-pure, then $M$ is geometrically semisimple. 
\end{remark}
\begin{definition}
	Let $\mathbf{\Delta} \subseteq \cvar{G}$ be a family of closed subsets. We define the duality functor
	\[
	K^\vee := \text{inv}^*D(K)
	\] 
	for all $K \in \derCat{c}{G}{}$. By Corollary \ref{COR_SupportDuality}, the duality functor factors through the $\mathbf{\Delta}-$localization, hence we obtain a $t$-exact duality
	\[
	\derCat{\mathbf{\Delta}}{G}{\Qbarl}\rightarrow \derCat{\mathbf{\Delta}}{G}{\Qbarl},~K \mapsto K^\vee.
	\]
	The $t$-exactness implies that we also obtain an exact duality
	\[
	\Perv{\mathbf{\Delta}}{G} \rightarrow \Perv{\mathbf{\Delta}}{G}, ~ M \mapsto M^\vee.
	\]
\end{definition}
\begin{remark}
	Let $\Delta$ be a closed subset. Note that the functors $\FMQ{-}$ are well-defined because they vanish at all $\Delta$-negligible complexes. Moreover, note that the cohomology functors $H^*_?(G, -_\chi)$ are well-defined for all $\chi \notin \Delta$ because they vanish at all $\Delta$-negligible complexes. 
\end{remark}
Theorem \ref{THM_SupportLemma} now takes the following form.
\begin{theorem}\label{THM_DeltaLocNegligibleComplexes}
	Let $\Delta$ be a closed subset and $K \in \derCat{\Delta}{G}{\Qbarl}$ a complex. We have $K = 0$ if and only if  the following conditions are satisfied
	\begin{enumerate}
		\item We have $H^*_c(G, K_\chi) = 0$ for all $\chi \notin \Delta$.
		\item We have $H^*(G, K_\chi) = 0$ for all $\chi \notin \Delta$.
		\item We have $\pervCoh{n}{K} = 0$ for all $n \in \BZ$. 
		\item We have $\FMSH{M} = 0$ for all $\chi \notin \Delta$.
		\item We have $\FMST{M} = 0$ for all $\chi \notin \Delta$.
	\end{enumerate}
\end{theorem}
\begin{proof}
	The equivalences between (1) - (3) follow from Theorem \ref{THM_SupportLemma}. Then (5) and (6) are equivalent to (2) and (3) by Lemma \ref{LEM_fmstalks}.
\end{proof}
\subsubsection{Criteria for isomorphism}
\begin{theorem}\label{THM_IsoCrits}
	Let $G$ be a connected commutative algebraic group over a finite field $k$ and $\Delta \subseteq \mathscr{C}(G)$ a a closed subset. Let $K, L \in \derCat{\Delta}{G}{\Qbarl}$ be complexes on $G$ and $K \rightarrow L$ a morphism. The following are equivalent: 
	\begin{enumerate}
		\item The morphism $K \rightarrow L$ is an isomorphism in $\derCat{\Delta}{G}{}$. 
		\item The morphism
		\[
		H^*_c(G, K_\chi) \rightarrow H^*_c(G, L_\chi)
		\]
		induces isomorphisms for all $\chi \notin \Delta$.
		\item The morphism
		\[
		H^*(G, K_\chi) \rightarrow H^*(G, L_\chi)
		\]
		induces isomorphisms for all $\chi \notin \Delta$. 
		\item The morphism
		\[
		\pervCoh{n}{K} \rightarrow \pervCoh{n}{L}
		\]
		is an isomorphism in $\Perv{\Delta}{G}$ for all $n \in \BZ.$
		\item The morphism
		\[
		\FMSH{K} \rightarrow \FMSH{L}
		\]
		is an isomorphism for all $\chi \notin \Delta$. 
		\item The morphism
		\[
		\FMST{K} \rightarrow \FMST{L}
		\]
		is an isomorphism for all $\chi \notin \Delta$.
	\end{enumerate}
\end{theorem}
\begin{proof}
	Apply Theorem \ref{THM_DeltaLocNegligibleComplexes} to the cone of $K \rightarrow L$. 
\end{proof}

\subsubsection{Criteria from propagation}
We now utilize the propagation theorems to obtain improved criteria for nullity of perverse sheaves on affine groups.
\begin{theorem}\label{THM_PropagationFM1}
	Let $G = T\times U$ and $M \in \Perv{\Delta}{G}$. We consider a closed subset $\Delta \subseteq \cvar{G}$. Suppose
	\[
	\pervCoh{0}{\text{FM}_{\chi, \Delta, !}(M)} = 0 
	\]
	for all $\chi \notin \Delta$. Then $M = 0$. Dually, suppose
	\[
	\pervCoh{0}{\text{FM}_{\chi, \Delta,*}(M)} = 0 
	\]
	for all $\chi \notin \Delta$. Then $M = 0$. 
	
	In both cases, we obtain
	\[
	\text{FM}_{\chi, \Delta, ?}(M) = 0
	\]
	for all $\chi \notin \Delta$. 
\end{theorem}
\begin{proof}
	We assume 
	\[
	\pervCoh{0}{\text{FM}_{\chi, \Delta, !}(M)} = 0
	\]
	for all $\chi \notin \Delta$. The dual version can be obtained from the same argument. 
	
	We write $$\Delta = \Delta_1\times Z_1\cup \ldots \cup \Delta_n\times Z_n$$ for closed subsets $\Delta_i \subseteq \cvar{T}$ and closed subsets $Z_i \subseteq \cvar{U}$. Let $\psi \in \cvar{U}$. Define $$I_\psi := \{i ~|~\psi \in Z_i,~ 1 \leq i \leq n\}.$$ 
	We put $$\Delta_\psi := \bigcup_{i \in I_\psi} \Delta_i.$$ 
	Note that it is sufficient if we can prove
	\[
	i_\psi^*(\pi_{\widehat{U}!}(\text{FM}(M)_\chi)) = 0
	\]
	for all $\chi \notin \Delta_\psi$ by Theorem \ref{THM_DeltaLocNegligibleComplexes}. We put $$V_\psi := \bigcap_{i \notin I_\psi} Z_i^c.$$ We have $\psi \in V_\psi$. Moreover, note that this is an open subset of $\cvar{U}$, hence it is the set of closed points of an open subset of $\widehat{U}$. We write the Fourier-Mellin transform as
	\[
	\FMSH{M} = \pi_{\widehat{U}!}(\text{FM}(M)_\chi).
	\]
	We have for all $\chi \notin \Delta_\psi$
	\[
	\pervCoh{0}{\pi_{\widehat{U}!}(\text{FM}(M)_\chi)}|_{V_\psi} = 0
	\] 
	by assumption. Theorem \ref{THM_Propagation1} implies
	\[
	\pi_{\widehat{U}!}(\text{FM}(M)_\chi)|_{V_\psi} = 0
	\]
	for all $\chi \notin \Delta_\psi$. 
\end{proof}
We can also utilize the second propagation theorem in this more general context. 
\begin{theorem}\label{THM_PropagationFM2}
	Let $G = T\times U$ and $M \in \Perv{}{G}$. We consider a closed subset $\Delta \subseteq \cvar{G}$. Suppose
	\[
	\pervCoh{1}{\text{FM}_{\chi, \Delta, !}(M)} = 0 
	\]
	for all $\chi \in \mathscr{C}(G)$. Then 
	\[
	\pervCoh{n}{\text{FM}_{\chi, \Delta, !}(M)} = 0 
	\]
	for all $\chi \in \cvar{G}$ and $n \geq 1.$ 
	
	Dually, suppose
	\[
	\pervCoh{-1}{\text{FM}_{\chi, \Delta, !}(M)} = 0 
	\]
	for all $\chi \in \mathscr{C}(G)$. Then 
	\[
	\pervCoh{-n}{\text{FM}_{\chi, \Delta, !}(M)} = 0 
	\]
	for all $\chi \in \cvar{G}$ and $n \geq 1.$ 
\end{theorem}
\begin{proof}
	The proof is just as in Theorem \ref{THM_PropagationFM1} but by appealing to Theorem \ref{THM_Propagation2}.
\end{proof}
This yields the following crucial criterion for isomorphism.
\begin{theorem}\label{THM_LocalizationTheCriterion}
	Let $K, L \in \derCat{c}{G}{\Qbarl}$ be two complexes and $K \rightarrow L$ a morphism. Let $\Delta \subseteq \mathscr{C}(G)$ be a closed subset. Suppose the forget supports map
		\[
		\FMSH{K}  \rightarrow \FMST{L}
		\]
		is an isomorphism for all $\chi \in \mathscr{C}(G)$. The morphism $K \rightarrow L$ is an isomorphism in $\derCat{\Delta}{G}{\Qbarl}$. Moreover, all arrows in the square
	\[
		\begin{tikzcd}
		\FMSH{K} \arrow[d] \arrow[r] & \FMSH{L} \arrow[d] \\
		\FMST{K} \arrow[r]           & \FMST{L}          
	\end{tikzcd}
	\]
	are isomorphisms.
\end{theorem}
\begin{proof} We can find a proper map $G \rightarrow G'$ from $G$ to an affine group $G'$ by Theorem \ref{THM_AppendixStructureTheorem}. By Lemma \ref{LEM_fmproperpush}, the conditions are invariant under pushforward to $G'$. Thus we can reduce to $G'$ and assume $G = T\times U$ is affine.
	
	We consider the spectral sequences
	\begin{align*}
	E^{pq}_2 :=	\pervCoh{p}{\FMSH{\pervCoh{q}{K}}} &\Rightarrow \pervCoh{p + q}{\FMSH{K}}\\
	E'^{pq}_2 := \pervCoh{p}{\FMST{\pervCoh{q}{L}}} &\Rightarrow \pervCoh{p + q}{\FMST{L}}
	\end{align*}
	for all $\chi \notin \Delta$. Note that the isomorphism from (1) is induced by a morphism of spectral sequences $E^{pq}_2 \rightarrow E'^{pq}_2$. We argue by induction on the integer $n \in \BZ$ with $$\pervCoh{p}{\text{FM}_{\chi, \Delta, !}(\pervCoh{q}{K})} = 0$$ for all $q < n$ and $p \geq 1$, for the vanishing $\pervCoh{p}{\text{FM}_{\chi, \Delta, !}(\pervCoh{q}{K})} = 0$ for all $p \geq 1$ and $q \leq n$. Note that such an integer exists, because the complexes are bounded.
	
	Suppose $n \in \BZ$ satisfies the above property. Let $\chi \in \cvar{G}$. We denote by $F^{\cdot}(\pervCoh{n + 1}{\FMSH{K}})$ and $F^{\cdot}(\pervCoh{n + 1}{\FMST{L}})$ the filtrations induced by the above spectral sequences. We have a commutative square
	\[
	\begin{tikzcd}
		F^{1}(\pervCoh{n + 1}{\FMSH{K}}) \arrow[r, hook] \arrow[d] & \pervCoh{n + 1}{\FMSH{K}} \arrow[d, "\sim"] \\
		F^1(\pervCoh{n + 1}{\FMST{L}}) \arrow[r, hook]             & \pervCoh{n + 1}{\FMST{L}}                  
	\end{tikzcd}
	\]
	By the induction hypothesis, we can evaluate these parts of the filtration to be
	\[
	\begin{tikzcd}
		E_2^{1n} \arrow[r, hook] \arrow[d]  & \pervCoh{n + 1}{\FMSH{K}} \arrow[d, "\sim"] \\
		0 \arrow[r, hook]                                   & \pervCoh{n + 1}{\FMST{L}}                  
	\end{tikzcd}
	\]
	Thus $E_2^{1n} = 0$. By Theorem \ref{THM_PropagationFM2}, we have $E^{pn}_2 = 0$ for all $p \geq 1$ and $\chi \notin \Delta$. Hence we have completed the induction step.
	
	Thus we have proven that $\text{FM}_{\chi, \Delta, !}(\pervCoh{q}{K})$ is perverse for all $\chi \notin \Delta$ and $q \in \BZ$. Note that we can apply the same argument to the dual $L^\vee \rightarrow K^\vee$. We obtain that $\text{FM}_{\chi, \Delta, *}(\pervCoh{q}{L})$ is perverse for all $q \in \BZ$.  This means that the above spectral sequences collapse and the natural map induces isomorphisms
	\[
	\text{FM}_{\chi, \Delta, !}(\pervCoh{q}{K}) \xrightarrow[]{} \text{FM}_{\chi, \Delta, *}(\pervCoh{q}{L}).
	\]
	
	Thus we can assume $K$ and $L$ are perverse by Theorem \ref{THM_IsoCrits}. Suppose $M, N \in \Perv{}{G}$ are perverse sheaves with a morphism $M \rightarrow N$ such that the forget supports morphism
	\[
	\text{FM}_{\chi, \Delta, !}(M) \rightarrow \text{FM}_{\chi,\Delta,*}(N)
	\]
	is an isomorphism for all $\chi \in \mathscr{C}(G)$. Let $\ker$ be the kernel of $M \rightarrow N$.	We have a commutative diagram for all $\chi \notin \Delta$
	\[
	\begin{tikzcd}[column sep=small]
		0 \arrow[r] & \pervCoh{0}{\text{FM}_{\chi, \Delta, !}(\ker)} \arrow[r, hook] & \pervCoh{0}{\text{FM}_{\chi, \Delta, !}(M)} \arrow[r, hook] \arrow[d, "\sim"] & \pervCoh{0}{\text{FM}_{\chi, \Delta, !}(N)} \arrow[ld] \\
		&                                                & \pervCoh{0}{\text{FM}_{\chi, \Delta, *}(N)}                                   &                                          
	\end{tikzcd}
	\]
	The upper row is exact by Artin's vanishing theo	Thus we have proven that $\text{FM}_{\chi, \Delta, !}(\pervCoh{q}{K})$ is perverse for all $\chi \notin \Delta$ and $q \in \BZ$. Note that we can apply the same argument to the dual $L^\vee \rightarrow K^\vee$. We obtain that $\text{FM}_{\chi, \Delta, *}(\pervCoh{q}{L})$ is perverse for all $q \in \BZ$.  This means that the above spectral sequences collapse and the natural map induces isomorphismsrem. Thus we obtain
	\[
	\pervCoh{0}{\text{FM}_{\chi, \Delta, !}(\ker)} = 0.
	\]
	for all $\chi \notin \Delta$. This implies $\ker = 0$ in the $\Delta$-localized category by Theorem \ref{THM_PropagationFM1}. Dually, we can prove that the cokernel of $M \rightarrow N$ is zero. Therefore, the map $M \rightarrow N$ is an isomorphism in the $\Delta$-localized category.  
	
	We return to the general case and prove that all arrows in the square are isomorphisms. Recall the commutative square
	\[
	\begin{tikzcd}
		\FMSH{K} \arrow[d] \arrow[r] & \FMSH{L} \arrow[d] \\
		\FMST{K} \arrow[r]           & \FMST{L}          
	\end{tikzcd}
	\]
	By Theorem \ref{THM_IsoCrits}, the arrows to the right are isomorphisms. Since the diagonal morphism is an isomorphism by assumption, all arrows in this commutative diagram are isomorphism. 
\end{proof}
The spectral sequence argument yields further conclusions. We do not require this conclusion, but it would imply the existence of a Tannakian category without an appeal to the decomposition theorem for affine groups. We simply record these consequences here.
\begin{corollary}
	Let $G = T\times U$. Suppose $K_1 \rightarrow K_2$ is a morphism of complexes $K_1, K_2 \in \derCat{c}{G}{\Qbarl}$. Let $\Delta \subseteq \mathscr{C}(G)$ be a closed subset and $K \rightarrow L$ an isomorphism. If the natural forget supports morphism
	\[
	\FMSH{K_1} \rightarrow \FMST{K_2}
	\]
	is an isomorphism for all $\chi \notin \Delta$, then we have
	\begin{align*}
	\FMQ{\pervCoh{n}{K_i}} &= \pervCoh{n}{\FMQ{K_i}} 
	\end{align*}
	for all $n \in \BZ$, $i = 1, 2$, and $\chi \notin \Delta$. Moreover, the arrows in the square
	\[
	\begin{tikzcd}
		\FMSH{K_1} \arrow[d] \arrow[r] & \FMSH{K_2} \arrow[d] \\
		\FMST{K_1} \arrow[r]           & \FMST{K_2}          
	\end{tikzcd}
	\]
	are isomorphisms for all $\chi \notin \Delta$.
\end{corollary}
We also obtian the following corollary from the spectral sequence argument.
\begin{proposition}
	Let $G = T\times U$ and $\Delta \subseteq \cvar{G}$ a closed subset. Let $K \in \derCat{c}{G}{\Qbarl}$ and suppose at least one of the following conditions: 
	\begin{enumerate}[]\item The complex $K$ has perverse amplitude $\geq 0$ and $\FMSH{K}$ is perverse for all $\chi \notin \Delta$.
		\item The complex $K$ has perverse amplitude $\leq 0$ and $\FMST{K}$ is perverse for all $\chi \notin \Delta$.
	 \end{enumerate}
	 Then $K$ is perverse in $\derCat{\Delta}{G}{\Qbarl}$.
\end{proposition}
\begin{proof}
	We only consider (1), then (2) follows from duality. Consider the spectral sequence for all $\chi \notin \Delta$
	\[
	E^{pq}_2 :=	\pervCoh{p}{\FMSH{\pervCoh{q}{K}}} \Rightarrow \pervCoh{p + q}{\text{FM}_{\chi, \Delta, !}(K)}
	\]
	The term $E^{10}_2$ has converged, hence $E^{10}_2 = 0$. We obtain $E^{p0}_2 = 0$ for all $p \geq 1$ by Theorem \ref{THM_PropagationFM2}. Now we can prove $E^{pq}_2 = 0$ for all $p \geq 1, q \geq 0$ by Theorem \ref{THM_PropagationFM1} and induction on $p$. 
\end{proof}
\subsection{Convolution}\label{SUBSEC_Convolution} In this section we define the convolution functors and record the consequences of the above isomorphism theorems for these convolutions from our previously proven theorems. 
\begin{definition}
	Let  $\mathbf{\Delta} $ be a family of closed subsets. We denote by $m\colon G\times G\rightarrow G$ the multiplication map. Let $K, L \in \derCat{\mathbf{\Delta}}{G}{\Qbarl}$ be two complexes on $G$. We define the \textit{$!$-convolution}
	\[
	K*_!L := m_!(K\boxtimes L)
	\]
	and the \textit{$*$-convolution}
	\[
	K*_* L := m_*(K\boxtimes L).
	\]
	These functors are well-defined by the Künneth formula and Theorem \ref{THM_IsoCrits}.
\end{definition}
We now state the well-known properties of these sheaf convolutions. We begin with the most basic piece of information:
\begin{proposition}
	Let $\mathbf{\Delta} $ be a family of closed subsets. The $?$-convolution defines a bifunctor
	\[
	-*_?-\colon \derCat{\mathbf{\Delta} }{G}{\Qbarl}\times \derCat{\mathbf{\Delta} }{G}{\Qbarl} \rightarrow \derCat{\mathbf{\Delta} }{G}{\Qbarl}.
	\]
	The $?$-convolution is associative, commutative, and admits the unit $\delta_0$. The $?$-convolution equips the category of $\mathbf{\Delta}$-localized complexes with the structure of a tensor category.
\end{proposition}
We have internal Homs for the $!$-convolution:
\begin{proposition}
	Let $\mathbf{\Delta} $ be a family of closed subsets. Let $K, L \in \derCat{\mathbf{\Delta} }{G}{\Qbarl}$. We have an adjunction
	\begin{align*}
		\text{Hom}(K*_!L, M) = \text{Hom}(K, L^\vee*_*M).
	\end{align*}
	The natural morphism
	\[
	K*_!L \rightarrow K*_*L
	\]
	is naturally isomorphic to the natural morphism
	\[
	K*_! L \rightarrow \underline{\text{Hom}}(K, L^\vee),
	\]
	where the object on the right is the internal Hom object with respect to the $!$-convolution.
\end{proposition}
Let us also state a few formulas.
\begin{proposition}
	Let  $\mathbf{\Delta} $ be a family of closed subsets. Let $K, L \in \derCat{\mathbf{\Delta}}{G}{\Qbarl}$ and $\chi \in \cvar{G}$. Then we have 
	\begin{align*}
		(K*_!L)^\vee &= K^\vee*_*L^\vee \\
		(K*_*L)^\vee &= K^\vee*_!L^\vee
	\end{align*}
\end{proposition}
\begin{proposition}
	Let $G$ be a connected commutative algebraic group over a finite field $k$ and $\mathbf{\Delta}$ a family of closed subsets. For each character $\chi \in \cvar{G}$, the twisting functor 
	\[
	(-)_\chi\colon \derCat{\mathbf{\Delta}}{G}{\Qbarl} \rightarrow \derCat{\overline{\chi}\cdot\mathbf{\Delta} }{G}{\Qbarl},~K \mapsto K_\chi	
	\]
	is well-defined. It defines a tensor functor. In particular, there are isomorphisms
	\begin{align*}
	(K*_!L)_\chi &= K_\chi*_!L_\chi \\
	(K*_*L)_\chi &= K_\chi*_*L_\chi. 
	\end{align*}
\end{proposition}
We state the basic functoriality statement.
\begin{proposition}\label{PROP_tannakafunc}
	Let $\pi \colon G \rightarrow G'$ be a group morphism. We consider families of closed subsets $\mathbf{\Delta}  \subseteq \cvar{G}$ and $\mathbf{\Delta} ' \subseteq \cvar{G'}$ such that for each $\Delta_i \in \mathbf{\Delta} $ there exists $\Delta_j \in \mathbf{\Delta}'$ with  $\Delta')_j \subseteq (\pi^*)^{-1}({\Delta}_i)$ The functors
	\[
	\pi_?\colon \derCat{\mathbf{\Delta} }{G}{\Qbarl} \rightarrow \derCat{\mathbf{\Delta} '}{G'}{\Qbarl}
	\]
	are tensor functors. In particular, we have an isomorphism
	\[
	\pi_?(K*_?L) = \pi_?(K)*_?\pi_?(L).
	\]
	 for all $K, L \in \derCat{\Delta}{G}{\Qbarl}$ in the category $\derCat{\Delta'}{G'}{\Qbarl}$.
\end{proposition}
\begin{definition}
	Let $\mathbf{\Delta}$ be a family of closed subsets. We say $\chi \notin \mathbf{\Delta}$ if $\chi \notin\Delta_i$ for all $\Delta_i \in \mathbf{\Delta}$.
\end{definition}
We can now prove the crucial:
\begin{proposition}
	Let $\mathbf{\Delta}$ a family of closed subsets. Let $\chi \notin \mathbf{\Delta} $. The functor
	\[
	\omega_{?, \chi}(-)\colon \derCat{\Delta}{G}{\Qbarl} \rightarrow D^b_{\text{coh}}(\Qbarl),~K \mapsto H^*_?(G, K_\chi)
	\]
	is a tensor functor with respect to the $?$-convolution. Moreover, the two variants are interchanged by duality, i.e. there are isomorphisms
	\begin{align*}
		H^*_c(G, K_\chi)^\vee = H^*(G, (K^\vee)_\chi) \\
		H^*(G, K_\chi)^\vee = H^*_c(G, (K^\vee)_\chi).
	\end{align*}
	for all $K \in \derCat{\Delta}{G}{\Qbarl}$. 
\end{proposition}
This is a direct consequence of all the stated properties so far. We now come to the main theorem of this section where we prove that the localized convolution product preserves pervesity for unramified perverse sheaves. Before we prove perversity of the convolution, we require a lemma on semisimple objects in derived categories.
\begin{lemma}\label{LEM_ConvolutionSectionSemisimple}
	Let $D$ be a triangulated category with a bounded $t$-structure. Let $K, L \in D$ be objects in $D$ such that there is a map $r\colon K \rightarrow L$ which admits a section $s\colon L \rightarrow K$. Suppose
	\[
	L \cong \bigoplus_{n \in \BZ}H^n(L)[-n].
	\]
	Then
	\[
	K \cong \bigoplus_{n \in \BZ}H^n(K)[-n].
	\]
	Moreover, $H^n(K)$ is a direct factor of $H^n(L)$ for all $n \in \BZ$.
	
	Furthermore, suppose we are given a commutative square
	\[
\begin{tikzcd}
	K \arrow[d, "\eta"] \arrow[r, "r"] & L \arrow[d, "\eta"] \arrow[r, "s"] & K \arrow[d, "\eta"'] \\
	{K[2]} \arrow[r, "r"]              & {L[2]} \arrow[r, "s"]              & {K[2]}              
\end{tikzcd}
	\]
	We assume the morphisms
	\[
	\eta^i\colon H^{-i}(L) \rightarrow H^i(L)
	\]
	induced by $\eta$ are isomorphisms. Then the morphisms
	\[
	\eta^i\colon H^{-i}(K) \rightarrow H^i(K)
	\]
	are isomorphisms.
\end{lemma}
\begin{proof}
	Let $n \in \BZ$ be the largest integer with $H^{n}(K)\neq 0$. We have a commutative diagram
	\[
\begin{tikzcd}
	\tau^{\leq n - 1}(K) \arrow[d] \arrow[r, "r"] & \tau^{\leq n - 1}(L) \arrow[d, "0"] \arrow[r, "s"] & \tau^{\leq n - 1}(K) \arrow[d] \\
	{H^{n}(K)[-n]} \arrow[r, "r"]                 & {H^{n}(L)[-n]} \arrow[r, "s"]                      & {H^{n}(K)[-n]}                
\end{tikzcd}
	\]
	where the rows compose to the identity. Hence
	\[
	\tau^{\leq n}(K) \rightarrow H^{n}(K)[-n]
	\] 
	is zero. We have a morphism of distinguished triangles
	\[
	\begin{tikzcd}
		\tau^{\leq n - 1}(K) \arrow[r, "0"] \arrow[d, equal] & H^{n}({K})[-n] \arrow[d, equal] \ar[r] & {\tau^{\leq n - 1}(K)[1]\oplus H^{n}(K)[-n]} \arrow[d, "\sim"] \arrow[r, "+"] & {} \\
		\tau^{\leq n - 1}(K) \arrow[r]                & H^{n}(K)[-n] \arrow[r] & {K[1]} \arrow[r, "+"]                                                                            & {}
	\end{tikzcd}
	\]
	We can proceed by induction because we have the truncated maps
	\[
	\tau^{\leq n - 1}K \rightarrow \tau^{\leq n- 1} L \rightarrow \tau^{\leq n - 1}K.
	\]
	Hence
	\[
	K = \bigoplus_{i \in \BZ}H^i(K)[-i].
	\]
	
	To prove the second claim, we note that we have a commutative square
	\[
	\begin{tikzcd}
		H^{-i}(K) \arrow[d, "\eta^i"] \arrow[r, "r"] & H^{-i}(L) \arrow[d, "\eta^i"] \\
		H^i(K) \arrow[r, "r"]                        & H^i(L)                       
	\end{tikzcd}
	\]
	The arrow on the right is an isomorphism hence $\eta^i\colon H^{-i}(K) \rightarrow H^{i}(K)$ is injective. By considering the dual diagaram, we see that this morphism is surjective. Hence it is an isomorphism. 
\end{proof}

\begin{theorem}\label{THM_isoconv}	Let $\mathbf{\Delta}$ be a famliy of closed subsets. Let $M, N \in \Perv{\mathbf{\Delta}}{G}$ be strongly $\mathbf{\Delta}$-unramified perverse sheaves on $G$. The morphism
	\[
	M*_!N \rightarrow M*_* N
	\] 
	is an isomorphism in $\derCat{\Delta}{G}{}$. The localized convolution product $M*_!N$ is perverse in $\derCat{\mathbf{\Delta}}{G}{\Qbarl}$ and strongly $\mathbf{\Delta}$-unramified.
\end{theorem}
\begin{proof} By definition, we can assume there exists $\Delta \in \mathbf{\Delta}$ such that $M$ and $N$ are $\Delta$-unramified. Note that we are free to base change to the algebraic closure of $k$. By filtering the sheaves $M$ and $N$ by a decomposition series, we can assume $M$ and $N$ are irreducible. The formula for the convolution recorded in Corollary \ref{COR_ftconv} implies that the forget supports map
	\[
	\FMSH{M*_!N} \rightarrow \FMST{M*_*N}
	\]
	is isomorphic to
	\[
	\FMSH{M}\otimes\FMSH{N}[-d] \rightarrow \FMST{M}\otimes\FMST{N}[-d].
	\]
	for all $\chi \notin \Delta$. In particular, it is an isomorphism for all $\chi \notin \Delta$ because $M$ and $N$ are $\Delta$-unramified. Therefore, Theorem \ref{THM_LocalizationTheCriterion} implies
	\[
	M*_!N \rightarrow M*_*N 
	\]
	is an isomorphism in $\derCat{\Delta}{G}{}$. 
	
	We construct a convolution which is certainly semisimple as follows: let $$j\colon G\times G \rightarrow X$$ be an open immersion such that the multiplication factors through a proper map $\overline{m}\colon X \rightarrow G$. Then we form the complex
	\[
	K := \overline{m}_*(j_{!*}(M\boxtimes N)).
	\]
	By the decomposition theorem, the complex $K$ is geometrically semisimple. The composite
	\[
	M*_!N \rightarrow K \rightarrow M*_*N.
	\]
	is an isomorphism in $\derCat{\Delta}{G}{\Qbarl}$. In particular, the map
	\[
	M*_!N \rightarrow K
	\]
	admits a section in $\derCat{\Delta}{G}{\Qbarl}$. Lemma \ref{LEM_ConvolutionSectionSemisimple} shows that
	\[
	M*_!N \cong \bigoplus_{n \in \BZ} \pervCoh{n}{M*_!N}[-n]
	\]
	in $\derCat{\Delta}{G}{\Qbarl}$. This implies
	\[
	\pervCoh{m}{\FMSH{\pervCoh{n}{M*_!N}}} = 0
	\]
	for $m \neq -n$ and all $\chi \notin \Delta$. The hard Lefschetz theorem and Lemma \ref{LEM_ConvolutionSectionSemisimple} imply
	\[
	\pervCoh{-n}{M*_!N} = \pervCoh{n}{M*_!N}(n)
	\]
	Thus we obtain
	\[
	\pervCoh{n}{M*_!N} = 0.
	\]
	for all $n \neq 0$ by Theorem \ref{THM_DeltaLocNegligibleComplexes} in the $\Delta$-localization. Thus $M*_!N$ is perverse in the $\Delta$-localized category. Furthermore, Theorem \ref{THM_LocalizationTheCriterion} also yields that the morphism
	\[
	\FMSH{M*_!N} \rightarrow \FMST{M*_!N}
	\]
	is an isomorphism. We have
	\[
	\FMQ{M*_!N} = \FMSH{M}\otimes\FMSH{N}[-d]
	\]
	as above. Therefore, $\FMQ{M*_!N}$ is lisse and perverse. Since $M*_!N$ is geometrically semisimiple, all perverse subquotients of $M*_!N$ are  $\Delta$-unramified by Proposition \ref{PROP_fmcritunr}.
\end{proof}
\subsection{Localized Tannakian categories} The goal of this section is to state the existence theorem. We construct the structure of a Tannkian category on the category of $\mathbf{\Delta}$-unramified perverse sheaes. 
We state the existence theorem.
\begin{theorem}\label{THM_tannakaexistence}
	Let $\mathbf{\Delta}$ be a family of closed subsets such that $\cvar{G} \notin \mathbf{\Delta}$. The localized $!$-convolution factors through to a convolution product
	\[
	-*_{\text{loc}}-\colon \Perv{\mathbf{\Delta}\text{-unr}}{G}\times \Perv{\mathbf{\Delta}\text{-unr}}{G} \rightarrow \Perv{\mathbf{\Delta}\text{-unr}}{G}.
	\]
	This convolution product equips the category $\Perv{\mathbf{\Delta}\text{-unr}}{G}$ with the structure of a neutral rigid Tannakian category. For each $\chi \notin \mathbf{\Delta}$, the functor
	\[
	\omega_\chi(-)\colon \Perv{\mathbf{\Delta}\text{-unr}}{G} \rightarrow \Qbarl\text{-Vec}
	\]
	is a fiber functor. 
\end{theorem}
\begin{proof}
	We only prove that $\omega_\chi$ is faithful. Let $M \in \PervUnr{\mathbf{\Delta}}{G}$ be a perverse object and $\chi \notin \mathbf{\Delta}$ a character with
	\[
	\omega_\chi(M) = 0.
	\]
	We can assume $M$ is $\Delta$-unramified for a $\Delta \in \mathbf{\Delta}$. Proposition \ref{PROP_GysinUnramifiedCriterion} implies 
	\[
	H^*_c(G, M_\chi) = 0
	\]
	for all $\chi \notin \Delta$. Theorem \ref{THM_DeltaLocNegligibleComplexes} implies $M = 0$ in $\text{Perv}_{\mathbf{\Delta}}(G)$, hence the claim.
	
	To prove that the category is neutral in full generality, one can proceed as in \cite[Prop.~3.14]{KowalskiTannaka}.
\end{proof}
\subsection{Internal Tannakian categories} We now state the existence theorem for the internal Tannakian categories. These Tannakian categories are constructed using a certain equivalence.
\begin{definition}
	Let $\mathbf{\Delta}$ be a family of closed subsets. The category $\PervInt{\mathbf{\Delta}}{G}$ is defined to be the full subcategory of strongly $\mathbf{\Delta}$-unramified perverse sheaves in $\Perv{}{G}$ such that no subobject and no quotient has proper Fourier support.
\end{definition}
We require a technical lemma.
\begin{lemma}\label{LEM_unramifiednegligible}
	Let $\Delta \subset \cvar{G}$ be a closed subset. Let $M \in \Perv{}{G}$ be strongly $\Delta$-unramified. Suppose $\text{FSupp}(M)\neq \cvar{G}$. Then $\text{FSupp}(M) \subseteq \Delta$.
\end{lemma}
\begin{proof}
	We can reduce to $G = S\times U $ by Lemma \ref{LEM_fsuppisogeny}, Proposition \ref{PROP_RamifiedFiniteDescent}, Proposition \ref{PROP_CharIsogenyClosed}, and Theorem \ref{THM_AppendixStructureTheorem}. By Theorem \ref{THM_SupportLemma}, we are free to assume $M$ is geometrically irreducible.
	
	For each $\chi \notin \Delta$, the cohomology 
	\[
	H^*(G, M_\chi)
	\]
	is concentrated in degree zero by Corollary \ref{COR_GysinUnrImpliesWunr}. By Corollary \ref{COR_ClassificationWeaklyUnramifiedChars}, a generic character $\chi \in \text{FSupp}(M)\cap \Delta^c$ has cohomology not supported in degree zero because $M$ is negligible. Thus we obtain $\text{FSupp}(M)\subseteq \Delta$. 
\end{proof}
\begin{proposition}\label{PROP_intermedconv}
	Let $\mathbf{\Delta}$ be a family of closed subsets. There exists an equivalence of categories
	\[
	\text{int}\colon \PervUnr{\mathbf{\Delta}}{G} \rightarrow \PervInt{\mathbf{\Delta}}{G}, ~ M \mapsto M^{\text{int}}
	\]
	If $\cvar{G} \notin\mathbf{\Delta}$, we declare this functor to be an equivalence of Tannakian categories. The resulting convolution product
	\[
	-* - \colon \PervInt{\mathbf{\Delta}}{G}\times \PervInt{\mathbf{\Delta}}{G} \rightarrow \PervInt{\mathbf{\Delta}}{G}.
	\]
	is called the \textit{intermediate convolution}.
\end{proposition}
\begin{proof}
	We define the functor as follows. Let $M \in \Perv{\mathbf{\Delta}}{G}$. By definition of the Serre localization, there exists an object $M \in \Perv{}{G}$ corresponding to $M$. For each $N \in \Perv{}{G}$, we define $N^t$ to be the largest quotient of $M$ with negligible kernel and $N_t$ the largest subobject of $M$. We put
	\[
	M^{\text{int}} := (M_t)^t.
	\]
	Lemma \ref{LEM_unramifiednegligible} implies that we have an isomorphism
	\[
	M \cong M^{\text{int}}
	\]
	in $\PervUnr{\mathbf{\Delta}}{G}$. Moreover, the definition of morphism sets in Serre localizations implies
	\[
	\text{Hom}_{\PervUnr{\mathbf{\Delta}}{G}}(M^{\text{int}}, N^{\text{int}}) = \text{Hom}_{\PervInt{\mathbf{\Delta}}{G}}(M^{\text{int}}, N^{\text{int}})
	\]
	Thus we obtain an equivalence.
\end{proof}
\subsection{Comparison} Note that our definition of intermediate convolution at least formally depends on the choice of a family of closed subsets. The goal of this section is to show that the intermediate convolution does not depend on the family of closed subsets.
\begin{definition}
	Let $\mathbf{\Delta}$ and $\mathbf{\Delta}'$ be two families of closed subsets. We write $\mathbf{\Delta} \subseteq\mathbf{\Delta}'$ if for all $\Delta \in \mathbf{\Delta}$ there exists $\Delta' \in \mathbf{\Delta}$ with $\Delta\subseteq\Delta'$.
\end{definition}
\begin{remark}
	Suppose $\mathbf{\Delta} \subseteq\mathbf{\Delta}'$. Proposition \ref{PROP_tannakafunc} applied to the identity morphism constructs a natural tensor functor
	\[
	\alpha\colon \PervUnr{\mathbf{\Delta}}{G} \rightarrow \PervUnr{\mathbf{\Delta}'}{G}.
	\] 
\end{remark}
Now we come to the statement, which will allow us to prove basic properties of the intermediate convolution.
\begin{proposition}\label{PROP_compindlimit}
	Suppose $\mathbf{\Delta} \subseteq\mathbf{\Delta}'$. We have a commutative square of functors
	\[
	\begin{tikzcd}
		\PervUnr{\mathbf{\Delta}}{G} \arrow[d, "\text{int}"] \arrow[r, "\alpha"] & \PervUnr{\mathbf{\Delta}'}{G} \arrow[d, "\text{int}"] \\
		\PervInt{\mathbf{\Delta}}{G} \arrow[r, "\text{inc}"]           & \PervInt{\mathbf{\Delta}'}{G}                        
	\end{tikzcd}
	\]
	If $\cvar{G} \notin \mathbf{\Delta}'$, then the fully faithful inclusion functor $\text{inc}$ is a tensor functor.
\end{proposition}
\begin{proof}
	This commutativity follows from the definition of the $\text{int}$-functor. The functor $\text{inc}$ is a tensor functor, because the functors $\text{int}$ are equivalences and $\alpha$ is a tensor functor. 
\end{proof}
\begin{proposition}\label{PROP_tannakacomp}
	Suppose $\mathbf{\Delta} \subseteq\mathbf{\Delta}'$ and $\cvar{G} \notin \mathbf{\Delta}'$. Let $M \in \PervInt{\mathbf{\Delta}}{G}$. We denote by $\langle M \rangle_{\mathbf{\Delta}}$ the Tannakian category generated by $M$ in $\PervInt{\mathbf{\Delta}}{G}$ and by $\langle M \rangle_{\mathbf{\Delta}'}$ the Tannakian category generated by $M$ in $\PervInt{\mathbf{\Delta}'}{G}$. The natural functor
	\[
	\text{inc}\colon \langle M \rangle_{\mathbf{\Delta}} \rightarrow \langle M \rangle_{\mathbf{\Delta}'}
	\]
	is an equivalence.
\end{proposition}
\begin{proof}
	By Proposition \ref{PROP_compindlimit}, the functor $\text{inc}$ is a tensor functor between neutral Tannakian categories. The associated morphism of groups is a closed immersion because both categories are generated by $M$. The tensor functor is fully faithful by Proposition \ref{PROP_compindlimit}. Let $N \in  \langle M \rangle_{\mathbf{\Delta}}$ and suppose $N' \subseteq \text{inc}(N)$ is a subobject. Then $N'$ is a subobject of the strongly $\mathbf{\Delta}$-unramified object $N$ (by basic properties of the Serre localization). Thus $N'$ is strongly $\mathbf{\Delta}$-unramified, i.e. $N' \in \PervInt{\mathbf{\Delta}}{G}$. This implies $N'$ is a subobject of $N$ itself. Thus we obtain that the morphism is an isomorphism of algebraic groups. Hence the above functor is an equivalence. 
\end{proof}
\newpage
\section{Tannakian monodromy groups} In this section, we attach Tannakian monodromy groups to perverse sheaves on a group $G$. We again work exclusively with groups $G$ defined over a finite field $k$.
\subsection{Geometric and arithmetic Tannakian monodromy groups} The goal of this section is to attach a geometric and an arithmetic Tannakian monodromy group to a perverse sheaf on a group $G$.
\begin{definition}
	We define the category $\Perv{\text{int}}{G}$ to be the full subcategory of objects $M \in \text{Perv}(G)$ such that $M$ has no quotient or subobject which is negligible. 
\end{definition}
\begin{remark}\label{RMK_constructionproduct}
	Note that we have an equality
	\[
	\Perv{\text{int}}{G}  = \PervUnr{\mathbf{\Delta}_{\text{FFK}}}{G}.
	\]
	by Theorem \ref{THM_VanishingTHM} and Theorem \ref{THM_ClassificationFull}. The category $\Perv{\text{int}}{G}$ is equal to the Tannakian category $\mathbf{P}_{\text{int}}(G)$ considered in \cite{KowalskiTannaka} by Theorem \ref{THM_ClassificationFull}.
	
	Thus we have the structure of a neutral Tannakian category on $\Perv{\text{int}}{G}$. The tensor product 
	\[
	-*-\colon \Perv{\text{int}}{G} \times\Perv{\text{int}}{G} \rightarrow \Perv{\text{int}}{G} 
	\]
	is called the \textit{intermediate convolution}. The duality functor is given by $M \mapsto M^\vee$. This convolution product  is equal to the convolution product $-*_{\text{int}}-$ constructed in \cite{KowalskiTannaka} by definition.
\end{remark}
\begin{definition}
	Let $M \in 	\Perv{\text{int}}{G}$. We define the \textit{arithmetic Tannakian monodromy group $\mathbf{G}_M^{\text{ari}}$} to be the automorphism group of a fiber functor of the Tannakian category $\langle M \rangle$ generated by $M$. We define the \textit{geometric Tannakian monodromy group $\mathbf{G}_M^{\text{geo}}$} to be the automorphism group of a fiber functor of $\langle M_{\overline{k}} \rangle$. 
\end{definition}
\begin{remark}
	Remark \ref{RMK_constructionproduct} implies that the Tannakian monodromy groups considered here are equal to the Tannakian monodromy groups considered in \cite{KowalskiTannaka}.
\end{remark}
\begin{theorem}\label{THM_fiberfunc}
	Let $\Delta \subseteq \cvar{G}$ be a closed subset and $M \in \Perv{\text{int}}{G}$. Suppose $M$ is strongly $\Delta$-unramified. \footnote{This means all subquotients of $M_{\overline{k}}$ are $\Delta$-unramified. For example, if $M$ is $\iota$-pure, then strongly $\Delta$-unramified and $\Delta$-unramified are equivalent.}For all $\chi \notin \Delta$, we have fiber functors given by
	\[
	\omega_\chi\colon \langle M \rangle \rightarrow \Qbarl\text{-Vec},~N \mapsto H^0_c(G, M_\chi).
	\]
	Moreover, all objects in the category $\langle M \rangle$ are strongly $\Delta$-unramified.
\end{theorem}
\begin{proof}
	When $\Delta = \cvar{G}$, there is nothing to prove. Otherwise, this follows from Proposition \ref{PROP_tannakacomp}, Proposition \ref{PROP_intermedconv}, and Theorem \ref{THM_tannakaexistence}. 
\end{proof}
\begin{definition}
	Let $\Delta \subseteq \cvar{G}$ be a closed subset and $M \in \Perv{\text{int}}{G}$ a strongly $\Delta$-unramified perverse sheaf. For each arithmetic character $\chi \notin \Delta$, Theorem \ref{THM_fiberfunc} constructs a \textit{Frobenius conjugacy class} $\theta_{M, \chi}$ in $\mathbf{G}^{\text{ari}}_M$ by considering the conjugacy class in $\mathbf{G}^{\text{ari}}_M$ associated to Frobenius acting on the fiber functor $\omega_\chi$. 
\end{definition}
\begin{remark}
	We retain the notation of the above definition. The conjugacy classes $\theta_{M, \chi}$ can be characterized by the following property. Let $N \in \langle M \rangle$ be an object in the Tannakian category generated by $M$. The conjugacy class $\theta_{M, \chi}$ naturally acts on $H^0_c(G, N_\chi)$ because this is a fiber functor. The Frobenius conjugacy class acting on $H^0_c(G, N_\chi)$ agrees with $\theta_{M, \chi}$.
	
	The trace of this operator can be evaluated from the following remark. Note that all the objects in the category $\langle M \rangle$ are strongly $\Delta$-unramifed. Thus we obtain
	\[
	\text{Tr}(\text{Fr}_{k_n} | H^0_c(G, N_\chi)) = \sum_{i \in \BZ} (-1)^i\text{Tr}(\text{Fr}_{k_n} | H^i_c(G, N_\chi)).
	\]
	The right hand side of this equation can be rewritten using the Grothendieck-Lefschetz fixed point formula. 
\end{remark}
\begin{lemma}\label{LEM_intconvform}
	Let $M, N \in \Perv{\text{int}}{G}$. Then we have
	\[
	M*N = \pervCoh{0}{M*_!N}^{\text{int}} = \pervCoh{0}{M*_*N}^{\text{int}}
	\]
\end{lemma}
\begin{proof}
	This follows from the definition of the convolution and Theorem \ref{THM_isoconv}.
\end{proof}
\begin{proposition}
	Let $\Delta \subseteq \cvar{G}$ be a closed subset. Suppose $M \in \Perv{\text{int}}{G}$ is $\iota$-pure of weight zero. All objects $N \in \langle M \rangle$ are $\iota$-pure of weight zero.
\end{proposition}
\begin{proof}
	Note that any object in $\langle M \rangle$ is a subquotient of
	\[
	M^{*a}*(M^\vee)^{*b}
	\]
	for some $a, b \geq 1$. We have
	\[
	M^{*a}*(M^\vee)^{*b} = \pervCoh{0}{	M^{*_?a}*_?(M^\vee)^{*_?b}}^{\text{int}}
	\]
	by Lemma \ref{LEM_intconvform}. Therefore, the generalized Riemann hypothesis implies that all subquotients of $M^{*a}*(M^\vee)^{*b}$ are pure of weight zero.
\end{proof}
\begin{corollary}
	Let $\Delta \subseteq \cvar{G}$ be a closed subset. Suppose $M \in \Perv{\text{int}}{G}$ is $\Delta$-unramified arithmetically semisimple, and $\iota$-pure of weight zero. Let $K\subset \mathbf{G}^{\text{ari}}_M\otimes_\iota\BC$ be a maximal compact subgroup. The semisimplified conjugacy classes $\iota(\Theta_{M, \chi}^{\text{ss}})$ are conjugate to an element in $K$.
\end{corollary}
\begin{proof}
	The morphism
	\[
	H^0_c(G, M_\chi) \rightarrow H^0(G, M_\chi)
	\]
	is an isomorphism by Proposition \ref{PROP_GysinUnramifiedCriterion}. The generalized Riemann hypothesis implies that the eigenvalues of $\iota(\Theta_{M, \chi})$ have absolute value $1$. Therefore, the semisimplification is conjugate to an element of $K$. 
\end{proof}
\begin{definition}
	Let $\Delta \subseteq \cvar{G}$ be a closed subset. Suppose $M \in \Perv{\text{int}}{G}$ is strongly $\Delta$-unramified, arithmetically semisimple, and $\iota$-pure of weight zero. We define the conjugacy class
	\[
	\mathbf{\Theta}_{M, \chi} := \iota(\Theta_{M,\chi}^{\text{ss}})
	\]
	in the space of conjugacy classes of a maximal compact subgroup of $\mathbf{G}^{\text{ari}}_{M}\otimes_\iota\BC$.
\end{definition}
\subsection{Equidistribution} We state a general stratification theorem for Fourier coefficients on arbitrary groups. Our previous stratification theorem only yielded bounds for general exponential sums indexed by the characters of a group. If we wish to apply the Tannakian categories developed here, note that we require the first stratum to be closed for the equidistribution theorem to apply.
\begin{theorem}\label{THM_stratsfourier}
	Let $M \in \Perv{}{G}$ be  perverse of weight $\leq 0$. There exists a stratification by thin subsets
	\[
	\cvar{G} = \Delta_0 \supset \Delta_1 \supseteq \ldots\supseteq \Delta_{n - 1} \supseteq \Delta_n = \emptyset
	\]
	such that
	\begin{enumerate}
		\item The subset $\Delta_i \subseteq\cvar{G}$ has codimension $\geq i$. 
		\item The perverse sheaf $M$ is strongly $\Delta_1$-unramified.
		\item We have
		\[
		\sum_{x \in G(k_n)} \chi(x)t_M(x) \ll |k_n|^{(i - 1)/2}
		\]
	for all $i \geq 0$, $n \geq 1$, and arithmetic characters $\chi \in \widehat{G}(k_n)$ with $\chi \notin \Delta_i$.
	\end{enumerate}
\end{theorem}
\begin{proof}
	We can reduce to a product $G = S\times U$ by Proposition \ref{PROP_CharIsogenyClosed}, Proposition \ref{PROP_GysinUnramifiedCriterion}, Proposition \ref{PROP_ProjectionFormula1}, and the $t$-exactness of pushforward under a finite map. In this case, it follows from the vanishing Theorem \ref{THM_VanishingTHM} and stratification Theorem \ref{THM_stratificationbounds}. 
	\end{proof}
The Leray spectral sequence allows us to deduce the following, slightly more general stratification theorem.
\begin{theorem}\label{THM_stratsfourier2}
	Let $K \in \derCat{c}{G}{}$ be a complex with perverse amplitude $\leq 0$ and of weight $\leq 0$. There exists a stratification by thin subsets
	\[
	\cvar{G} = \Delta_0 \supset \Delta_1 \supseteq \ldots\supseteq \Delta_{n - 1} \supseteq \Delta_n = \emptyset.
	\]
	We have:
	\begin{enumerate}[]\item The stratum $\Delta_i \subseteq \cvar{G}$ has codimension $\geq i$.
		\item We have
		\[
		\sum_{x \in G(k_n)} \chi(x)t_K(x) \ll |k_n|^{(i - 1)/2}
		\]
		for all $i \geq 0$, $n \geq 1$, and arithmetic characters $\chi \in \widehat{G}(k_n)$ with $\chi \notin \Delta_i$.
	\end{enumerate}
\end{theorem}
\begin{proof}
	This stratification theorem follows immediately by applying Theorem \ref{THM_stratsfourier} to the perverse cohomology sheaves $\pervCoh{n}{K}$ and the \quotationMark{additivity} of the trace.
\end{proof}
	To state the equidistribution theorem in a convenient form, we make the following definition (see also \cite[Def.~4.1]{KowalskiTannaka}).
	\begin{definition}
		Let $K$ be a compact Lie group and $Y_n \subseteq K^\sharp$ a finite set of conjugacy classes for all $n \geq 1$. 
		\begin{enumerate}
			\item We say the conjugacy classes $Y_n$ \textit{equidistributes on average in $K$} as $n \rightarrow \infty$ in $K$ if
		\[
		\lim_{N \rightarrow \infty} \frac{1}{N}\sum_{n = 1}^N \frac{1}{|Y_n|} \sum_{\mathbf{\Theta} \in Y_n} f(\mathbf{\Theta}) = \int_K f(g)dg
		\]
		for all continuous class-functions $f\colon K \rightarrow \BC$. The integral is taken with respect to the Haar probability measure.
		\item  We say the set of conjugacy classes $Y_n$ equidistributes in $K$ as $n \rightarrow \infty$ if
		\[
		\lim_{N \rightarrow \infty} \frac{1}{|Y_n|} \sum_{\mathbf{\Theta} \in Y_n} f(\mathbf{\Theta}) = \int_K f(g)dg. 
		\]
		for all continuous class-functions $f\colon K \rightarrow \BC$. The integral is taken with respect to the Haar probability measure.
	\end{enumerate}		
	\end{definition}
	We also require the notion of character codimension (see \cite[Def.~1.22]{KowalskiTannaka}).
	\begin{definition}
		Let $\Delta\subseteq\cvar{G}$ be a subset and $d \geq 0$. We say $\Delta$ has \textit{character codimension $d$} if
		\[
		\frac{|\Delta\cap\widehat{G}(k_n)|}{|\widehat{G}(k_n)|} \ll |k|^{-nd}
		\]
		for all $n \geq 1$.
	\end{definition}
	\begin{remark}
		Proposition \ref{PROP_thinsubsetsestimate} shows that a thin subset $\Delta \subseteq\cvar{G}$ of codimension $d$ has character codimension $d$. In particular, if $\Delta$ is a proper thin subset it has character codimension $\geq 1$. 
	\end{remark}
	\begin{theorem}\label{THM_Equi1}
		Let $M \in \Perv{}{G}$ be $\iota$-pure of weight zero and $\Delta$-unramified for a closed subset $\Delta \subset \cvar{G}$ of character codimension $\geq 1$. Let $K$ be a maximal compact subgroup in the arithmetic Tannkian monodromy group of $M$. The conjugacy classes $\mathbf{\Theta}_{M, \chi}$ for $\chi \in \widehat{G}(k_n)$ with $\chi \notin \Delta$ equidistribute on average in $K$ as $n \rightarrow \infty$.
	\end{theorem}
	\begin{proof}
		Theorem \ref{THM_stratsfourier} allows us to apply the argument from \cite[Thm.~4.11]{KowalskiTannaka}.
	\end{proof}
	\begin{theorem}\label{THM_Equi2}
		Let $M \in \Perv{}{G}$ be $\iota$-pure of weight zero and $\Delta$-unramified for a for a closed subset $\Delta \subset \cvar{G}$ of character codimension $\geq 1$. Let $K$ be a maximal compact subgroup in the arithmetic Tannkian monodromy group of $M$. We assume the arithmetic and the geometric Tannakian mondoromy group of $M$ agree. The conjugacy classes $\mathbf{\Theta}_{M, \chi}$ for $\chi \in \widehat{G}(k_n)$ with $\chi \notin \Delta$ equidistribute in $K$ as $n \rightarrow \infty$.
	\end{theorem}
	\begin{proof}
		By Theorem \ref{THM_stratsfourier}, we can apply the argument from \cite[Thm.~4.15]{KowalskiTannaka}.
	\end{proof}

%% file: Appendix.tex
	\section{Appendix}
	\subsection{The $\ell$-adic character group} \label{SUBSEC_CharGroup}We prove exactness for the group of $\ell$-adic characters of a profinite abelian group. For this section, we denote with $p \in \BN$ any prime and with $\ell$ a fixed prime.
	
	\begin{definition}
		Let $\pi, \pi', \pi''$ be profinite commutative groups. We say a sequence of continuous morphisms
		\[
		0 \rightarrow \pi' \rightarrow \pi \rightarrow \pi'' \rightarrow 0
		\]
		is exact if and only if it is exact pointwise, i.e. as a sequence of commutative groups.
	\end{definition}
	Recall the following structural theorem, which follows from the classification of finitely generated abelian groups. 
	\begin{lemma}\label{LEM_abprofinexact}
		Let $\pi$ be a profinite commutative group. Then
		\[
		\pi = \prod_{p} \pi_p
		\]
		where $\pi_p$ is the maximal pro-$p$-quotient of $\pi$ for a prime $p \in \BN$. The functor
		\[
		\pi \mapsto \pi_p
		\]
		is an exact functor from the category of profinite commutative groups to the category of profinite commutative groups.
	\end{lemma}
	We introduce the group of $\ell$-adic characters for a topological group.
	\begin{definition}
		Let $\pi$ be a topological group. An \textit{$\ell$-adic character of $\pi$} is a continuous group morphism $\chi\colon \pi \rightarrow K_\lambda^*$ for a finite extension $K_\lambda/\BQ_\ell$. We denote by $\SC(\pi)$ the group of all $\ell$-adic characters.
	\end{definition}
	Our objective is to prove the following theorem.
	\begin{theorem}\label{THM_AppendixCharacterGroupExact}	Let $\pi, \pi'$, and $\pi''$ be commutative profinite groups. We assume the maximal pro-$\ell$ quotient of $\pi, \pi'$ and $\pi''$ is finitely generated. If we are given an exact sequence
		\[
		0 \rightarrow \pi' \rightarrow \pi \rightarrow \pi'' \rightarrow 0.
		\]
		then we have an exact sequence of commutative groups
		\[
		0 \rightarrow \mathscr{C}(\pi'') \rightarrow \mathscr{C}(\pi) \rightarrow \mathscr{C}(\pi') \rightarrow 0.
		\]
	\end{theorem}
	\begin{proof} We have
		\[
		\cvar{\pi} = \bigoplus_{p \in \BZ}\cvar{\pi_p}.
		\]
	This decomposition splits the sequence on profinite groups and character groups into a direct product and a direct sum of sequences on profinite groups and abstract groups by Lemma \ref{LEM_abprofinexact}. Hence we can assume $\pi$ is pro-$p'$ for a prime $p' \in \BN$.
	
	If $p' \neq \ell$, the claim follows form the $p$-divisibility of $\BQ_p/\BZ_p$. 
	
	Suppose $p' = \ell$. Let $M_\lambda$ be the pro-$\ell$ factor of the profinite group $\CO_\lambda^*$. Then we put
	\[
	M := \varinjlim M_\lambda.
	\]
	The abelian group $M$ is an $\ell$-divisible $\BZ_\ell$-module. Hence it is an injective $\BZ_\ell$-module. The sequence of character groups identifies with the sequence
	\[
	0 \rightarrow \text{Hom}_{\BZ_\ell}(\pi'', M) \rightarrow \text{Hom}_{\BZ_\ell}(\pi, M) \rightarrow \text{Hom}_{\BZ_\ell}(\pi', M) \rightarrow 0.
	\]
	This sequence is exact because $M$ is injective.
\end{proof}
\begin{remark}
	The maximal pro-$\ell$-quotient of $\pi_1^{\text{char}}(G)$ is finitely generated. We can assume $G$ connected by Corollary \ref{COR_ImageFundGroup}. Let $U \subseteq G$ be the maximal unipotent subgroup of $G$, then $G/U$ is semiabelian. By Propositon \ref{PROP_FundamentalGroupExact}, we have an exact sequence
	\[
	0 \rightarrow \pi_1^{\text{char}}(U) \rightarrow \pi_1^{\text{char}}(G) \rightarrow \pi_1^{\text{char}}(G/U) \rightarrow 0.
	\]
	The maximal pro-$\ell$ quotient is exact by Lemma \ref{LEM_abprofinexact} and the group $\pi_1^{\text{char}}(U)$ is pro-$p$ by Lemma \ref{LEM_charuniprop}, thus we can assume $G$ is semiabelian. In this case, it follows from Proposition \ref{PROP_fundcompsemiab} and the well-known comparison of the tame fundamental group with the dual of the first integral cohomology group. 
\end{remark}

	\subsection{Structure theorem for algebraic groups} We require the following structural theorem for algebraic groups.
	\begin{theorem}\label{THM_AppendixStructureTheorem}
		There exists an isogeny
		\[
		G \rightarrow A\times T \times U
		\]
		such that $U$ is a smooth unipotent group.
	\end{theorem}
	\begin{proof} Note that $G$ admits an isogeny to a smooth, algebraic group by \cite[Prop.~2.9.2]{BrionStructureTheoremsAlgebraicGroups}. Thus we may assume $G$ is smooth. By \cite[Cor.~5.5]{BrionStructureTheoremsAlgebraicGroups}, there exists an abelian variety $A \subseteq G$ and a connected, affine subgroup $L \subseteq G$ such that $G = L\cdot A$ and $L\cap A$ is finite. In particular, the map
		\[
		G \rightarrow G/L\times G/A
		\]
		is an isogeny. The group $G/L$ admits the isogeny $A \rightarrow G/L$, hence it is an abelian variety. The group $G/A$ admits the isogeny $L \rightarrow G/A$, hence the group $G/A$ is affine and connected. By \cite[Thm.~5.3.1~(2)]{BrionStructureTheoremsAlgebraicGroups}, it is the product of a torus with a unipotent group. We apply \cite[Prop.~2.9.2]{BrionStructureTheoremsAlgebraicGroups} to construct an isogeny with $U$ smooth.
	\end{proof}
	
	We prove a Lemma, which is required during the proof of the classification.
	\begin{lemma}\label{COR_AppendixMapToCircle}
		Let $S$ be a semiabelian variety of dimension $d = d_a + d_t$ over a finite field $k$. If $d_t > 0$, there exists a finite extension $k'/k$ and a surjective group morphism $S_{k'} \rightarrow \BG_{m}$ with connected kernel.
	\end{lemma}
	\begin{proof}
		By Theorem \ref{THM_AppendixStructureTheorem}, there exists an isogeny
		\[
		S \rightarrow T\times A.
		\]
		There exists a finite extension $k'/k$ such that the torus $T$ splits, i.e.
		\[
		T_{k'} \cong \BG_m^{d_t}.
		\]
		This means we can find a group morphism
		\[
		S_{k'} \rightarrow \BG_m.
		\]
		Let $S'$ be the kernel, then we have the isogeny
		\[
		S_{k'}/(S')^0 \rightarrow \BG_m.
		\]
		Thus
		\[
		S_{k'}/(S')^0  \cong \BG_m,
		\]
		hence the Corollary.
	\end{proof}
	\subsection{Canonical map}\label{SUBSEC_CanonicalMap}	We verify the commutativity of a certain diagram. This section is written for Noetherian schemes, because that is the natural generality for the commutativity of this diagram. 
	\begin{definition}
		Let $f\colon X \rightarrow Y$ be a separated morphism of finite type between Noetherian schemes and $K, L \in \derCat{c}{X}{\Qbarl}$ (see \cite[Ch.~10.1]{LeiFu}). We define the map $f^* K \otimes f^! L \rightarrow f^!(K\otimes L)$ as the adjoint of
		\[
		f_!(f^*K \otimes f^! L) \xrightarrow{\sim} K\otimes f_!f^* L \rightarrow K \otimes L 
		\]
		where the second map is the counit and the first arrow is the isomorphism coming from the projection formula. 
	\end{definition}
	\begin{remark} Let $i_y\colon \overline{y} \rightarrow Y$ be the inclusion of a geometric point supported at a point $y \in Y$ into $y$. This map is not of finite type. Therefore, the functor $i_y^!$ is strictly speaking not defined. We define it as follows: let $i\colon \overline{\{y\}} \rightarrow Y$ be the closed immersion of the reduced subscheme supported over the closure of $\{y\}$. Then we define $i_y^!K := i_y^*i^!K$. This also applies to the canonical map.
	\end{remark}
	\begin{remark}
		When $f$ is proper, there are natural transformations $f^*f_! \rightarrow 1$ and $1 \rightarrow f_!f^*$. Using these maps map, one can write the morphism defining the map as
		\[
		f_!(f^*K\otimes f^! L) \leftarrow f_!(f^*K\otimes f^*f_!f^! L) \leftarrow K\otimes f_!f^!L \rightarrow K\otimes L.
		\]
		The first two maps compose to give the projection formula, hence the composite is an isomorphism.
	\end{remark}
	We now state that a certain diagram is commuative
	\begin{lemma}\label{LEM_Diagram1}
		Consider a cartesian diagram of separated morphisms of finite type between Noetherian schemes
		\[
		\begin{tikzcd}
			X' \arrow[r, "f'"] \arrow[d, "g'"] & Y' \arrow[d, "g"] \\
			X \arrow[r, "f"]                   & Y                
		\end{tikzcd}
		\]
		Let $K \in \derCat{c}{X}{\Qbarl}$ and $L \in \derCat{c}{Y}{\Qbarl}$. We form the diagram
		\[
		\begin{tikzcd}
			g^*f_*K\otimes g^!L \arrow[d] \arrow[r]  & g^!(f_*K\otimes L) \arrow[dd]             \\
			(f'_*g'^*K)\otimes g^!L \arrow[d]        &                                           \\
			f'_*(g'^* K\otimes f'^*g^! L) \arrow[d]  & g^!f_* (K \otimes f^*L) \arrow[d, "\sim"] \\
			f'_*(g'^* K \otimes g'^! f^*L) \arrow[r] & f'_*g'^! (K\otimes f^* L)                
		\end{tikzcd}
		\]
		It is commutative.
	\end{lemma}
	\begin{proof}
		We can assume $g$ is proper. Consider the diagram:
		\[
		\begin{tikzcd}[column sep=tiny]
			g_!(g^*f_* K\otimes g^! L) \arrow[d]           & g_!g^*(f_*K\otimes g_!g^! L) \arrow[l] \arrow[d]             & f_*K\otimes g_!g^! L \arrow[l] \arrow[r] \arrow[d]      & f_*K\otimes L \arrow[d]        \\
			g_!(f'_*g'^* K \otimes g^!L) \arrow[d]         & g_!g^*f_*(K\otimes f^*g_!g^! L) \arrow[d]                    & f_*(K\otimes f^*g_!g^! L) \arrow[l] \arrow[r] \arrow[d] & f_*(K\otimes f^* L)            \\
			g_!f'_{*}(g^{'*}K\otimes f^{'*}g^!L) \arrow[d] & g_!f'_*g'^{*}(K\otimes g'_!f^{'*}g^{!}L) \arrow[d] \arrow[l] & f_*(K\otimes g'_!f^{'*}g^! L) \arrow[d]                 &                                \\
			f_*g'_!(g^{'*}K\otimes g^{'!}f^* L)            & f_*g'_!g^{'*}(K\otimes g'_!g^{'!}f^*L) \arrow[l]             & f_*(K\otimes g_!'g^{'!}f^* L) \arrow[l] \arrow[r]       & f_*(K\otimes f^* L). \arrow[uu, equal]
		\end{tikzcd}
		\]
		The arrows on the left are constructed by pulling $g^*$ into the tensor product and then applying functoriality. It is commutative: most of the squares commute by functoriality. The middle square commutes by functoriality as well, it can be factored. The diagram on the bottom right commutes because adjunction  unites commute with basechange. The diagram on the top left commutes because the projection formula commutes with base change.
		
		The functor $g_!$ is right adjoint to $g^!$, so we obtain a commutative diagram
		\[
		\begin{tikzcd}
			g^*f_*K\otimes g^! L \arrow[d] \arrow[r]  & g^!(f_*K\otimes L)    \arrow[ddd]            \\
			(f'_*g'^* K) \otimes g^! L \arrow[d]           &                                   \\
			f'_*(g'^* K\otimes f'^* g^!L) \arrow[d] &                                   \\
			g^!f'_*(g'^*K\otimes g'^!f^*L) \arrow[r]     & g^!(f_*(K\otimes f^*L))
		\end{tikzcd}
		\]
		For each map $g'_! M \rightarrow N$ between complexes $M, N \in \derCat{c}{X'}{\Qbarl}$, have a commutative diagram
		\[
		\begin{tikzcd}
			& g^!g_! f'_* M \arrow[r] \arrow[dd] & g^! f_* N\arrow[dd] \\
			f'_* M \arrow[ru] \arrow[rd] &                                    &                      \\
			& f'_*g'^!g'_! M \arrow[r]           & f'_*g'^! N         
		\end{tikzcd}
		\]
		We can put this diagram with suitable choices of $M$ and $N$ together with the one we obtained from the large diagram on the bottom to obtain the diagram from the Lemma.
	\end{proof}
	\begin{lemma}\label{LEM_Diagram2}
		Consider a cartesian diagram of morphisms of finite type between Noetherian schemes
		\[
		\begin{tikzcd}
			X' \arrow[r, "f'"] \arrow[d, "g'"] & Y' \arrow[d, "g"] \\
			X \arrow[r, "f"]                   & Y                
		\end{tikzcd}
		\]
		Let $u\colon X \rightarrow \overline{X}$ be an open immersion over $Y$ such that $\overline{X}$ is a proper $Y$-scheme. Let $K \in \derCat{c}{X}{\Qbarl}$ be a complex on $X$ and $L \in \derCat{c}{Y}{\Qbarl}$. We form the diagram
		\[
		\begin{tikzcd}
			g^*f_!K\otimes g^!L \arrow[d, "\sim"] \arrow[r] & g^!(f_!K\otimes L) \arrow[dd]                      \\
			(f'_!g'^*K)\otimes g^!L \arrow[d, "\sim"]       &                                                    \\
			f'_!(g'^* K\otimes f'^*g^! L) \arrow[d]         & g^!f_! (K \otimes f^*L) \arrow[d]                  \\
			f'_!(g'^* K \otimes g'^! f^*L) \arrow[r]        & \overline{f}'_*\overline{g}'^! u_!(K\otimes f^* L)
		\end{tikzcd}
		\]
		It is commutative.
	\end{lemma}
	\begin{proof}
		Apply Lemma \ref{LEM_Diagram1} to $u_! K$ and $L.$
	\end{proof}
	\begin{proof}[Proof of Prop.~\ref{PROP_GysinDiagram}]\label{PROOF_GysinDiagram}
		The natural forget supports map can be used to define a morphism from the diagram in Lemma \ref{LEM_Diagram1} to the diagram in Lemma \ref{LEM_Diagram2}. The base change in Lemma \ref{LEM_Diagram1} is an isomorphism for $i_y^!$ because we are assuming the morphism is smooth (see \cite[Prop.~8.2.4~(i)]{LeiFu})  
	\end{proof}
	\subsection{Perverse sheaves} We record a few lemmas on perverse sheaves, since they do not fit the text otherwise. The following lemma is based on \cite[Prop.~6.7.1]{KatzPerversityExpSums}.
	\begin{lemma}\label{LEM_PerverseLisseCrit}
		Let $X$ be an irreducible geometrically unibranch scheme of dimension $d$. Let $M \in \Perv{}{X}$ be a perverse sheaf on $X$. If the stalks of $M$ are concentrated in degree $-d$ and have the same dimension at all closed points, then $M$ is lisse.
	\end{lemma}
	\begin{proof}
		Let $U \subseteq X$ be a dense open subset where $M$ is lisse and denote by $Z := X - U$ the complement of $U$. We consider immersions $i\colon Z \rightarrow X$ and $j\colon U \rightarrow X$. We have a triangle
		\[
		i_*i^!M \rightarrow  M \rightarrow j_*j^*M \xrightarrow{+}.
		\] 
		The sheaf $i_*i^!M$ is concentrated in non-negative perverse degrees and is concentrated on $Z$. Hence
		\[
		H^{-d}(i_*i^! M) = 0.
		\]
		In particular, we obtain an injection
		\[
		H^{-d}(M) \rightarrow H^{-d}(j_*j^*M).
		\]
		The specialization maps on $j_*j^*M$ are injective because $X$ is geometrically unibranch. Hence the specialization maps on $H^{-d}(M)$ are injective. Note that all the stalks in $H^{-d}(M)$ have the same dimension, hence the specialization maps are isomorphisms. Thus $H^{-d}(M)$ is lisse and hence $M$ is lisse.
	\end{proof}
	We begin by proving the following Lemma:
	\begin{lemma}\label{LEM_PervRelativePerversity}
		Let $K \in \derCat{c}{X_{\overline{k}}}{\Qbarl}$ be a complex with perverse amplitude $[a, b]$. Let $X \rightarrow Y$ be a morphism to an irreducible scheme $Y$ of dimension $d$. There exists a dense open subset $U \subseteq Y$ such that $i_y^*(M)[-d]$ has perverse amplitude $[a, b]$ for all points $y \in U$. 
	\end{lemma}
	\begin{proof} We can assume $Y$ is affine. By Noether normalization, we can assume $Y = \BA^d$. The problem is Zariski local on $X$, hence we can assume $X \rightarrow Y$ is affine. There is a closed embedding $X \rightarrow \BA^N_Y$, hence we can assume $X = \BA^N_Y$. There exists a dense open subset $U \subseteq Y$ over which $K$ is  universally locally acyclic. Proposition \ref{PROP_GysinIsomorphismULA} implies that the canonical map
		\[
		i_y^*(M)\otimes i_y^!(\Qbarl) \rightarrow i_y^!(M)
		\]
		is an isomorphism for all $y \in U$. Let $y \in U$. Artin's vanishing theorem (and the existence of systems of parameters) implies that the cohomological amplitude of $i_y^*$ is $[-d, -d+c]$ and the cohomological amplitude of $i_y^!$ is $[-d + c, -d + 2c]$. Smoothness implies
		\[
		i_y^!(\Qbarl) = \Qbarl(-c)[-2c].
		\]
		The statement follows by comparing the cohomological amplitude of the left- and the right-hand side of the canonical map.  

	\end{proof}
	We note two consequences of this lemma.
	\begin{proposition}\label{PROP_PervRelativePerversity}
		Let $K \in \derCat{c}{X}{\Qbarl}$ have perverse amplitude $[a, b]$ and $X \rightarrow Y$ a morphism. Suppose $Y$ has pure dimension $d$. There is a stratification
		\[
		Y = Z_{0} \supseteq Z_1 \supseteq \ldots \supseteq Z_{d - 1}\supseteq Z_d = \emptyset
		\]
		by closed subsets with the following properties. 
		\begin{enumerate}
			\item For all $0 \leq i \leq d$ the subset $Z_{i}$ has codimension $\geq i$. 
			\item For all $0 \leq i \leq d - 1$ and all points $y \notin Z_{i + 1}$ the complex $i_y^*(M)$ has amplitude $[-d + a, -d + i + b]$. 
		\end{enumerate}
	\end{proposition}
	\begin{proof}
	We can assume $Y$ is affine.  By Noether normalization, we can assume $Y = \BA^d$.  By Lemma \ref{LEM_PervRelativePerversity}, there is a divisor $Z_0 \subseteq Y$ such that $i_x^*(M)$ has perverse amplitude $[-d + a, -d + b]$ for all $x \notin Z_0$. By Artin vanishing, we can pass to $Z_0$ and proceed by induction on dimension because $Z_0$ pulls back to a locally principal subset on $X$.
	\end{proof}
	When the complex is an irreducible perverse sheaf, we obtain an improved stratification. This is not required in the text, but perhaps it is useful in another context.
	\begin{corollary}\label{COR_PervRelPerversity}
		Let $M \in \Perv{}{X}$ be an irreducible perverse sheaf and $X \rightarrow Y$ a morphism. Suppose $Y$ has pure dimension $d$ and the support of $M$ maps densely onto $Y$. There is a stratification
		\[
		Y = Z_0 \supset Z_1 \supseteq \ldots \supseteq Z_{d - 2} \supseteq Z_{d - 1}\supseteq \emptyset
		\]
		by closed subsets with the following properties. 
		\begin{enumerate}
			\item For all $1 \leq i \leq d - 1$ the subset $Z_{i}$ has codimension $\geq i + 1$. 
			\item For all $0 \leq i \leq d - 2$ and all points $y \notin Z_{i + 1}$ the complex $i_y^*(M)$ has amplitude $[-d, -d + i]$. 
		\end{enumerate}
	\end{corollary}
	\begin{proof}
		We can assume $Y$ is affine. By Noether normalization, we can assume $Y = \BA^n$. Lemma \ref{LEM_PervRelativePerversity} proves that the function
		\[
		\varphi(x) := \text{max}\{n \in \BZ ~|~ \pervCoh{n}{i_x^*(M)} \neq 0\}
		\]
		is constructible and generically admits the value $-d$. We consider the set
		\[
		Z_0' := \{x \in X~|~ \varphi(x) > - d\}.
		\]
		The set $Z_0'$ is locally closed. Suppose there is a point $y \in Z_0'$ of codimension $1$ in $Z_0'$. There is an open subset $U \subseteq Z_0'$  with $y \in U$ such that 
		\[
		\pervCoh{-n + 1}{i_x^*M} \neq 0
		\]
		for all $x \in U$. Let $D \subset Y$ be the closure of $y$. Lemma \ref{LEM_PervRelativePerversity} implies 
		\[
		\pervCoh{0}{i_D^*(M)} \neq 0.
		\]
		This contradicts the irreducibility of $M$ by Artin vanishing. We can apply Proposition \ref{PROP_PervRelativePerversity} to $M$ restricted to a suitable enlargement of $Z_0'$. 
	\end{proof}